\documentclass[12pt]{amsart}
\usepackage{amssymb,amscd}

\input xy \xyoption{all}

\let\mc\mathcal
\usepackage{eucal}
\let\euc\mathcal
\let\mathcal\mc

\let\frak\mathfrak

\def\>{\relax\ifmmode\mskip.666667\thinmuskip\relax\else\kern.111111em\fi}
\def\<{\relax\ifmmode\mskip-.333333\thinmuskip\relax\else\kern-.0555556em\fi}
\def\:{\relax\ifmmode\mskip.333333\thinmuskip\relax\else\kern.0555556em\fi}
\def\?{\relax\ifmmode\mskip-.666667\thinmuskip\relax\else\kern-.111111em\fi}
\def\vsk#1>{\vskip#1\baselineskip} \let\bls\baselineskip
\def\vv#1>{\vadjust{\vsk#1>}\ignorespaces}
\def\vvn#1>{\vadjust{\nobreak\vsk#1>\nobreak}\ignorespaces}
\def\vvgood{\vadjust{\penalty-500}} \let\alb\allowbreak
\def\fratop{\genfrac{}{}{0pt}1}
\def\satop#1#2{\fratop{\scriptstyle#1}{\scriptstyle#2}}
\let\dsize\displaystyle \let\tsize\textstyle \let\ssize\scriptstyle
\let\sssize\scriptscriptstyle \def\tfrac{\tsize\frac}
 \let\vp\vphantom \let\hp\hphantom
\def\dfrac{\dsize\frac}

\let\Smallskip\smallskip
\def\smallskip{\par\Smallskip}
\let\Medskip\medskip
\def\medskip{\par\Medskip}
\let\Bigskip\bigskip
\def\bigskip{\par\Bigskip}

\let\Maketitle\maketitle
\def\maketitle{\Maketitle\thispagestyle{empty}\let\maketitle\empty}

\newtheorem{thm}{Theorem}[section]
\newtheorem{cor}[thm]{Corollary}
\newtheorem{lem}[thm]{Lemma}
\newtheorem{prop}[thm]{Proposition}
\newtheorem{conj}[thm]{Conjecture}

\numberwithin{equation}{section}

\theoremstyle{definition}
\newtheorem*{rem}{Remark}
\newtheorem*{example}{Example}

\newtheorem*{defn}{Definition}

\let\nc\newcommand

\let\al\alpha
\let\bt\beta
\let\dl\delta
\let\Dl\Delta

\let\gm\gamma
\let\Gm\Gamma
\let\ka\kappa
\let\la\lambda

\let\pho\phi
\let\phi\varphi
\let\si\sigma

\let\thi\vartheta

\let\om\omega
\let\Om\Omega

\let\der\partial
\let\Hat\widehat

\let\ox\otimes
\let\Tilde\widetilde
\let\bra\langle
\let\ket\rangle

\let\ge\geqslant
\let\geq\geqslant
\let\le\leqslant
\let\leq\leqslant

\let\on\operatorname
\let\bi\bibitem
\let\bs\boldsymbol
\let\xlto\xleftarrow
\let\xto\xrightarrow

\let\Empty\varnothing

\def\C{{\mathbb C}}
\def\P{{\mathbb P}}
\def\Z{{\mathbb Z}}

\def\De{{\euc D}}
\def\Ue{{\euc U}}

\def\Ac{{\mc A}}
\def\B{{\mc B}}

\def\F{{\mc F}}

\def\Kc{{\mc K}}
\def\O{{\mc O}}

\def\Rc{{\mc R}}

\def\V{{\mc V}}
\def\Zc{{\mc Z}}

\def\+#1{^{\{#1\}}}
\def\lsym#1{#1\alb\dots\relax#1\alb}
\def\lc{\lsym,}
\def\lox{\lsym\ox}

\def\End{\on{End}}

\def\Hom{\on{Hom}}
\def\id{{\on{id}}}
\def\Res{\on{Res}}

\def\tbigoplus{\mathop{\textstyle{\bigoplus}}\limits}
\def\tbigcup{\mathop{\textstyle{\bigcup}}\limits}

\def\Wr{\on{Wr}}

\def\Ipr{$\textrm I^{\:\prime}$}
\def\pti{\textit{pt}}
\def\hor{\textit{hor}}
\def\vert{\textit{vert}}

\def\aa{a,\:a}
\def\ab{a,\:b}
\def\ai{a,\:i}
\def\bj{b,\:j}
\def\ii{i,\:i}
\def\ij{i,\:j}
\def\ik{i,\:k}

\def\ji{j,\:i}
\def\jj{j,\:j}

\def\kj{k,\:j}

\def\kab{k\<,\:a,\:b}
\def\IIc{I\?,\:I}
\def\IJ{I\?,J}
\def\JI{J\?,\:I}
\def\ioi{i+1,\:i}
\def\pp{p,\:p}

\def\pci{p,\:i}
\def\pcj{p,\:j}
\def\poi{p+1,\:i}
\def\poj{p+1,\:j}
\def\xyc{x\<,\:y}
\def\ps{p,\:s}

\def\gln{\mathfrak{gl}_N}

\def\hg{\mathfrak h}

\def\G{{\frak G}}

\def\Yn{Y\<(\gln)}

\def\Uqn{U_q(\Tilde{\gln}\<)}
\def\Uen{\Ue(\Tilde{\gln}\<)}
\def\Upn{\Ue'(\Tilde{\gln}\<)}

\def\Uenc{\Ue_{\:c}(\Tilde{\gln}\<)}

\def\Ueh{\Ue(\:\hg\:)}

\def\beq{\begin{equation}}
\def\eeq{\end{equation}}
\def\be{\begin{equation*}}
\def\ee{\end{equation*}}

\nc{\bea}{\begin{eqnarray*}}
\nc{\eea}{\end{eqnarray*}}
\nc{\bean}{\begin{eqnarray}}
\nc{\eean}{\end{eqnarray}}
\nc{\Ref}[1]{{\rm(\ref{#1})}}

\def\fc{\check f}
\def\gc{\check g}
\def\Ic{\check I}
\def\Qc{\check Q}
\def\sch{\check s}

\def\Lh{\Hat L}
\def\Rh{\Hat R}
\def\sha{\hat s}

\def\Wh{\rlap{$\>\Hat{\?\phantom W\<}\:$}W}
\def\Yh{\Hat Y}
\def\sih{\hat\si}

\def\Dt{\rlap{$\,\Tilde{\!\phantom D}$}D}

\def\ft{\tilde f}
\def\gt{\tilde g}

\def\Kt{\Tilde K}
\def\Lt{\Tilde L}
\def\Mt{\rlap{$\,\Tilde{\!\phantom M}$}M}
\def\qt{\tilde q}
\def\sti{\tilde s}

\def\Vct{\Tilde\V}

\def\Wt{\rlap{$\>\Tilde{\?\phantom W\<}\:$}W}
\def\Yt{\Tilde Y}

\def\ett{\tilde\eta}
\def\sit{\tilde\si}
\def\thit{\Tilde\thi}

\def\Wbar{\rlap{$\>\overline{\?\phantom W\?}\>$}W}
\def\Wtb{\Wt_{\?\sssize\bullet}}

\def\N{\mathbb{Z}_{\geq0}}
\def\R{{\mathbb R}}

\def\XX{{\mc X}}
\def\I{{\mc I}}
\def\II{{\mc I}}

\def\bb{{\bs b}}
\def\GG{{\bs\Gm}}
\def\ppi{{\bs \pi}}

\def\TT{{\bs t}}
\def\ttt{{\bs t}}

\def\zz{{\bs z}}

\nc{\Il}{{\II_{\bs\la}}}
\nc{\bla}{{\bs\la}}
\nc{\Fla}{\F_\bla}
\nc{\tfl}{{T^*\!\Fla}}
\nc{\GL}{{GL_n(\C)}}
\nc{\GLC}{{GL_n(\C)\times\C^*}}

\def\mub{{\bs\mu}}
\def\Di{{\tfrac1D}}
\def\Dti{{\tfrac1{\>\Tilde{\!D}}}}

\def\DiV{\Di\V}
\def\DV{\Di\V^-}
\def\DL{\DV_\bla}
\def\DLl{\DVt_\bla}

\def\DVt{\Di\Vct^-}

\def\leqid{\leq_{\:\id}}

\def\CNn{(\C^N)^{\ox n}}
\def\CNnl{\CNn_{\:\bs\la}}

\def\Chh{\C[\:h^{\pm1}\:]}
\def\Czhp{\C[\zz,h\:]}
\def\Czhm{\C[\zz^{-1}\?,h\:]}
\def\Czps{\C[\zz]^{\>S_n}}

\def\Czpshr{\C[\zz]^{\>S_n}\!\ox\C(h)}
\def\Czpl{\C[\zz]^{\>S_\bla}}
\def\Czphl{\Czhp^{\>S_\bla}}

\def\Czms{\C[\zz^{-1}\:]^{\>S_n}}

\def\Czb{\C[\zz^{\pm1}\:]}
\def\Czh{\C[\zz^{\pm1}\?,h^{\pm1}\:]}
\def\CzhD{\C[\zz^{\pm1}\?,h^{\pm1}\?,D^{-1}\:]}

\def\CZH{\C[\zz^{\pm1}\?,h^{\pm1}\:]^{\>S_n}}
\def\CZS{\C[\zz^{\pm1}]^{\>S_n}}
\def\Czhr{\Czb\ox\C(h)}
\def\CZSr{\CZS\!\ox\C(h)}

\def\Czhh{\CNn\?\ox\C(\zz,h)}
\def\Czhhl{\CNnl\?\ox\C(\zz,h)}

\def\CGs{\C[\:\GG^{\pm1}\:]^{\>S_\bla}}

\def\CGl{\C[\:\GG\:]^{\>S_\bla}}
\def\CGlh{\CGl\?\ox\C[\:h^{\pm 1}\:]}
\def\CGlhr{\CGl\?\ox\C(h)}

\def\Czl{\C[\zz^{\pm1}]^{\>S_\bla}}
\def\Czhlr{\C[\zz^{\pm1}]^{\>S_\bla}\?\ox\C(h)}
\def\Czhl{\C[\zz^{\pm1}\?,h^{\pm1}\:]^{\>S_\bla}}

\def\Elzphr{\End\bigl(\CNnl\bigr)\ox\C[\zz]\ox\C(h)}

\def\KTX{K_T(\XX_n)}

\def\zzz{z_1\lc z_n}

\def\Imx{I^{\:\max}}
\def\Imn{I^{\:\min}}
\def\Imax{{\Imx}}
\def\Imin{{\Imn}}

\def\Iminp{{I^{\min,\prime i}}}
\def\Slmax{S^{\:\max}_\bla}
\def\Slmin{S^{\:\min}_\bla}

\def\ib{\bs i}
\def\jb{\bs j}
\def\kb{\bs k}
\def\iib{\ib,\:\ib}
\def\ijb{\ib,\:\jb}
\def\ikb{\ib,\:\kb}
\def\jib{\jb,\:\ib}

\def\kjb{\kb,\:\jb}

\def\zzzn{z_n\lc z_1}
\def\zzzsi{z_{\si(1)}\lc z_{\si(n)}}

\def\ip{\:i\>\prime}
\def\ipi{\>\prime\:i}

\def\iset{\{\:i\:\}}

\def\top{\on{top}}

\def\Bin{B^{\:\infty}}
\def\Bci{\B^{\:\infty}}

\def\asq{\ast^{q\<}}
\def\Dlq{\Dl(q_1\lc q_N)}
\def\elq{e^{\:q}}

\def\Kcqt{\Tilde\Kc^{\:Q}}
\def\muq{\mu^q}
\def\nuq{\nu^{\:q}}

\def\psq{\psi^{\:q}}

\def\thtq{\thit^{\>q}}
\def\Wrq{\Wr^{\:q}}

\def\all{\al_\bla}
\def\btl{\bt_\bla}

\def\Kcql{\Kc^{\:q}_\bla}
\def\mul{\mu_\bla}
\def\muql{\muq_\bla}
\def\nul{\nu_\bla}
\def\nuql{\nuq_\bla}
\def\nutql{\tilde\nu^{\:q}_\bla}
\def\phol{\pho_\bla}
\def\psil{\psi_\bla}
\def\psql{\psq_\bla}
\def\rhol{\rho_\bla}
\def\thil{\thi_\bla}
\def\thitl{\thit_\bla}
\def\thtql{\thtq_\bla}
\def\bcl{{\bb,\:c,\:\bla}}
\def\Kcqb{\Kc^{\:q}_\bcl}

\def\kkk{q_1\lc q_N}

\def\bk{b^{\:q}}
\def\Bk{B^{\:q}}
\def\Bck{\B^{\:q}}

\def\Wk{W^q}
\def\Whk{\Wh^q}

\def\ddk_#1{q_{#1}\:\frac\der{\der\:q_{#1}}}

\def\bul{\mathbin{\raise.2ex\hbox{$\sssize\bullet$}}}
\def\intt{\mathchoice
{\mathop{\raise.2ex\rlap{$\,\,\ssize\backslash$}{\intop}}\nolimits}
{\mathop{\raise.3ex\rlap{$\,\sssize\backslash$}{\intop}}\nolimits}
{\mathop{\raise.1ex\rlap{$\sssize\>\backslash$}{\intop}}\nolimits}
{\mathop{\rlap{$\sssize\:\backslash$}{\intop}}\nolimits}}

\def\GZ/{Gelfand\:-Zetlin}
\def\GZi/{Gelfand\:-\<Zetlin}
\def\KZ/{{\slshape KZ\/}}
\def\qKZ/{{\slshape qKZ\/}}
\def\XXX/{{\slshape XXX\/}}
\def\XXZ/{{\slshape XXZ\/}}

\def\Sym{\on{Sym}}
\def\St{{\on{Stab}}}
\def\Stab{\on{Stab}}

\def\Slope{\on{Slope}}
\def\Loc{\on{Loc}}

\def\kr{{\varkappa}}

\begin{document}

\hrule width0pt
\vsk->

\title[Weight functions as K-theoretic stable envelope maps]
{Trigonometric weight functions as K-theoretic stable\\[2pt]
envelope maps for the cotangent bundle of\\[2pt] a flag variety}

\author
[R.\,Rim\'anyi, V.\,Tarasov, A.\,Varchenko]
{ R.\,Rim\'anyi$\>^{\star}$,
V.\,Tarasov$\>^\circ$, A.\,Varchenko$\>^\diamond$}

\maketitle

\begin{center}
{\it $^{\star\,\diamond}\<$Department of Mathematics, University
of North Carolina at Chapel Hill\\ Chapel Hill, NC 27599-3250, USA\/}

\vsk.5>
{\it $\kern-.4em^\circ\<$Department of Mathematical Sciences,
Indiana University\,--\>Purdue University Indianapolis\kern-.4em\\
402 North Blackford St, Indianapolis, IN 46202-3216, USA\/}

\vsk.5>
{\it $^\circ\<$St.\,Petersburg Branch of Steklov Mathematical Institute\\
Fontanka 27, St.\,Petersburg, 191023, Russia\/}
\end{center}

{\let\thefootnote\relax
\footnotetext{\vsk-.8>\noindent
$^\star\<${\sl E\>-mail}:\enspace rimanyi@email.unc.edu\>,
supported in part by NSF grant DMS-1200685\\
$^\circ\<${\sl E\>-mail}: \enspace vt@math.iupui.edu\>, vt@pdmi.ras.ru\\
$^\diamond\<${\sl E\>-mail}: \enspace anv@email.unc.edu\>,
supported in part by NSF grant DMS-1362924}}

\begin{abstract}
We consider the cotangent bundle $\tfl$ of a $GL_n$ partial flag variety,
$\bla=(\la_1\lc\la_N)$, $|\bla|=\sum_i\la_i=n$, and the torus $T=(\C^\times)^{n+1}$ equivariant
K-theory algebra $K_T(\tfl)$. We introduce K-theoretic stable envelope maps
$\Stab_{\:\si}\?:\bigoplus_{|\bla|=n}K_T((\tfl)^T)\to\bigoplus_{|\bla|=n}K_T(T^*\Fla)$,
where $\si\in S_n$. Using these maps we define a quantum loop algebra action on
$\bigoplus_{|\bla|=n}K_T(T^*\Fla)$. We describe the associated Bethe algebra $\Bck(K_T(\tfl))$ by generators and relations
in terms of a discrete Wronski map. We prove that the limiting Bethe algebra $\Bci(K_T(\tfl))$, called the
\GZ/ algebra, coincides with the algebra of multiplication operators of the algebra $K_T(\tfl)$.
We conjecture that the Bethe algebra $\Bck(K_T(\tfl))$ coincides with the algebra of quantum multiplication on
$K_T(\tfl)$ introduced by Givental and Lee \cite{G, GL}.

The stable envelope maps are defined with the help of Newton polygons of
Laurent polynomials representing elements of $K_T(\tfl)$ and with the help of
the trigonometric weight functions introduced in \cite{TV1, TV3} to construct
\,$q$-hypergeometric solutions of trigonometric \qKZ/ equations.

The paper has five appendices. In particular, in
Appendix 5 we describe the Bethe algebra of the \XXZ/ model by generators and relations.
\end{abstract}

\setcounter{footnote}{0}
\renewcommand{\thefootnote}{\arabic{footnote}}

{\small \tableofcontents }

\section{Introduction}

In \cite{MO}, Maulik and Okounkov study the classical and quantum equivariant
cohomology of Nakajima quiver varieties for a quiver $Q$. They construct a Hopf algebra
$Y_Q$, the Yangian of $Q$, acting on the cohomology of these
varieties, and show that the associated Bethe algebra $\Bck$ acting on the cohomology of these
varieties coincides with the algebra of quantum multiplication.
The construction of the Yangian and the Yangian action is based on the notion
of the stable envelope maps introduced in \cite{MO}.
In this paper we construct the analog of the stable envelope maps for the
equivariant K-theory algebras of the cotangent bundles of the $GL_n$ partial
flag varieties.

\vsk.2>
Let $\tfl$ be the cotangent bundle
of a $GL_n$ partial flag variety $\Fla$,
$\bla=(\la_1\lc\la_N)$, $|\bla|=\sum_i\la_i=n$. We consider the torus $T=(\C^\times)^{n+1}$ equivariant
K-theory algebra $K_T(\tfl)$ and introduce K-theoretic stable envelope maps
$\Stab_{\:\si}\<:\bigoplus_{|\bla|=n}K_T((T^*\Fla)^T)\to
\bigoplus_{|\bla|=n}K_T(T^*\Fla)$,
where $(T^*\Fla)^T\subset T^*\Fla$ is the torus fixed point set and $\si$ is an element of the symmetric group $S_n$.
We describe the composition maps $\Stab^{-1}_{\:\si'}\circ\Stab_{\:\si}$
in terms of the standard $\gln$ trigonometric \,$R$-matrix.
Using these maps we define a quantum loop algebra action on
$\bigoplus_{|\bla|=n}K_T(T^*\Fla)$.
We describe the associated Bethe algebra $\Bck(K_T(\tfl))$ by generators and relations
in terms of a discrete Wronski map. We prove that the limiting Bethe algebra $\Bci(K_T(\tfl))$, called the
\GZ/ algebra, coincides with the algebra of multiplication operators of the algebra $K_T(\tfl)$.
We conjecture that the Bethe algebra $\Bck(K_T(\tfl))$ is isomorphic to the algebra of quantum multiplication on
$K_T(\tfl)$ introduced by Givental and Lee \cite{G, GL}. That conjecture is the K-theoretic analog of the main theorem in
\cite{MO} that describes the quantum multiplication.

\vsk.2>
In \cite{MO}, the stable envelope maps for equivariant cohomology of Nakajima varieties are defined axiomatically and then
the uniqueness and existence are proved. Our definition of K-theoretic stable envelope maps for the cotangent bundles of partial flag varieties
is also axiomatic and then we also prove the uniqueness and existence. The difference with axioms in \cite{MO}
is that we do not consider the supports of the stable envelope maps and we replace
the notion of the degree of a polynomial with the notion of the Newton polygon of a Laurent polynomial. Another difference with \cite{MO}
is that we prove the existence by giving an explicit formula for the stable envelope maps. The formula for
the stable envelope maps is given in terms of the trigonometric weight functions introduced in \cite{TV1, TV3}
to construct $q$-hypergeometric solutions of the trigonometric qKZ equations. The arguments of the weight functions in
\cite{TV1, TV3} are $h$, $\zzz$, $t_{i,j}$, where $h$ is the parameter of the quantum loop algebra,
$\zzz$ are positions of sites in the associated \XXZ/ model and $t_{i,j}$ are the integration variables in
the $q$-hypergeometric integrals. Another interpretation of the variables $t_{i,j}$ in \cite{TV1, TV3} is that they are variables in
the Bethe ansatz equations associated with the \XXZ/ model. In this paper, the arguments
$h$, $\zzz$ are interpreted as the equivariant parameters corresponding
to the torus $T$ and the arguments \,$t_{i,j}$ are interpreted as the Chern
roots of the associated bundles over $\Fla$. This correspondence between the
variable in the Bethe ansatz equations and the Chern roots is the indication of
a K-theoretic Landau-Ginzburg mirror correspondence.

\vsk.2>
The paper is organized as follows. In Section~\ref{Preliminaries from Geometry}, we introduce our geometric objects:
cotangent bundles of partial flag varieties, the torus action, the equivariant K-theory algebras.
In Section~\ref{sec:axiomatic}, we formulate axioms defining certain classes $\ka_{\si,I}\in K_T(\tfl)$ and formulate
Theorem \ref{thm:axiomatic} that such classes exist and are unique. The classes $\ka_{\si,I}$ are building blocks for
the K-theoretic stable envelope maps and are the main novelty objects of our paper. Theorem \ref{thm:axiomatic} is our first main result. In Section~\ref{sec:uni}, we prove
the uniqueness of the classes $\ka_{\si,I}$. In Section~\ref{sec:weight functions}, we introduce the trigonometric weight functions
and in Section~\ref{sec comb} describe useful combinatorial presentation of the weight functions as sums of 'elementary' rational functions
assigned to certain 'filled tables'. In Section~\ref{Properties of weight functions}, we describe properties of the weight functions and
prove the existence of the classes $\ka_{\si,I}$.
In Section~\ref{sec:stab}, using the classes $\ka_{\si,I}$,
we define the stable envelope maps and describe the compositions $\Stab^{-1}_{\si'}\circ\Stab_\si$ in terms of the standard
trigonometric \,$R$-matrix. In Section~\ref{invstab}, we describe the inverse map to the stable envelope map $\Stab_{\id}$.

\vsk.2>
In Section~\ref{alg sec}, we consider the space $\CNn\?\ox\C(\zzz,h)$
with an $S_n$-action. We introduce the important subspace
\,$\DV\!\subset\CNn\?\ox\C(\zzz,h)$ invariant with
respect to the $S_n$-action. In Section~\ref{Uq}, we define the quantum loop
algebra and its commutative Bethe subalgebra $\Bck$ depending on parameters
$q=(q_1\lc q_N)$. In Section~\ref{Yaction}, we describe a quantum loop
algebra action on \,$\DV\!\ox\C(h)$.

\vsk.2>
In Section~\ref{sec Equi}, we describe a quantum loop algebra action on
$\bigoplus_{|\bla|=n}K_T(\tfl)\ox\C(h)$. This is done through the isomorphism
$\nu:\bigoplus_{|\bla|=n}K_T(\tfl)\ox\C(h)\to
\DV\!\ox\C[z_1^{\pm1}\lc z_n^{\pm1}\:]\ox\C(h)$ defined
with the help of the map $\Stab^{-1}_{\>\id}$. We describe the close relations
between our quantum loop algebra action on $
\oplus_{|\bla|=n}K_T(T^*\Fla)\ox\C(h)$ and the quantum loop algebra action
studied by Ginzburg and Vasserot in \cite{GV,Vas1,Vas2}. In Theorem
\ref{regrep}, we identify the action on $K_T(\tfl)$ of the \GZ/ algebra $\Bci$
and the action on $K_T(\tfl)$ of its own elements by multiplication. This is
our second main result.

\vsk.2>
In Section~\ref{sec:Wr}, we introduce the discrete Wronski map $\Wrq_\bla$
and define the associated algebra $\Kcql$ by generators and relations.
We describe a construction that identifies the Bethe algebra \,$\Bck$ action
on $K_T(\tfl)$ and the regular representation of the algebra $\Kcql$.
This statement is formulated in Corollary~\ref{regKqK}.

\vsk.2>
In Section~\ref{sec:*}, we introduce a new commutative associative
multiplication \,$\asq$ \:on \,$K_T(\tfl)$\>, depending on the parameters
\,$\kkk$. In Section~\ref{sec conjectures}, we formulate Conjecture~\ref{conj}
that the new multiplication \,$\asq$ on \,$K_T(\tfl)$ \,coincides with
the quantum multiplication introduced by Givental and Lee in \cite{G, GL}.
By taking the $h\to 0$ limit of this conjectural statement, we formulate in
Section~\ref{Limit h to 0} a conjectural description of the equivariant quantum
K-theory algebra of the partial flag variety $\Fla$ by generators and
relations.

\vsk.2>
In Appendix 1, we show how the weight functions specialize to Grothendieck polynomials introduced by
Lascoux and Schutzenberger in \cite{LS}. In Appendix 2, we give an interpolation definition of the K-theory classes of Schubert
varieties in the equivariant K-theory algebra $K_{(\C^\times)^n}(\Fla)$ of a partial flag variety $\Fla$.
This definition is the $h\to 0$ limit of the definition of the classes $\ka_{\si,I}\in K_T(\tfl)$. In Appendix 3,
we describe presentations and structure constants of the K-theory algebras
associated with the projective line. Some of them follows from the conjecture in Section~\ref{sec conjectures} and others are
known. In Appendix 4, we give formulae for multiplication of the classes $\ka_{\id,I}$. These formulae can be viewed as
a beginning of Schubert calculus on $K_T(\tfl)$.

\vsk.2>
In Appendix 5 we describe the Bethe algebra of the \XXZ/ model by generators and relations. In particular we show that
the Bethe algebra of the \XXZ/ model on $(\C^N)^{\otimes n}$ is a maximal commutative subalgebra of \,$\End\bigl(\CNn\bigr)$\>.

\section{Preliminaries from Geometry}
\label{Preliminaries from Geometry}

\subsection{Partial flag varieties}
\label{sec Partial flag varieties}

Fix natural numbers $N, n$. Let \,$\bla\in\Z^N_{\geq 0}$, \,$|\bla|=\la_1\lsym+\la_N =n$.
Consider the partial flag variety
\;$\Fla$ parametrizing chains of subspaces
\be
0\,=\,F_0\subset F_1\lsym\subset F_N =\,\C^n
\ee
with \;$\dim F_i/F_{i-1}=\la_i$, \;$i=1\lc N$.
Denote by \,$\tfl$ the cotangent bundle of \;$\Fla$, and let $\pi: \tfl \to \Fla$ be the projection of the bundle.
Denote
\be
\XX_n\:=\coprod_{|\bla|=n}\<\tfl\,.
\ee
\begin{example}
If $n=1$, then $\bla=(0\lc 0,1_i,0\lc 0)$, $\tfl$ is a point and $\XX_1$
is the union of $N$ points.

\vsk.2>
If $n=2$ then $\bla= (0\lc 0,1_i,0\lc 0,1_j,0\lc 0)$ or $\bla= (0\lc 0,2_i,0\lc 0)$.
In the first case $\tfl$ is the cotangent bundle of projective line, in the second case $\tfl$ is a point.
Thus $\XX_2$ is the union of $N$ points and $N(N-1)/2$ copies of the cotangent bundle of projective line.
\end{example}

Let $I=(I_1\lc I_N)$ be a partition of $\{1\lc n\}$ into disjoint subsets
$I_1\lc I_N$. Denote $\Il$ the set of all partitions $I$ with
$|I_j|=\la_j$, \;$j=1\lc N$.

Let $\epsilon_1\lc\epsilon_n$ be the standard basis of $\C^n$.
For any $I\?\in\Il$, let $x_I\in\Fla$ be the point corresponding to the
coordinate flag $F_1\lsym\subset F_N$, where $F_i$ \,is the span of the
standard basis vectors \;$\epsilon_j\in\C^n$ with \,$j\in I_1\lsym\cup I_i$.
We embed $\Fla$ in $\tfl$ as the zero section and consider the points
$x_I$ as points of $\tfl$.

\subsection{Schubert cells, conormal bundles}
\label{sec:flagvar}
For any $\si\in S_n$, we consider the
coordinate flag in $\C^n$,
\be
V^{\si}\ : \ 0\,=\,V_0 \subset V_1\lsym\subset V_n=\,\C^n
\ee
where $V_i$ is the span of $\epsilon_{\si(1)}\lc\epsilon_{\si(i)}$. For $I\?\in\Il$ we define the {\it Schubert cell}
\be
\Om_{\si,\:I}=\{ F\in \Fla\ |
\dim (F_p\cap V^\si_q) = \#\{ i \in I_1\lsym\cup I_p \ |\ \si^{-1}(i)\leq q\} \ \forall p\leq N, \forall q\leq n \}.
\ee
The Schubert cell $\Om_{\si,\:I}$ is an affine space of dimension
\be
\ell_{\si,\:I}=\#\{(i,j)\in \{1\lc n\}^2\ |\ \si(i)\in I_a,\;\si(j)\in I_b,
\;a<b, \;i>j\}.
\ee
For a fixed $\si$, the flag manifold is the disjoint union of the cells $\Om_{\si,\:I}$. We have $x_I\in \Om_{\si,\:I}$,
see e.g. \cite[Sect.2.2]{FP}.

For $\si\in S_n$, we define the {\it geometric} partial ordering on the set
$\Il$. For $I,J\in\Il$, we say that $J\leq_g I$ if $x_J$ lies in the closure
of $\Om_{\si,\:I}$.

We also define the {\it combinatorial} partial ordering.
For $I,J\in\Il$, let
\be
\si^{-1}\bigl(\tbigcup_{\ell=1}^k I_\ell\bigr)\,=\,
\{a_1^k\lsym<a_{\la^{(k)}}^k\},
\qquad
\si^{-1}\bigl(\tbigcup_{\ell=1}^k J_\ell\bigr)\,=\,
\{b_1^k\lsym<b_{\la^{(k)}}^k\}
\ee
for $k=1\lc N-1$, where \,$\la^{(k)}\<=\la_1\lsym+\la_k$\,.
We say that $J\leq_c I$ if $b_i^k\leq a_i^k$ for $k=1\lc N-1$,
$i=1\lc\la^{(k)}$.

\begin{lem}
The geometric and combinatorial partial orderings are the same.
\end{lem}

\begin{proof}
This is the so-called ``Tableau Criterion'' for the Bruhat (i.e.~geometric)
order, see e.g.~\cite[Theorem 2.6.3]{BB}\:.
\end{proof}

In what follows we will denote both partial orderings by $\leq_\si$.

\vsk.2>
The Schubert cell $\Om_{\si,\:I}$ is a smooth submanifold of $\Fla$,
hence we can consider its conormal space
\vvn-.3>
$$
C\Om_{\si,\:I}=\{\al\in \pi^{-1}(\Om_{\si,\:I}) \ | \
\al(T_{\pi(\al)}\Om_{\si,\:I})=0\} \subset\tfl\,.
\vv.2>
$$
The conormal space $C\Om_{\si,\:I}$ is the total space of a vector
subbundle of $\tfl$ over $\Om_{\si,\:I}$. The rank of this subbundle
is $\dim \Fla - \dim \Om_{\si,\:I}$. Hence, as a manifold
$C\Om_{\si,\:I}$ is an affine cell of dimension $\dim \Fla$.
In particular, the dimension is independent of $\si, I$.
Define
\vvn.3>
\beq
\label{SLOPE}
\Slope_{\si,\:I}\>=\,\tbigcup_{J \leq_\si I} C\Om_{\si,\:J}.
\eeq

\subsection{Equivariant K-theory}
\label{sec:equivK}
The diagonal action of the torus \,$(\C^\times)^n$ on $\C^n$ induces an action
on $\Fla$, and hence on the cotangent bundle $\tfl$\>.

\begin{rem} One may use this action to give an equivalent (``unstable submanifold") definition of
the spaces $C\Om_{\si,\:I}$ of the last section.
Namely
$x\in C\Om_{\si,\:I}$ if and only if
\vvn.2>
\be
\lim_{z\to 0} (z^{-\si(1)}, z^{-\si(2)}\lc z^{-\si(n)})\cdot x = x_I,
\vv-.2>
\ee
cf.~\cite[Section 3.2.2]{MO}.
\end{rem}

We extend this \,$(\C^\times)^n\<$-action to the action of
\,$T=(\C^\times)^n\times\C^\times$ so that the extra \,$\C^\times$ acts
on the fibers of \,$\tfl\to\Fla$ by multiplication.

\vsk.2>
We consider the equivariant K-theory algebras $K_T(\tfl)$ and
\vvn.3>
\beq
\label{KTX}
\KTX\,=\,\tbigoplus_{|\bla|=n} K_T(\tfl).
\eeq
Our general reference for equivariant K-theory is \cite[Ch.5]{ChG}.

\vsk.2>
Denote $S_\bla=S_{\la_1}\!\lsym\times S_{\la_N}$ the product of symmetric groups.
Consider variables $\Gm_i=\{\gm_{i,1}\lc\gm_{i,\la_i}\}$, $i=1\lc N$.
Let \;$\GG=(\Gm_1\:\lsym;\Gm_N)$. The group $S_\bla$ acts on the set
$\GG$ by permuting the variables with the same first index.
Let $\C[\GG^{\pm1}]$ be the algebra of Laurent polynomials in variables
$\gm_{i,j}$ and $\CGs$ the subalgebra of
invariants with respect to the $S_\bla$-action.

\vsk.2>
Consider variables $\zz=\{\zzz\}$ and $h$. The group $S_n$ acts
on the set $\zz$ by permutations.
Let $\Czh$ be the algebra of Laurent polynomials in
variables $\zz,h$ and $\Czh^{\>S_n}$ the subalgebra of
invariants with respect to the $S_n$-action. We have
\vvn.2>
\beq
\label{Hrel}
K_T(\tfl)\,=\,\CGs\?\ox\Czh\;\big/\bigl\bra\:f(\GG)=f(\zz)\quad
\on{for\ any} \;f\in\CZS\bigr\ket\,.
\eeq
Here $\gm_{i,j}$ correspond to (virtual) line bundles also denoted
by $\gm_{i,j}$ with
\be
\tbigoplus_{j=1}^{\la_i}\,\gm_{i,j}\>=\>F_i/F_{i-1}
\ee
while $z_a$ and $h$ correspond to the factors of
$T=(\C^\times)^n\times\C^\times$.

\vsk.2>
The algebra \,$K_T(\tfl)$ \>is a module over \,$K_T(\pti\>;\C)=\Czh$.

\begin{example}
If $n=1$, then
\vvn-.6>
\be
K_T(\XX_1)\,=\,\tbigoplus_{i=1}^N\:K_T(T^*\!\F_{(0\lc 0,1_i,0\lc 0)})
\ee
is naturally isomorphic to \,$\C^N\ox\C[z_1^{\pm1},h^{\pm1}]$ \,with
the basis \,$v_i=(0\lc 0,1_i,0\lc 0)$\>, \,$i=1\lc N$.
\end{example}

\subsection{Fixed point sets}
\label{fixed}

\vsk.2>
The set \,$(\tfl)^{T}\!$ of fixed points of the torus $T$ action is
\,$(x_I)_{I\in\Il}$\>. We have
\vvn-.3>
\be
(\XX_n)^T = \XX_1\lsym\times\XX_1.
\vv.2>
\ee
The algebra $K_T\bigl((\XX_n)^T\bigr)$ is naturally isomorphic
to \,$\CNn\ox\Czh$\>. This isomorphism sends the identity element
\;$1_I\in K_T(x_I)$ \>to the vector
\vvn.1>
\beq
\label{v_I}
v_I\>=\,v_{i_1}\lox v_{i_n}\,,
\vv.2>
\eeq
where \,$i_j=k$ \,if \,$j\in I_k$\>. We denote by \,$\CNnl$ \:the span of
\,$\{\:v_I\;|\;I\in\Il\}$.

\subsection{Equivariant localization}
\label{sec:loc}

Consider the {\it equivariant localization} map
\vvn.2>
\beq
\label{eqn:loc}
\Loc\,:\,K_T(\tfl)\,\to\,K_T((\tfl)^T)\,=\,\tbigoplus_{I\in\Il} K_T(x_I)
\eeq
whose components are the restrictions to the fixed points $x_I$.
Namely, the \,$I$-component \;$\Loc_I$ \,of this map is the substitution
\beq
\label{LocI}
\{\gm_{k,1}\lc\<\gm_{k,\la_k}\}\mapsto\{z_a\,|\;a\in I_k\}
\qquad \text{for all}\;\,\,k=1\lc N\:.
\eeq
Equivariant localization theory (see e.g. \cite[Ch.5]{ChG},
\cite[Appendix]{RoKu}) asserts that $\Loc$ is an injection of algebras.
Moreover, an element of the right-hand side is in the image of \;$\Loc$
\,if the difference of the $I$-th and $s_{i,j}(I)$-th components is divisible
by $1-z_i/z_j$ in $\Czh$ for all $I\?\in\Il$ and $i,j\in \{1\lc n\}$.
Here $s_{i,j}(I)$ is the partition obtained from $I$ by switching the numbers
$i$ and $j$.

\vsk.2>
Let $\C(\zz,h)$ be the algebra of rational functions in $\zzz,h$.
The map
\beq\label{eqn:loc2}
\Loc\ :\ K_T(\Fla)\ox\C(\zz,h) \xrightarrow{\ \ \cong\ \ }
\oplus_{I\in\Il} K_T(x_I)\ox\C(\zz,h)
\eeq
is an isomorphism.

\subsection{K-theory fundamental class of \;$\Om_{\si,\:I}$ \>at \,$x_I$}
\label{sec:edefs}
Define the following classes
\vvn.3>
\beq
\label{ehor}
e_{\si,\:I\<,+}^{\hor}=\,\prod_{k<l}\>\prod_{\si(a)\in I_k}\<
\prod_{\satop{\si(b)\in I_l}{b<a}}\!(1-z_{\si(a)}/z_{\si(b)})\,,
\kern1.2em
e_{\si,\:I\<,-}^{\hor}=\,\prod_{k<l}\>\prod_{\si(a)\in I_k}\<
\prod_{\satop{\si(b)\in I_l}{b>a}}\!(1-z_{\si(a)}/z_{\si(b)})\,,\kern-.6em
\eeq
\beq
\label{evert}
e_{\si,\:I\<,+}^{\vert}=\,\prod_{k<l}\,\prod_{\si(a)\in I_k}\<
\prod_{\satop{\si(b)\in I_l}{b>a}}\!(1-hz_{\si(b)}/z_{\si(a)})\,,
\kern1.2em
e_{\si,\:I\<,-}^{\vert}=\,\prod_{k<l}\,\prod_{\si(a)\in I_k}\<
\prod_{\satop{\si(b)\in I_l}{b<a}}\!(1-hz_{\si(b)}/z_{\si(a)})\,,\kern-.6em
\eeq
in \,$K_T(x_I)=\Czh$\,. We also set
\,$e_{\si,\:I}=\:e_{\si,\:I\<,-}^{\hor}\>e_{\si,\:I\<,-}^{\vert}$\,.

\vsk.2>
Recall that if \,$\C^\times\!$ acts on a line \,$\C$ \>by
\,$\al\cdot x=\al^rx$,
then the \,$\C^\times\<$-equivariant Euler class of the line bundle
$\C \to \{0\}$ is $e(\C\to \{0\})=1-z^r \in K_{\C^\times}($point$)=\C[z^{\pm1}]$.
Thus standard knowledge on the tangent bundle of flag manifolds imply that
\vvn.2>
\beq\label{eqn:einter1}
e(T\Om_{\si,\:I}|_{x_I})=\,e_{\si,\:I\<,+}^{\hor}\,, \qquad
e(\nu(\Om_{\si,\:I}\subset\Fla)|_{x_I})=\,e_{\si,\:I\<,-}^{\hor}\,,
\vv.2>
\eeq
where \,$\nu(A\subset B)$ means the normal bundle of a submanifold $A$
in the ambient manifold $B$, and $\xi|_x$ means the restriction of
the bundle $\xi$ over the point $x$ in the base space.
Therefore we also have
\vvn-.2>
\be
e(C\Om_{\si,\:I}|_{x_I})\>=\>e_{\si,\:I\<,+}^{\vert}\,,\qquad
e((\pi^{-1}(\Om_{\si,\:I})-C\Om_{\si,\:I})|_{x_I})\>=\>e_{\si,\:I\<,-}^{\vert}\,,
\vv.2>
\ee
where $C\Om_{\si,\:I}$ and $\pi^{-1}(\Om_{\si,\:I})$ are
considered bundles over $\Om_{\si,\:I}$.
Now consider \,$C\Om_{\si,\:I}$ \,as a(n open) submanifold of \,$\tfl$\>.
Then we obtain
\vvn.2>
\beq\label{eqn:enu}
e(\nu( C\Om_{\si,\:I} \subset\tfl)|_{x_I})\,=
\,e^{\hor}_{\si,I,-}\>e^{\vert}_{\si,I,-}\>=\,e_{\si,\:I}\,.
\vv.2>
\eeq

\section{Axiomatic definition of the $\ka_{\si,\:I}$ classes}
\label{sec:axiomatic}

\subsection{Main result}
\label{sec main result}

In this section we phrase a theorem that axiomatically defines
some special classes in $K_T(\Fla)$.

\vsk.2>
Define the {\it polarization} of \,$\si\in S_n$\,, \,$I\?\in\Il$\,, to be
\vvn.2>
\beq
\label{PsiI}
P_{\si,\:I}\,=\,\prod_{k<l}\,\prod_{\si(a)\in I_k}\<
\prod_{\satop{\si(b)\in I_l}{b>a}}\!\bigl(-\:z_{\si(b)}/z_{\si(a)}\bigr)\,.
\eeq
Observe that $P_{\si,\:I}$ is an invertible element of $K_T(x_I)$.
It is the inverse of the ``top'' term of $e_{\si,\:I\<,-}^{\hor}$. In particular,
the number of $-z_{\si(b)}/z_{\si(a)}$ factors is the codimension of
$\Om_{\si,\:I}$ in $\Fla$, see \Ref{eqn:einter1}\:. The quantity
\vvn.2>
\beq
\label{Pe}
P_{\si,\:I}\;e_{\si,\:I}\,=\,\prod_{k<l}\,\prod_{\si(a)\in I_k}\!
\biggl(\>\prod_{\satop{\si(b)\in I_l}{b<a}}\!(1-hz_{\si(b)}/z_{\si(a)})\!
\prod_{\satop{\si(b)\in I_l}{b>a}}\!(1-z_{\si(b)}/z_{\si(a)})\biggr)
\vv-.1>
\eeq
will play a role below.

\vsk.2>
For a Laurent polynomial $f\in \C[\zz^{\pm1}]$ let $N(f)\subset\R^n$
be the Newton polygon of $f$, that is the convex hull of the points
$m\in \Z^n\subset\R^n$ such that the coefficient of $\prod z_a^{m_a}$ in
$f$ is not 0. For $I\?\in\Il$, consider the linear map $\phi_I:\R^n\to \R$ defined by
$$\phi_I(\epsilon_a)=k \ \text{if} \ a\in I_k.$$
We will study the convex set (closed interval) $\phi_I(N(f)) \subset\R$ for
certain $f$'s. For example $\phi_{(\{1\},\{2\})}(N(1-z_2/z_1))=[0,1]$. Observe that
\beq
\label{eqn:phiEeI}
\phi_I(N(P_{\si,\:I}\,e_{\si,\:I}))\,=\,\Bigl[\:0\>,
\sum_{1\leq k<l\leq N}\!\la_k\:\la_l \>(l-k)\>\Bigr]
\eeq
is independent of $\si$ and $I$; it only depends on $\bla$.
Define an element $f \in K_T(x_I)$ to be {\em $I$-small} if
\beq
\label{eqn:Ismall}
\phi_I(N(f)) \subset
\Bigl[\:0\>,\sum_{1\leq k<l\leq N}\!\la_k\:\la_l \>(l-k)-1\>\Bigr]\,.
\eeq
Note that in our Newton polygon considerations $h\in K_T(x_I)$ is considered
a constant, not a variable.

\begin{thm}
\label{thm:axiomatic}
Let \,$n,N,\bla$ \,be as above, and \,$\si\in S_n$\>.
Then for any \,$I\?\in\Il$\>, there exists a unique element
\,$\ka_{\si,\:I}\in K_T(\tfl)$ satisfying the following conditions:
\renewcommand{\theenumi}{\Roman{enumi}}
\begin{enumerate}
\item{} $\,\ka_{\si,\:I}|_{x_J}$ is divisible by \,$e^{\vert}_{\si,\:J\<,-}$
for all \,$J$;
\item{} $\,\ka_{\si,\:I}|_{x_I}= P_{\si,\:I}\,\:e_{\si,\:I}$\>;
\item{} \label{case:small}
$\,\ka_{\si,\:I}|_{x_J}$ is \,$J$-small if \,$J\ne I$.
\end{enumerate}
\end{thm}

\begin{rem} \rm
From the proof of Theorem \ref{thm:axiomatic} it will turn out that
the condition
\begin{enumerate}
\item[(0)] $\,\ka_{\si,\:I}|_{x_J}\?=\:0$ if \,$J\not\leq_\si I$
\end{enumerate}
also holds.
Condition (0) together with property (I) is a ``localized'' version of
the statement
\begin{enumerate}
\item[(\Ipr)] $\,\ka_{\si,\:I}$ is supported on \;$\Slope_{\si,\:I}$\>.
\end{enumerate}
Indeed, if a class is supported on $\Slope_{\si,\:I}$ then (0) and (I) follow,
see Section~\ref{sec:edefs}. Condition (\Ipr) appeared in \cite{MO} and
\cite{RTV}.
\end{rem}

First we prove uniqueness in Section~\ref{sec:uni} following the arguments
of \cite{MO} in which we replace the property of smallness of the degrees of
restrictions $\ka_{\si,\:I}|_{x_J}$ by the smallness of their Newton polygons
$N(\ka_{\si,\:I}|_{x_J})$.

One novelty of our treatment is that our axioms (I)-(III) are all {\em local}
properties of $\ka_{\si,I}$ (unlike (\Ipr)). That is, they are properties of
fix point restrictions. Another key novelty of our treatment is
Section~\ref{sec: existence} where we will show existence by giving a {\em
formula} for $\ka_{\si,\:I}$. Moreover, this formula for $\ka_{\si,\:I}$ is
a version of the trigonometric weight functions that had appeared in
\cite{TV1}\,--\,\cite{TV3} in hypergeometric solutions of \qKZ/ equations.

\subsection{Proof of uniqueness in Theorem \ref{thm:axiomatic}}
\label{sec:uni}

Suppose $\ka_{\si,\:I}$ and $\ka'_{\si,\:I}$ both satisfy the conditions.
Let $\om$ be their difference.

\vsk.2>
Refine the partial order\ $\leq_\si$
to a total order $\preceq$ on $\Il$.

Assume $\om\ne0$. Since the $\Loc$ map of Section~\ref{sec:loc} is injective
there is at least one $J\in \Il$ such that the restriction $\om|_{x_J}\ne0$.
Let $J$ be the $\preceq$-largest such element of $\Il$.

\vsk.2>
We claim that $\om|_{x_J}$ is
\begin{itemize}
\item{} $J$-small,
\item{} divisible by \,$e^{\vert}_{\si,\:J\<,-}$\,,
\item{} divisible by \,$e^{\hor}_{\si,\:J\<,-}$\,.
\end{itemize}
The first two properties follow explicitly from conditions (I), (II), (III).

To prove the third property choose a pair $a<b$ such that
$\si(a)\in J_k, \si(b)\in J_l$ and $k<l$. Let
$U=s_{\si(a),\si(b)}(J)\in \Il$. From
the definition of $\leq_\si$ it follows that $J<_\si U$ and hence $J
\prec U$. The choice of $J$ implies that $\om|_{x_{U}}=0$. Differences of
components of the $\Loc$ map satisfy divisibility conditions, see Section
\ref{sec:loc}, hence we obtain that $\om|_{x_J}$ is divisible by
$1-z_{\si(a)}/z_{\si(b)}$. This argument holds for all $a,b$ with
$\si(a)\in J_k, \si(b)\in J_l$, $k<l$ hence we proved the third claim.

\vsk.2>
A Laurent polynomial divisible by $e^{\vert}_{\si,\:J\<,-}$ and $e^{\hor}_{\si,\:J\<,-}$ must be divisible by their product too. Comparing \Ref{eqn:phiEeI} and \Ref{eqn:Ismall} we see that a
$J$-small class divisible by $e_{\si,\:J}=e^{\hor}_{\si,\:J\<,-}e^{\vert}_{\si,\:J\<,-}$ must be 0. Hence $\om|_{x_J}$ is 0. This is a contradiction, which proves that $\om=0$.

\section{Trigonometric weight functions}
\label{sec:weight functions}

\subsection{Definition}

Let $n,N\in \N$ and let $\bla\in \N^N$ be such that $\sum_{k=1}^N \la_k=n$.
Set \,$\la^{(k)}\<=\la_1\lsym+\la_k$ \,for $k=0\lc N$ and
$\la^{\{1\}}\<=\la^{(1)}\lsym+\la^{(N-1)}$\>.

\vsk.3>
Recall that the set of partitions $I=(I_1\lc I_N)$ of $\{1\lc n\}$
with $|I_k|=\la_k$ is denoted by $\Il$. For $I\?\in \I_\bla$ we will use the notation
$\bigcup_{a=1}^k I_a = \{i^{(k)}_1<\ldots< i^{(k)}_{\la^{(k)}}\}.$

\vsk.2>
Consider variables $t^{(k)}_a$ for $k=1\lc N$, $a=1\lc\la^{(k)}$,
where $t^{(N)}_a=z_a, a=1\lc n$. Denote
$t^{(j)}=(t^{(j)}_k)_{k\leq\la^{(j)}}$ and \,$\TT=(t^{(1)}\lc t^{(N-1)})$.

\vsk.2>
For $I\?\in\Il$, define the {\it trigonometric weight function}
\beq
\label{WI}
W_I(\ttt,\zz,h)\,=\,(1-h)^{\la^{\{1\}}}
\Sym_{\>t^{(1)}} \ldots \Sym_{\>t^{(N-1)}} U_I\,,
\eeq
where
\beq
\label{UI}
U_I\,=\,\prod_{k=1}^{N-1}\,\prod_{a=1}^{\la^{(k)}}\,\Bigl(\!
\prod_{\satop{c=1}{i^{(k+1)}_c<i^{(k)}_a}}^{\la^{(k)}}\!\!
\left( 1- {h t^{(k+1)}_c}/{t^{(k)}_a} \right)\!
\prod_{\satop{c=1}{i^{(k+1)}_c>i^{(k)}_a}}^{\la^{(k)}}\!\!
\left( 1- { t^{(k+1)}_c}/{t^{(k)}_a} \right)
\prod_{b=a+1}^{\la^{(k)}}\? \frac{ 1- {ht^{(k)}_b}/{t^{(k)}_a } }
{ 1- {t^{(k)}_b}/{t^{(k)}_a } }\,\Bigr)\>,\kern-1em
\eeq
and \;$\Sym_{\>t^{(k)}}$ is the symmetrization with respect to the variables
$t^{(k)}_1\lc t^{(k)}_{\la^{(k)}}$,
\vvn.2>
\be
\Sym_{\>t^{(k)}} f\bigl(t^{(k)}_1\lc t^{(k)}_{\la^{(k)}}\bigr)=\sum_{\si\in S_{\la^{(k)}}} f\bigl(t^{(k)}_{\si(1)}\lc t^{(k)}_{\si(\la^{(k)})}\bigr).
\ee
The trigonometric weight functions are Laurent polynomials in the $\ttt,\zz,h$
variables since the factors $1-{t^{(k)}_b}/{t^{(k)}_a}$ in the denominator
cancel out in the symmetrization.

\vsk.2>
For $\si\in S_n$ and $I\?\in\Il$, define the {\it trigonometric weight function}
\vvn.3>
\beq
\label{Wsi}
W_{\si,\:I}(\ttt,\zz, h)\,=\,W_{\si^{-1}(I)}(\ttt, z_{\si(1)}\lc z_{\si(n)},h)\,,
\vv.3>
\eeq
where \,$\si^{-1}(I)=\bigl(\si^{-1}(I_1)\lc\si^{-1}(I_N)\bigr)$.
Hence, \,$W_I=W_{\id\:,\:I}$.

\begin{rem}
The weight functions were described in \cite{TV1} to solve the \qKZ/ equations
and describe eigenvectors of the Hamiltonians of the \XXZ/-type integrable
models. Namely, the \,$\CNnl$\<-valued solutions of the \qKZ/ equations have
the form
\vvn.1>
\be
I(\zz)\,=\,\int\,\Phi(\ttt,\zz,h)\,
\sum_{I\in\Il}\:W_{\id\:,\:I}(\ttt,\zz,h)\;v_I\,d\:\ttt
\vv-.1>
\ee
where \,$\Phi(\ttt,\zz,h)$ \,is a suitable scalar (master) function.
If \;$\ttt$ \,satisfies Bethe ansatz equations, the vector
\,$\sum_{I\in\Il}\!W_{\id\:,\:I}(t,z,h)\,v_I$ \,becomes an eigenvector of
the Hamiltonians of the \XXZ/-type model. Convenient formulae for the weight
functions were suggested in \cite{TV3}. The weight functions
\,$W_{\id\:,\:I}$\,, \,$I\<\in\Il$\>, in this paper are Laurent polynomials.
They differ from the corresponding weight functions in \cite{TV3}\:, that are
rational functions, by a common factor independent of \,$I$\>.
\end{rem}

\subsection{Remarks}
\label{sec rem}
In this subsection we first give an alternative formula for the weight
function, and then define two modified versions. Strictly speaking none of
these are necessary in the rest of the paper, they just help the reader to get
familiar with weight functions.

Consider variables $u^{(k)}_a$ for $k=1\lc N$, $a=1\lc\la^{(k)}$. Define
\vvn.2>
\begin{align*}
U'_I\,& {}=\,
\prod_{k=1}^{N-1}\,\prod_{a,b=1}^{\la^{(k)}}
\bigl(1-{t^{(k)}_a}/{u^{(k)}_{b}}\bigr)^{-1}\times{}
\\
& {}\>\times\,
\prod_{k=1}^{N-1}\,\prod_{a=1}^{\la^{(k)}}\,\biggl(\!
\prod_{\satop{c=1\vp|}{i^{(k+1)}_c\?<\:i^{(k)}_a}}^{\la^{(k)}}\!\!\!
(1-{h u^{(k+1)}_c}/{u^{(k)}_a})\!\!
\prod_{\satop{c=1\vp|}{i^{(k+1)}_c\<>\:i^{(k)}_a}}^{\la^{(k)}}\!\!\!
(1-{u^{(k+1)}_c}/{u^{(k)}_a})
\prod_{b=a+1}^{\la^{(k)}} \frac{ 1- {hu^{(k)}_b}/ {u^{(k)}_a }}
{1- {u^{(k)}_a}/ {u^{(k)}_b }}\>\biggr)\>.
\end{align*}
Replace in $U'$ each variable $u^{(N)}_a$ with $z_a$, $a=1\lc n$.
The obtained function $U''$ depends on the variables $\ttt,\zz,h$ and
$u^{(k)}_a$, $ k=1\lc N-1$, $a=1\lc\la^{(k)}$.

\begin{thm}
We have
\vvn-.6>
\beq
\label{Res formula}
W_I(\ttt, \zz, h)=(h-1)^{\la^{\{1\}}} \Res\Bigl( U''_I \cdot
\prod_{k=1}^{N-1}\prod_{a=1}^{\la^{(k)}} {du^{(k)}_a}/{u^{(k)}_a} \Bigr),
\vv.2>
\eeq
where \,$\Res$ is the iterated application of the
\,$\Res_{u^{(k)}_a=0}+\Res_{u^{(k)}_a=\infty}$
operations for all $k=1\lc N-1$, \;$a=1\lc\la^{(k)}$
(in arbitrary order).
\end{thm}

\begin{proof}
Consider the right-hand side of \Ref{Res formula}\:. Applying the Residue Theorem
for each \,$u^{(k)}_a$, we obtain that this is equal to
\vvn-.5>
\be
(h-1)^{\la^{\{1\}}} (-1)^{\la^{\{1\}}} \sum_{\si_1\lc\si_{N-1}} \Res_{u^{(k)}_a=t^{(k)}_{\si(a)}}
\Bigl( U''_I \cdot \prod_{k=1}^{N-1}\prod_{a=1}^{\la^{(k)}} {du^{(k)}_a}/{u^{(k)}_a} \Bigr),
\vv.1>
\ee
where $\si_k$ is a map $\{1\lc\la^{(k)}\} \to \{1\lc\la^{(k)}\} $ for each $k$.
First observe that the term corresponding to $\si_1\lc\si_{N-1}$ is zero unless each of
the $\si_k$'s are permutations --- this is essentially due to the factors $(1-u^{(k)}_a/u^{(k)}_b)$
in the numerator of $U'_I$. The term corresponding to permutations $\si_k$ are exactly
the analogous terms in the definition of weight functions.
\end{proof}

The following modified version of weight functions are sometimes useful.

\begin{defn}
For $\bla\in \N^N$ define $q(\bla)$ to be the greatest $q$ with \>$\la_q>0$.
Let $M_1W_I$ be obtained from $W_I$ by substituting $t^{(k)}_i=z_i$ for
all \,$k\geq q(\bla)$ \,and all \,$i$\>.
\end{defn}

\begin{defn}
Consider variables $\gm_{k,a}$ for $k=1\lc q(\bla)-1$, $a=1\lc\la_k$.
Let $M_2W_I$ be obtained from $M_1W_I$ by carrying out the substitution
\be
( t^{(k)}_1\lc t^{(k)}_{\la^{(k)}})
\ \ \mapsto \ \
(\gm_{1,1}\lc\gm_{1,\la_1},
\gm_{2,1}\lc\gm_{2,\la_2},
\ldots,
\gm_{k,1}\lc\gm_{k,\la_k})
\ee
for all \,$k<q(\bla)$.
\end{defn}

Here are some examples of weight functions.
\begin{itemize}
\item For $N=2$, $n=2$, $\bla=(1,1)$, we have
\be
W_{(\{1\},\{2\})}=(1-h)( 1-{z_2}/{t^{(1)}_1}), \qquad
W_{(\{2\},\{1\})}=(1-h)( 1-{hz_1}/{t^{(1)}_1}).
\ee
The residue formula \Ref{Res formula} gives
\begin{align*}
W_{(\{1\},\{2\})} \,&{}=\,
(h-1)\Bigl( \Res_{u^{(1)}_1=0}+\Res_{u^{(1)}_1=\infty} \Bigr)
\Bigl( \frac{1-z_2/u^{(1)}_1}{1-t^{(1)}_1/u^{(1)}_1} \frac{du^{(1)}_1}{u^{(1)}_1} \Bigr),
\\
W_{(\{2\},\{1\})} \,&{}=\,
(h-1)\Bigl( \Res_{u^{(1)}_1=0}+\Res_{u^{(1)}_1=\infty} \Bigr)
\Bigl( \frac{1-hz_1/u^{(1)}_1}{1-t^{(1)}_1/u^{(1)}_1} \frac{du^{(1)}_1}{u^{(1)}_1} \Bigr).
\end{align*}

\item For $N=2$, $\bla=(1,n-1)$, we have
\be
W_{(\{i\},\{1\lc n\}-\{i\})}\,=\,(1-h)\,
\prod_{j=1}^{i-1}\,( 1-{hz_j}/{t^{(1)}_1})\,
\prod_{j=i+1}^{n}\,( 1-{z_j}/{t^{(1)}_1})\,.
\ee
\item
Let $\bla=(0\lc 0,1,0\lc 0,1,0\lc 0)\in \N^N$, where the
nonzero coordinates are at positions $i$ and $j$, $i<j$. The set $\Il$
consists of two elements
\begin{align*}
& I=(\{\}\lc\{\}, \{1\}, \{\}\lc\{\}, \{2\},\{\}\lc\{\}),
\\
& J=(\{\}\lc\{\}, \{2\}, \{\}\lc\{\}, \{1\},\{\}\lc\{\}).
\end{align*}
We have
\begin{align*}
M_1W_I \,&{}=\,
(1-h)^{2N-i-j}
( 1- {hz_2}/{z_1})^{N-j} ( 1- {hz_1}/{z_2})^{N-j}
( 1-{z_2}/{t^{(j-1)}_1}),
\\
M_1W_J \,&{}=\,
(1-h)^{2N-i-j}
( 1- {hz_2}/{z_1})^{N-j} ( 1- {hz_1}/{z_2})^{N-j}
( 1-{hz_1}/{t^{(j-1)}_1}),
\\
M_2W_I\,&{}=\,
(1-h)^{2N-i-j}
( 1- {hz_2}/{z_1})^{N-j} ( 1- {hz_1}/{z_2})^{N-j}
( 1-{z_2}/{\gm_{i,1}}),
\\
M_2W_J\,&{}=\,
(1-h)^{2N-i-j}
( 1- {hz_2}/{z_1})^{N-j} ( 1- {hz_1}/{z_2})^{N-j}
( 1-{hz_1}/{\gm_{i,1}}).
\end{align*}

\item We have
\begin{align*}
M_2W_{(\{1\},\{2\},\{3\})}\,= {}& \,
(1-h)^3\>(1-{z_2}/{\gm_{1,1}})\>(1-{z_3}/{\gm_{1,1}})\times{}
\\[3pt]
&{}\times(1-{z_3}/{\gm_{2,1}})\>(1-{hz_1}/{\gm_{2,1}})\>
(1-{h\gm_{2,1}}/{\gm_{1,1}})\,,
\end{align*}
although $W_{(\{1\},\{2\},\{3\})}$ does not factor into analogous simple
factors.
\end{itemize}

\section{Combinatorics of terms of the weight function}
\label{sec comb}

In this section we show a diagrammatic interpretation of the rich combinatorics
encoded in the weight function. Let $I\?\in\Il$. Consider a table with $n$ rows
and $N$ columns. Number the rows from top to bottom and number the columns from
left to right. Certain boxes of this table will be distinguished, as follows.
In the first column distinguish boxes in the $i$'th row if $i\in I_1$, in the
second column distinguish boxes in the $i$'th row if $i\in I_1 \cup I_2$, etc.
This way all the boxes in the last column will be distinguished since
$I_1\lsym\cup I_N=\{1\lc n\}$\>.

\vsk.2>
Now we will define fillings of the tables by putting various variables in
the distinguished boxes. First, put the variables $\zzz$ into the last
column from top to bottom. Now choose permutations
$\si_1\in S_{\la^{(1)}}$$\>, \,\si_2\in S_{\la^{(2)}}$\>, $\ldots$,
$\si_{N-1}\in S_{\la^{(N-1)}}$\>. Put the variables
$t^{(k)}_{\si_k(1)}\lc t^{(k)}_{\si_k(\la^{(k)})}$
in the $k$'th column from top to bottom.

\vsk.2>
Each such filled table will define a rational function as follows. Let $u$ be
a variable in the filled table in one of the columns $1,\ldots,N-1$. If $v$ is
a variable in the next column, but above the position of $u$ then consider
the factor $1-hv/u$ (`type-1 factor'). If $v$ is a variable in the next column,
but below the position of $u$ then consider the factor $1-v/u$ (`type-2
factor'). If $v$ is a variable in the same column, but below the position
of $u$ then consider the factor $(1-hv/u)/(1-v/u)$ (`type-3 factor').
The rule is illustrated in the following figure.
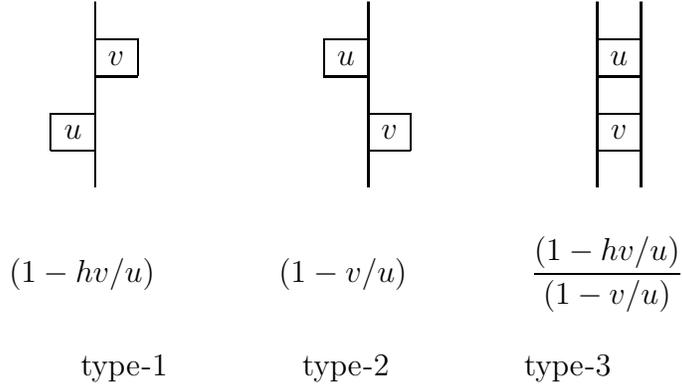
\[
\begin{tabular}{c|c}
& \\
\cline{2-2}
& \multicolumn{1}{|c|}{$v$} \\
\cline{2-2}
& \\
\cline{1-1}
\multicolumn{1}{|c|}{$u$} & \\
\cline{1-1}
&
\end{tabular} \qquad\qquad\qquad
\begin{tabular}{c|c}
& \\
\cline{1-1}
\multicolumn{1}{|c|}{$u$} & \\
\cline{1-1}
& \\
\cline{2-2}
& \multicolumn{1}{|c|}{$v$} \\
\cline{2-2}
&
\end{tabular} \qquad\qquad\qquad
\begin{tabular}{|c|}
\\
\cline{1-1}
$u$ \\
\cline{1-1}
\\
\cline{1-1}
$v$ \\
\cline{1-1}
\\
\end{tabular}
\]
\[(1-hv/u) \qquad\qquad (1-v/u) \qquad\qquad \frac{(1-hv/u)}{(1-v/u)}\]
\[\text{type\:-1} \qquad\qquad\ \text{type\:-2} \qquad\qquad\ \text{type\:-3}\]
For each variable $u$ in the table consider all these factors and multiply them together. This is ``the term associated with the filled table''.

One sees that $W_I$ is the sum of terms associated with the filled tables corresponding to all choices $\si_1,\ldots, \si_{N-1}$. For example, $W_{\{2\},\{1\},\{3\}}$ is the sum of two terms associated with the filled tables
\[
\begin{tabular}{|c|c|c|}
\hline
& $t^{(2)}_1$ & $z_1$ \\
\hline
$t^{(1)}_1$ & $t^{(2)}_2$ & $z_2$ \\
\hline
& & $z_3$ \\
\hline
\end{tabular},\qquad\qquad
\begin{tabular}{|c|c|c|}
\hline
& $t^{(2)}_2$ & $z_1$ \\
\hline
$t^{(1)}_1$ & $t^{(2)}_1$ & $z_2$ \\
\hline
& & $z_3$ \\
\hline
\end{tabular}.
\]
The term corresponding to the first filled table is hence
\[
\underbrace{(1-ht^{(2)}_1/t^{(1)}_1)\>(1-hz_1/t^{(2)}_2)}_{type-1}\,
\underbrace{(1-z_2/t^{(2)}_1)\>(1-z_3/t^{(2)}_1)\>(1-z_3/t^{(2)}_2)}_{type-2}\,
\underbrace{\frac{(1-ht^{(2)}_2)}{(1-t^{(2)}_1)} }_{type-3}\;,
\]
and the term corresponding to the second filled table is
\[
\underbrace{(1-ht^{(2)}_2/t^{(1)}_1)\>(1-hz_1/t^{(2)}_1)}_{type-1}\,
\underbrace{(1-z_2/t^{(2)}_2)\>(1-z_3/t^{(2)}_2)\>(1-z_3/t^{(2)}_1)}_{type-2}\,
\underbrace{\frac{(1-ht^{(2)}_1)}{(1-t^{(2)}_2)} }_{type-3}\;.
\]
Note that variables in consecutive columns but in the same row do not produce
a factor.

\vsk.2>
In the next section we will substitute $z_i$'s into the $t^{(k)}_a$ variables
according to some rules. Thus we obtain terms corresponding to tables filled
with only $z_a$ variables (no $t^{(k)}_a$'s). If in such a substitution we have
a filled table containing
\begin{tabular}{c|c}
& \\
\cline{1-1}
\multicolumn{1}{|c|}{$z_a$} & \\
\cline{1-1}
& \\
\cline{2-2}
& \multicolumn{1}{|c|}{$z_a$} \\
\cline{2-2}
&
\end{tabular},
then the term corresponding to that table is 0. This phenomenon is behind the substitution lemmas of the next section.

\section{Properties of weight functions}
\label{Properties of weight functions}

\subsection{Substitutions}

Recall that for $I\?\in \I_\bla$ we use the notation
$\cup_{a=1}^k I_a = \{i^{(k)}_1<\ldots$ $< i^{(k)}_{\la^{(k)}}\}.$
For a function $f(\ttt,\zz,h)$,
we denote $f(\zz_I,\zz,h)$ the substitution under
\be
t^{(k)}_{a}=z_{i^{(k)}_a}\qquad \text{for} \qquad k=1\lc N,\ \ a=1\lc\la^{(k)}.
\ee
For a function $f(\GG\<,\zz,h)$,
we denote $f(\zz_I,\zz,h)$ the substitution under
\be \{\gm_{k,a} \ |\ a=1\lc\la_k\} \mapsto \{ z_a \ |\ a\in I_k \} ,
\ee
cf.~equivariant localization in Section~\ref{sec:loc}.

Observe that the various substitutions are set up in such a way that
\vvn.3>
\be
W_I(\zz_J,\zz,h)=M_1W_I(\zz_J,\zz,h)=M_2W_I(\zz_J,\zz,h).
\ee

Define
\vvn-.6>
\beq
\label{Ela}
E(\ttt,h)\,=\,\prod_{k=1}^{N-1}\,\prod_{a=1}^{\la^{(k)}}\,
\prod_{b=1}^{\la^{(k)}}\,\bigl(1-{h\:t^{(k)}_b}\?/{t^{(k)}_a}\bigr)\,.
\vv.2>
\eeq

\begin{lem}
\label{prop:W divisible by E}
For any \,$\si\in S_n$ and \,$I,J\in\Il$, the function
\,$W_{\si,\:I}(\zz_J,\zz,h)$ \,is divisible by \,$E(\zz_J,h)$
in the algebra of Laurent polynomials \,$\Czh$.
\end{lem}
\begin{proof}
For notational simplicity we consider the case \,$\si=\id$\>. As explained in
Section~\ref{sec comb}, \,$W_{I}(\zz_J,\zz,h)$ \>is the sum of terms
corresponding to certain tables filled with the variables \,$z_a$\>. Consider
such a term, and the corresponding filled table. Let us fix \,$k\leq N-1$ \>and
\,$a\ne b\in J_1\lsym\cup J_k$. In the next paragraph we will specify some
positions in the filled table that are responsible for the appearance of the
factors \,$(1-hz_a/z_b)\>(1-hz_b/z_a)$ \>in this term. These positions will be
different for different triples \,$(k,a,b)$\>.

\vsk.2>
Suppose \,$z_a$ \>is above \,$z_b$ \>in the \,$k$-th column. Then the type-3
factor \,$(1-hz_b/z_a)$ \>is the factor of this term, because both \,$z_a$ and
\,$z_b$ are in the \,$k$-th column. Also, \,$z_a\in J_1\lsym\cup J_{k+1}$ and,
hence, \,$z_a$ is in the \,$(k+1)$-st column as well. If our term is nonzero,
then the position of \,$z_a$ in the \,$(k+1)$-st column is {\sl weakly above\/}
the position of \,$z_a$ in the \,$k$-th column. Hence the position of \,$z_a$
in the \,$(k+1)$-st column is {\em strictly above\/} the position of \,$z_b$ in
the \,$k$-th column, as in the picture
\begin{tabular}{|c|c}
& \multicolumn{1}{|c|}{$z_a$} \\
\cline{1-1}
\multicolumn{1}{|c|}{$z_a$} & \multicolumn{1}{|c|}{} \\
\cline{1-2}
& \\
\cline{1-1}
\multicolumn{1}{|c|}{$z_b$} & \\
\cline{1-1}
& \\
\end{tabular}.
Then the \,$z_b$ variable in the \,$k$-th column and the \,$z_a$ variable in
the \,$(k+1)$-st column yield the type-1 factor \,$(1-hz_a/z_b)$\>.

\vsk.2>
The function \,$E(\zz_J,h)$ \>is a product of factors
\,$(1-hz_a/z_b)\>(1-hz_b/z_a)$ \>for certain triples \,$(k,a\ne b)$\>,
and the factor \,$(1-h)^{\la^{\{1\}}}$. The argument above shows that
\,$W_I(\zz_J,\zz,h)$ \>is divisible by the desired product of factors
\,$(1-hz_a/z_b)\>(1-hz_b/z_a)$\>. The factor \,$(1-h)^{\la^{\{1\}}}$ is
explicit in the definition of the weight functions.
\end{proof}

Define
\vvn-.4>
\beq
\label{Wti}
\Wt_{\si,\:I}(\ttt,\zz,h)\,=\,\frac{W_{\si,\:I}(\ttt,\zz,h)}{E(\ttt,h)}\;.
\eeq
The function $\Wt_{\si,\:I}$ is not a Laurent {\em polynomial} in the \,$\ttt$
variables any more, but Lemma~\ref{prop:W divisible by E} asserts that all
its \,$\zz_J$-substitutions are Laurent polynomials in the $\zz,h$ variables.

\begin{lem}
\label{lem:triang}
We have \,$\Wt_{\si,\:I}(\zz_J,\zz,h)=0$ \,unless \,$J\<\leq_\si\? I$.
\end{lem}
\begin{proof}
If the condition \,$J\<\leq_\si\? I$ \>is not satisfied, then the table of
every term in \,$W_{\si,\:I}(\zz_J,\zz,h)=0$ \>contains a part described in the
last paragraph of Section~\ref{sec comb}. Hence every term is \,$0$\>, yielding
\,$W_{\si,\:I}(\zz_J,\zz,h)=0$ \>and \,$\Wt_{\si,\:I}(\zz_J,\zz,h)=0$\>.
\end{proof}

\begin{lem}
\label{lem:div_e_ver}
For all \,$I,J\in\Il$\>, the function \,$\Wt_{\si,\:I}(\zz_J,\zz,h)$
is divisible by \,$e^{\vert}_{\si,\:J\<,-}$\,.
\end{lem}
\begin{proof}
This proof is a continuation of the proof of Lemma \ref{prop:W divisible by E},
so, in particular, we focus on the special case \,$\si=\id$\>. Let us chose
\,$a\in J_k$ and \,$b\in J_l$ \>with \,$k<l$\>, \,$a>b$. Consider again a term
in \,$W_I(\zz_J,\zz,h)$ \,and its filled table. Our goal is to specify a pair
of variables in the table that produces the \,$1-hz_b/z_a$ \>factor.

\vsk.2>
From the variables \,$z_a$ and \,$z_b$ only \,$z_a$ appears in the \,$k$-th
column and both of them appear in the \,$l$-th column. We will study two cases.

\vsk.2>
Assume first that in the \,$l$-th column \,$z_a$ is below \,$z_b$\>. Then,
in the \,$(l-1)$-st column \,$z_a$ is further below the position of \,$z_b$ in
the \,$l$-th column (otherwise the term equals \,$0$\>)\:. This pair, \,$z_a$
in the \,$(l-1)$-st column and \,$z_b$ in the \,$l$-th column is the desired
pair --- they produce the type-1 factor \,$1-hz_b/z_a$\>. This factor was not
indicated and specified in the proof of Lemma \ref{prop:W divisible by E}.

\vsk.2>
Assume now that in the \,$l$-th column \,$z_a$ is above \,$z_b$\>. Since
\,$a>b$\>, their position is reversed in the \>$N$\?-th column. Hence there
must exist a number \,$s$ \,such that in the \,$s$-th column \,$z_a$ is above
\,$z_b$, and in the \,$(s+1)$-st column \,$z_a$ is below \,$z_b$. Since \,$z_a$
in the $s$-th column is below \,$z_a$ in the $s+1$-th column (otherwise the
term equals \,$0$\>)\:, we have that \,$z_a$ in the \,$s$-th column and \,$z_b$
in the \,$(s+1)$-st column is the desired pair --- they produce the type-1
factor \,$1-hz_b/z_a$\>.

\vsk.2>
In both cases above we found positions in the filled table which are different
from positions already ``used'' in the proof of
Lemma~\ref{prop:W divisible by E}. Hence we proved that \,$1-hz_b/z_a$
\>divides not only every non-zero term of \,$W_{\si,\:I}(\zz_J,\zz,h)$\>,
but also \,$\Wt_{\si,\:I}(\zz_J,\zz,h)$\>.
\end{proof}

\begin{lem}
\label{lem:princ}
For \,$I\?\in\Il$\>, we have
\,$\Wt_{\si,\:I}(\zz_I,\zz,h)=P_{\si,\:I}\,\:e_{\si,\:I}$\,.
\end{lem}

\begin{proof}
In this case only one term of the symmetrization in \Ref{WI} is nonzero,
see Section~\ref{sec comb}. This term equals the right-hand side of the formula
of the lemma.
\end{proof}

Recall the notion of $f(\zz,h)$ being \,$J$-small from
Section~\ref{sec:axiomatic}.

\begin{lem}
\label{lem:W_small}
For all \,$I,J\in\Il$, \,$I\ne J$, the function \,$\Wt_{\si,\:I}(\zz_J,\zz,h)$
is \,$J$-small.
\end{lem}

\begin{proof}
By definition we have
\be
W_{\si,\:I}(\zz_J,\zz,h)=(1-h)^{\la^{\{1\}}}
\sum_{\ppi}\Big(\prod_{(a,b)} (1-hz_a/z_b)\prod_{(c,d)} (1-z_c/z_d)
\prod_{(e,f)} \frac{ 1-hz_e/z_f }{1-z_e/z_f}
\Big),
\ee
where the products are for certain pairs $(a,b)$, $(c,d)$, and $(e,f)$, and the summation is for an
$(N-1)$-tuple of permutations $\ppi=(\pi_1,\ldots,\pi_{N-1})$ with $\pi_k\in S_{\la^{(k)}}$.
The last product can be rewritten as
\be
\prod_{(e,f)} \frac{ 1-hz_e/z_f }{1-z_e/z_f} =
\prod_{(e,f)} \frac{ z_f-hz_e }{z_f-z_e}.
\vv-.3>
\ee
Denote
\be
A_{I,J,\ppi}=\prod_{(e,f)}\:(z_f-hz_e)\,,\quad
B_{I,J,\ppi}=\prod_{(e,f)}\:(z_f-z_e)\,,\quad
C_{I,J,\ppi}=\prod_{(a,b)}\:(1-hz_a/z_b)\>\prod_{(c,d)}\:(1-z_c/z_d)\,.
\ee
Observe that $A_{I,J,\ppi}$ and $B_{I,J,\ppi}$ do not depend on $I$ and for
different $\ppi$'s the products $B_{I,J,\ppi}$ only differ possibly by a sign.
Denote $A_{J,\ppi}=A_{I,J,\ppi}$ and $B_J=B_{I,J,\bs{\id}}$. Then
\beq
\label{eqn:WABC1}
B_J W_{\si,\:I}(\zz_J,\zz,h)=(1-h)^{\la^{\{1\}}}
\sum_{\ppi} \pm C_{I,J,\ppi} A_{J,\ppi}.
\eeq
If $I$ were equal to $J$, then only one term of the summation is nonzero and
\beq
\label{eqn:WABC2}
B_J W_{\si,\:J}(\zz_J,\zz,h)=(1-h)^{\la^{\{1\}}}
C_{J,J,\bs{\id}} A_{J,\bs{\id}} .
\eeq
That equation leads to the statement of Lemma \ref{lem:princ}.
For $I\ne J$, we reason as follows.

\vsk.2>
Let $U_1=[m_1,M_1]$ and $U_2=[m_2,M_2]$ be closed intervals with
$m_1,m_2,M_1,M_2\in \Z$. For the purpose of this proof let $U_1 \ll U_2$
mean that $U_1 \subset [m_2,M_2-1]$. Also for Laurent polynomials $f,g$\>,
let $f\ll_J g$ mean $\phi_J(N(f)) \ll \phi_J(N(g))$.

\vsk.2>
We claim that $C_{I,J,\ppi} \ll_J C_{J,J,\bs{\id}}$ for all $\ppi$. Indeed, let
\begin{align*}
V &{}=\>\{\:(k,a,b)\ | \ k=1,\ldots,N-1\,,\;\;a\in J_1\lsym\cup J_k\,,
\;\;b\in J_1\lsym\cup J_{k+1}\:\}\,,
\\[3pt]
V_1 &{}=\>
\{\:(k,a,b)\ | \ k=1,\ldots,N-1\,,\;\;a\in J_1\lsym\cup J_k\,,
\;\;b\in J_1\lsym\cup J_{k+1}\,,\;\;a=b\:\}\,,
\end{align*}
Then
\vvn-.3>
\beq
\label{eqn:sm1}
C_{J,J,\bs{\id}}=\prod_{(\kab)\in V-V_1}(1-h_{\kab}\,z_b/z_a)\,,
\eeq
where the factors \,$h_{\kab}$ equal either \,$1$ or \,$h$ \,depending on
the subscript \,$(k,a,b)$\>. In other words, all factors of the product
\Ref{eqn:sm1} are either \,$1-z_b/z_a$ \>or \,$1-hz_b/z_a$\>.
We also have
\beq
\label{eqn:sm2}
C_{I,J,\bs{\pi}}=\prod_{(k,a,b)\in V-V_2}(1-h_{\kab}\,z_b/z_a)
\eeq
with the same meaning of \,$h_{\kab}$. Here $V_2 \subset V$, $V_1\not=V_2$, and
$|V_1|=|V_2|$. Therefore $C_{I,J,\bs{\pi}}$ either contains either a factor
\,$(1-z_a/z_a)=0$ or a factor \,$(1-hz_a/z_a)=1-h$\>.%
\footnote{A note to the experts: this observation is enough for the argument
in cohomology, where $J$-smallness is measured by the smallness of $z$-degree.
In K-theory the argument of the next few sentences is needed.}

If $f,g$ are Laurent polynomials, then $\phi_J(N(fg))= \phi_J(N(f)) + \phi_J(N(g))$, where $+$ is the Minkowski sum. Thus $\phi_J(N(C_{I,J,\pi}))$ is the Minkowski
sum of the intervals labeled by $(k,a,b)\in V-V_2$, with the vertices of the corresponding interval at 0 and $\phi_J(N(z_b/z_a))$.

First notice that $V_2-V_2\cap V_1$ is not empty if $J\ne I$ and hence there are elements $(k,a,b)$ which are present in the product \Ref{eqn:sm1} but not in the product \Ref{eqn:sm2}\:.
We claim that there is an element $(k,a,b)\in V_2-V_2\cap V_1$ with $\phi_J(b)>\phi_J(a)$. Indeed, choose $k$ with $I_1\cup\ldots\cup I_{k+1}=J_1\cup\ldots\cup J_{k+1}$, $I_1\cup\ldots\cup I_k\not=J_1\cup\ldots\cup J_k$, and any $b\in I_1\cup\ldots\cup I_k - J_1\cup\ldots\cup J_k$; then there is an $a$ with
$(k,a,b)\in V_2-V_2\cap V_1$. The appearance of a factor $(1-(h)z_b/z_a)$ with $\phi_J(b)>\phi_J(a)$ in \Ref{eqn:sm1} but not in \Ref{eqn:sm2} proves that $C_{I,J,\ppi} \ll_J C_{J,J,\bs{\id}}$.

Consider the Laurent polynomials $A_{J,\ppi}$ in equations \Ref{eqn:WABC1}
and \Ref{eqn:WABC2}\:. Clearly the Newton polygon of $A_{J,\ppi}$ does not
depend on $\ppi$. Therefore
\be
(1-h)^{\la^{\{1\}}} \sum_{\ppi} \pm C_{I,J,\ppi} A_{J,\ppi} \ \ll_J \
(1-h)^{\la^{\{1\}}}
C_{J,J,\bs{\id}} A_{J,\bs{\id}}\,.
\vv-.4>
\ee
Consequently, we have
\be
W_{\si,\:I}(\zz_J,\zz,h)\ \ll_J\ W_{\si,\:J}(\zz_J,\zz,h).
\ee
The Laurent polynomials on both sides of this relation are divisible by
the same Laurent polynomial $E(\zz_J,h)$, see Lemma \ref{lem:div_e_ver}.
Hence the same relation holds for the quotients,
\be
\Wt_{\si,\:I}(\zz_J,\zz,h) \ll_J \Wt_{\si,\:J}(\zz_J,\zz,h).
\ee
By Lemma \ref{lem:princ} this means that $\Wt_{\si,\:I}(\zz_J,\zz,h)$ is $J$-small.
\end{proof}

\subsection{Orthogonality}
\label{sec Orthogonality property}

The number of inversions in an ordered sequence $j_1\lc j_n$ is
the number of pairs $(a,b)$ with $a<b$, $j_a>j_b$.
Let $I\?\in \I_\bla$ where $I_k=\{i^{(k)}_1\lsym<i^{(k)}_{\la_k}\}$
as before. Let $p(I)$ denote the number of inversions in the ordered sequence
\be
I_N, I_{N-1}\lc I_1\:=\,
i^{(N)}_1\lc i^{(N)}_{\la_N}, i^{(N-1)}_1\lc i^{(N-1)}_{\la_{N-1}}\lc
i^{(1)}_1\lc i^{(1)}_{\la_1}\,.
\ee
We saw in Section~\ref{sec:flagvar} that \,$p(I)$ \,is the codimension of
\;$\Om_{\id,I}$ in \,$\Fla$.

\begin{thm}
\label{thm orth}
Let \,$\si_0$ be the longest permutation in \,$S_n$.
For \,$J,K \in\Il$, we have
\vvn.3>
\beq
\label{ORT}
\sum_{I\in\Il}\,h^{p(K)}\>P(\zz_I)\,
\frac{\Wt_{\id\:,\:J}(\zz_I,\zz,h)\,
\Wt_{\si_0,\:K}(\zz_I^{-1}\?,\zz^{-1}\?,h^{-1})}
{R(\zz_I)\,Q(\zz_I,h)}\,=\,\dl_{J,K}\,,
\vv-.2>
\eeq
where
\beq
\label{PI}
P(\zz_I)\,=\,P_{\id\:,\:I}\>P_{\si_0,\:I}\,=\,
\prod_{k<l}\,\prod_{a\in I_k}\,\prod_{b\in I_l}\,(-z_b/z_a)\,,
\eeq
\beq
\label{RI}
R(\zz_I)\,=\,P(\zz_I)\;e_{\id\:,\:I,+}^{\hor}\;e_{\id\:,\:I,-}^{\hor}\:=\,
\prod_{k<l}\,\prod_{a\in I_k}\,\prod_{b\in I_l}\,(1-z_b/z_a)\,,
\eeq
\beq
\label{QI}
Q(\zz_I,h)\,=\,e_{\id\:,\:I,+}^{\vert}\,e_{\id\:,\:I,-}^{\vert}\:=\,
\prod_{k<l}\,\prod_{a\in I_k}\,\prod_{b\in I_l}\,(1-hz_b/z_a)\,,
\eeq
and \,$\Wt_{\si_0,K}(\zz_I^{-1}\?,\zz^{-1}\?,h^{-1})$ \>is the function
obtained from \,$\Wt_{\si_0,K}(\zz_I,\zz,h)$ \>by the substitution
$\{\zzz,h\}\mapsto \{z_1^{-1}\lc z_n^{-1},h^{-1}\}$\,,
\end{thm}

The proof is given in Section~\ref{Proof of orthogonality relations}.

\subsection{Proof of existence in Theorem \ref{thm:axiomatic}}
\label{sec: existence}

We show the existence of $\ka_{\si,\:I}$ by giving an explicit formula for it.

\begin{thm}
\label{thm:existence}
For any \,$\si\in S_n$ and $I\?\in\Il$, there exists a unique element
$[\Wt_{\si,\:I}] \in K_T(\tfl)$ such that for any $J\in\Il$ we have
$\Loc_J [\Wt_{\si,\:I}]=\Wt_{\si,\:I}(\zz_J,\zz,h)$. Moreover, the classes
\beq
\label{kappa}
\ka_{\si,\:I}=[\Wt_{\si,\:I}] \ \in K_T(\tfl)
\eeq
satisfy conditions \,{\rm(\:I\:--\:III\:)} of Theorem~\ref{thm:axiomatic}.
\end{thm}
\begin{proof}
According to Lemma \ref{prop:W divisible by E},
\,$\Wt_{\si,\:I}(\zz_J,\zz,h)$ is a Laurent polynomial for all $J$, and hence
$(\Wt_{\si,\:I}(\zz_J,\zz,h))_{J\in\Il}$ is an element of the right hand side
of \Ref{eqn:loc}\:. We claim that
it is in the image of $\Loc$. Consider $\Wt_{\si,\:I}(\zz_J,\zz,h)$ and
$\Wt_{\si,\:I}(\zz_{s_{i,j}(J)},\zz,h)$. The $z_i=z_j$ substitution makes
these two Laurent polynomials equal. Hence their difference is divisible by
$1-z_i/z_j$. Therefore the element $[\Wt_{\si,\:I}] \in K_T(\tfl)$ with
$\Loc_J [\Wt_{\si,\:I}]=\Wt_{\si,\:I}(\zz_J,\zz,h)$ exists.

\vsk.2>
Properties (I)\,--\,(III) are all about restrictions of $[\Wt_{\si,\:I}]$
to fixed points $x_J$. Hence they are computed by various $\ttt=\zz_J$
substitutions in $\Wt_{\si,\:I}$. Therefore, Lemmas \ref{lem:triang},
\ref{lem:div_e_ver}, \ref{lem:princ} prove properties (I), (II), (III)
respectively.
\end{proof}

Observe that a byproduct of our proof of Theorem \ref{thm:axiomatic}
and Lemma \ref{lem:triang} is that conditions (I)\,--\,(III) imply
$\ka_{\si,\:I}|_{x_J}=0$ if $J\not\leq_\si I$.

\vsk.2>
Like in the proof above, we observe that for any \,$\si$ \,and \,$I$\>,
there is a unique element \,$[W_{\si,\:I}]\in K_T(\tfl)$ with
\;$\Loc_J [W_{\si,\:I}]=W_{\si,\:I}(\zz_J,\zz,h)$ \,for all \,$J$\>.

\begin{thm}
\label{thm:basis}
For a fixed $\si\in S_n$,
\begin{itemize}
\item{}
The set $\{[W_{\si,\:I}]\}_{I\in \I_\bla}$ is a basis of the
$\C(\zz,h)$-module $K_T(\tfl)\ox\C(\zz,h)$.

\item{} The set $\{\ka_{\si,\:I}\}_{I\in \I_\bla}$
is a basis of the $\C(\zz,h)$-module $K_T(\tfl)\ox\C(\zz,h)$.

\end{itemize}
\end{thm}

\begin{proof}
As we claimed in Section~\ref{sec:loc} the map \Ref{eqn:loc2} is
an isomorphism. Hence the statements follow from the triangularity properties
\[
\Loc_J [W_{\si,\:I}] =
\begin{cases}
\;\,\,0 & \text{if}\ J \not\leq_{\si} I \\
\;\ne 0 & \text{if}\ J=I,
\end{cases}
\qquad\qquad\qquad
\Loc_J [\ka_{\si,\:I}] =
\begin{cases}
\;\,\,0 & \text{if}\ J \not\leq_{\si} I \\
\;\ne 0 & \text{if}\ J=I,
\end{cases}
\]
see Lemmas \ref{lem:triang}, \ref{lem:princ}.
\end{proof}

\subsection{Recursive properties}

Let $\bla\in\Z_{\geq 0}^N$\>, \,$|\bla|=n$. Define an action of the symmetric
group \,$S_n$ on the set $\Il$. Let $I=(I_1\lc I_N)\in\Il$, where
$I_j=\{i_1\lc i_{\la_j}\}\subset\{1\lc n\}$\>.
For \>$\si\in S_n$\>, recall \>$\si(I) = (\si(I_1)\lc\si(I_N))$\>.

\vsk.2>
Let
\vvn-.5>
\beq
\label{btx}
\bt(x_1,x_2,y_1,y_2)\,=\,\Sym_{\>x_1,\>x_2}\,(1-h\:y_1/x_2)\>(1-y_2/x_1)\,
\frac{1-h\:x_2/x_1}{1-x_2/x_1}\;.
\eeq
It is straightforward to see that
\beq
\label{bty}
\bt(x_1,x_2,y_1,y_2)\,=\,\Sym_{\>y_1,\>y_2}\,(1-h\:y_1/x_2)\>(1-y_2/x_1)\,
\frac{1-h\:y_2/y_1}{1-y_2/y_1}\;.
\eeq
\begin{lem}
\label{btxy}
$\,\beta(x_1,x_2,y_1,y_2)$ \>is symmetric in \,$x_1,x_2$ \>and in \,$y_1,y_2$.
\qed
\end{lem}

\vsk.3>
Let \,$s_{\ab}\in S_n$ \>denote the transposition of \>$a$ and \>$b$.

\begin{thm}
\label{thm:recur}
Let \,$\si\in S_n$ \>be such that \,$\si(a)\in I_k$ and \,$\si(a+1)\in I_l$\>.
Then
\beq
\label{k=l}
W_{\si s_{a,a+1},\:I}\,=\,W_{\si,\:I}
\eeq
for \,$k=l$\,,
\vvn-.2>
\beq
\label{k<l}
W_{\si s_{a,a+1},\:I}\>=\,
h\,\frac{1-z_{\si(a)}/z_{\si(a+1)}}{1-hz_{\si(a)}/z_{\si(a+1)}}\,W_{\si,\:I}
+ \frac{1-h}{1-hz_{\si(a)}/z_{\si(a+1)}}\,W_{\si,\,s_{\si(a),\:\si(a+1)}(I)}
\vv-.1>
\eeq
for \,$k<l$\,, and
\vvn-.1>
\beq
\label{k>l}
W_{\si s_{a,a+1},\:I}\>=\,
\frac{1-z_{\si(a)}/z_{\si(a+1)}}{1-hz_{\si(a)}/z_{\si(a+1)}}\,W_{\si,\:I}+
(1-h)\,\frac{z_{\si(a)}/z_{\si(a+1)}}{1-hz_{\si(a)}/z_{\si(a+1)}}\,
W_{\si,\,s_{\si(a),\:\si(a+1)}(I)}
\vv-.2>
\eeq
for \,$k>l$\,.
\end{thm}
\begin{proof}
By formula \Ref{Wsi}\:, it suffices to prove the statement when \,$\si$ \,is
the identity permutation. Next, formula \Ref{UI} implies that it is enough
to consider only the case \,$n=2$\>. To simplify the notation, we write
$W_I=W_{\>\id,I}$ and \,$s=s_{1,2}$\>.

\vsk.2>
Let \,$k=l$, \,$I=(\Empty\lc\Empty,\{1,2\},\Empty\lc\Empty)$\>, the set
\,$\{1,2\}$ \,being at the \,$k$-th place. We compute \>$W_I$ starting
symmetrization in formula \Ref{WI} from \;$t_1^{(k)}\?,\:t_2^{(k)}$, and
using formula \Ref{btx} and Corollary \ref{btxy}\>:
\vvn-.2>
\be
W_I(\ttt,z_1,z_2)\,=\,(1-h)^{2\:\dl_{k\<,\<1}}\:
\bt(t_1^{(N-1)}\?,\:t_2^{(N-1)}\?,\:z_1,\:z_2)\,
\prod_{p=k}^{N-2}\:
\bt(t_1^{(p)}\?,\:t_2^{(p)}\?,\:t_1^{(p+1)}\?,\:t_2^{(p+1)})\,.
\vv.2>
\ee
Hence by Corollary \ref{btxy}, we have
\,$W_{s,I}(\ttt,z_1,z_2)=W_I(\ttt,z_2,z_1)=W_I(\ttt,z_1,z_2)$\>.

\vsk.2>
Let \,$k<l$\>,
\,$I=(\Empty\lc\Empty,\{1\},\Empty\lc\Empty,\{2\},\Empty\lc\Empty)$\>,
the sets \,$\{1\}$ \,and \,$\{2\}$ \,being at the \,$k$-th and \,$l$-th places,
respectively, \,and
\,$s(I)=(\Empty\lc\Empty,\{2\},\Empty\lc\Empty,\{1\},\Empty\lc\Empty)$\,.
Formula \Ref{k<l} is equivalent to the equality
\beq
\label{k<le}
h\>W_I(\ttt,z_1,z_2)+W_{s(I)}(\ttt,z_1,z_2)\,=\,
\Sym_{\>z_1,z_2} W_{s(I)}(\ttt,z_1,z_2)\,\frac{1-hz_2/z_1}{1-z_2/z_1}\;.
\eeq
We compute the left\:-hand side of \Ref{k<le} starting symmetrization in
formula \Ref{WI} from \;$t_1^{(l)}\?,\:t_2^{(l)}$, and using formula \Ref{btx}
and Corollary \ref{btxy}\>. The result of calculation is
\vvn.2>
\begin{align}
\label{k<lans}
(1-h)^{\dl_{k\<,\<1}}\:&
(1+h-h\:t_1^{(l)}\!/t_1^{(l-1)}-ht_1^{(l)}\!/t_1^{(l-1)})\times{}
\\[2pt]
&{}\times\bt(t_1^{(N-1)}\?,\:t_2^{(N-1)}\?,\:z_1,\:z_2)\,
\prod_{p=l}^{N-2}\:
\bt(t_1^{(p)}\?,\:t_2^{(p)}\?,\:t_1^{(p+1)}\?,\:t_2^{(p+1)})\,.
\notag
\end{align}
We compute the right-hand side of \Ref{k<le} starting symmetrization
from \;$z_1,z_2$, and using formula \Ref{btx} and Corollary \ref{btxy}\>,
and get the same answer \Ref{k<lans}\:. Formula \Ref{k<l} is proved.

\vsk.2>
The proof of formula \Ref{k>l} is similar.
\end{proof}

Theorem \ref{thm:recur} implies recursions for weight functions.
Set
\be
s_{\aa+1}(\zz)\>=\>(z_1\lc z_{a-1},z_{a+1},z_a,z_{a+2}\lc z_n)\,.
\ee

\begin{cor}
\label{cor:recur}
Suppose \,$a\in I_k$ and \,$a+1\in I_l$\>. Then
\vvn.3>
\beq
\label{k=lc}
W_I\bigl(\ttt,s_{\aa+1}(\zz)\bigr)\,=\,W_I(\ttt,\zz)
\eeq
for \,$k=l$\,,
\vvn-.2>
\beq
\label{k<lc}
W_{s_{a,a+1}(I)}(\ttt,\zz)\,=\,
\frac{1-hz_a/z_{a+1}}{1-z_a/z_{a+1}}\,W_I\bigl(\ttt,s_{\aa+1}(\zz)\bigr)+
(h-1)\,\frac{z_a/z_{a+1}}{1-z_a/z_{a+1}}\,W_I(\ttt,\zz)
\vv-.1>
\eeq
for \,$k<l$\,, and
\vvn-.1>
\beq
\label{k>lc}
W_{s_{a,a+1}(I)}(\ttt,\zz)\,=\,
\frac{1-h^{-1}z_{a+1}/z_a}{1-z_{a+1}/z_a}\,
W_I\bigl(\ttt,s_{\aa+1}(\zz)\bigr)+
(h^{-1}\?-1)\,\frac{z_{a+1}/z_a}{1-z_{a+1}/z_a}\,W_I(\ttt,\zz)
\vv-.2>
\eeq
for \,$k>l$\,.
\end{cor}
\begin{proof}
We take \,$\si=\id$ \;in Theorem~\ref{thm:recur} and apply formula~\Ref{Wsi}\:.
Then formulae \Ref{k=lc}\:, \Ref{k<lc}\:, \Ref{k>lc} are respective
counterparts of formulae \Ref{k=l}\:, \Ref{k>l}\:, \Ref{k<l}\:.
\end{proof}

\begin{rem}

For a function $f(x,y)$\:, define
\vvn.2>
\beq
\label{diff opers}
\der_{\xyc}\>f(x,y)\,=\,\frac{f(x,y)-f(y,x)}{x-y}\;,\qquad
\pi_{\xyc}\>f(x,y)\,=\,\der_{\xyc}\bigl(xf(x,y)\bigr)\,.\kern-.6em
\eeq
We call \;$\der_{\xyc}$ \:and \,$\pi_{\xyc}$ \:the rational and trigonometric
divided difference operators, respectively.
Formulae \Ref{k<lc} and \Ref{k>lc} respectively read
\begin{align}
\label{k<lpi}
W_{s_{a,a+1}(I)}\,&{}=\,\pi_{z_a\<,\:z_{a+1}}W_I -
hz_a \cdot\der_{z_a\<,\:z_{a+1}}W_I\,,
\\[4pt]
W_{s_{a+1,a}(I)}\,&{}=\,\pi_{z_{a+1}\<,\:z_a}W_I -
h^{-1}\<z_{a+1}\cdot\der_{z_{a+1}\<,\:z_a}W_I\,.
\notag
\end{align}
\end{rem}

\section{Stable envelope maps and $R$-matrices}
\label{sec:stab}

\subsection{Definition}
For $\si\in S_n$, we define the {\it stable envelope map}
\vvn.2>
\beq
\label{def stab map}
\Stab_\si:\>K_T\bigl((\XX_n)^T\bigr)\,\to\,\KTX\,,\qquad
1_I \mapsto \ka_{\si,\:I}\,,
\eeq
where \>$I\?\in\Il$ and \,$\bla\in\Z^n_{\geq 0}$\>, \,$|\bla|=n$\>.

\vsk.2>
The maps $\on{Stab}_\si$ become isomorphisms after tensoring the K-theory
algebras with $\C(\zz,h)$\>, see Theorem \ref{thm:basis}.
For $\si',\si\in S_n$, we define the {\it geometric \>$R$-matrix\/}
\vvn.2>
\beq
\label{Rgeom}
R_{\si'\!,\>\si} =
\St_{\si'}^{-1}\circ \St_{\si} \in
\End(K_T\bigl((\XX_n)^T\bigr))\ox\C(\zz,h)=
\End\bigl(\CNn\bigr)\ox\C(\zz,h).
\eeq

\subsection{Trigonometric $R$-matrix}
\label{Trigonometric $R$-matrix}

Let $h, z$ be parameters.
Define the {\it trigonometric $R$-matrix\/}, an element
$\Rc(z,h)\in\End(\C^N\!\ox\C^N)\ox\C(z,h)$, by the conditions:
\begin{enumerate}
\item[$\bullet$] For $i=1\lc N$,
\beq
\label{RR1}
\Rc(z,h)\>:\>v_i\ox v_i\,\mapsto\,v_i\ox v_i\,,
\eeq

\item[$\bullet$] For $1\le i<j\le N$, on the two-dimensional subspace with
ordered basis $v_i\ox v_j$, \,$v_j\ox v_i$, the trigonometric $R$-matrix
is given by the matrix
\vvn.2>
\beq
\label{RR2}
\left( \begin{array}{cccc}
\dfrac{1-z}{1-hz}
& \dfrac{1-h}{1-hz}
\\[9pt]
\dfrac{(1-h)z}{1-hz} & \dfrac{h(1-z)}{1-hz}
\end{array} \right).
\vv.2>
\eeq
\end{enumerate}
The trigonometric $R$-matrix depends on two parameters $z,h$.
We often omit in the notation the dependence on $h$.

\vsk.2>
The $2\times2$-matrix in \Ref{RR2}\:, satisfies the following relation
\vvn.2>
\beq
\label{dual R}
\left( \begin{array}{cccc}
\dfrac{1-z}{1-hz}
& \dfrac{(1-h)z}{1-hz}
\\[9pt]
\dfrac{1-h}{1-hz} & \dfrac{h(1-z)}{1-hz}
\end{array}\right)
=
\left( \begin{array}{cccc}
h^{-1} & 0
\\
0 & 1
\end{array}\right)
\left( \begin{array}{cccc}
\dfrac{1-z^{-1}}{1-h^{-1}z^{-1}}
& \dfrac{1-h^{-1}}{1-h^{-1}z^{-1}}
\\[9pt]
\dfrac{(1-h^{-1})z^{-1}}{1-h^{-1}z^{-1}} &
\dfrac{h^{-1}(1-z^{-1})}{1-h^{-1}z^{-1}}
\end{array}\right)
\left( \begin{array}{cccc}
1 & 0
\\
0 & h
\end{array}\right).\kern-.6em
\vv.2>
\eeq

\vsk.2>
The trigonometric $R$-matrix satisfies the Yang-Baxter equation
\vvn.2>
\beq
\label{YBE}
\Rc^{(1,2)}(z_2/z_1)\>\Rc^{(1,3)}(z_3/z_1)\>\Rc^{(2,3)}(z_3/z_2)\,=\,
\Rc^{(2,3)}(z_3/z_2)\>\Rc^{(1,3)}(z_3/z_1)\>\Rc^{(1,2)}(z_2/z_1)\,.\kern-.4em
\vv.2>
\eeq
This is an identity in $\End((\C^N)^{\ox 3})$ and $\Rc^{(i,j)}(z_j/z_i)$ is
the $R$-matrix $\Rc(z_j/z_i)$ acting on the $i$-th and $j$-th factors of
$(\C^N)^{\ox 3}$.

\vsk.2>
The trigonometric $R$-matrix satisfies the inversion relation
\vvn.2>
\beq
\label{Rinv}
\Rc^{(1,2)}(z_2/z_1)\>\Rc^{(2,1)}(z_1/z_2)\,=\,1\,.
\eeq

\subsection{Geometric $R\:$-matrix for \,$n=2$\>}
The group \,$S_2$ has two elements: the identity \,$\id$
\,and the transposition \,$s$\>. After the identification
\,$K_T\bigl((\XX_n)^T\bigr)\ox\C(\zz,h)=(\C^N)^{\ox2}\ox\C(\zz,h)$,
we calculate the geometric \,$R$-matrix \,$R_{\:s\<,\:\id}$ as follows.

For $\bla=(0\lc 0,2,0\lc 0)$ with the coordinate \,$2$ \,being at \,$i$-th
position, both maps \;$\St_{\>\id}$ \>and \;$\St_{s}$ send the vector
\,$v_i\ox v_i$ to \,$1\in K_T(\tfl)$\>.
Hence, $R_{\:s\<,\:\id}\>(v_i\ox v_i)=\>v_i\ox v_i$\>.

\vsk.2>
For \,$\bla=(0\lc 0,1,0\lc 0,1,0\ldots,0)$ with the nonzero coordinates \,$1$
being at \,$i$-th \,and \,$j$-th positions, \,$i<j$, the set $\Il$ consists of
two elements:
\,$I=(\Empty\lc\Empty,\{1\},\Empty\lc\Empty,\{2\},\Empty\lc\Empty)$ \,and
\,$J=(\Empty\lc\Empty,\{2\},\Empty\lc\Empty,\{1\},\Empty\lc\Empty)$\,.

\vsk.2>
By formulae for \,$M_2W_I$ and \,$M_2W_J$ from Section~\ref{sec rem} and
the equality
\vvn.2>
\be
E(\zz_I,h)\>=\>(1-h)^{2N-i-j}(1-hz_2/z_1)^{N-j}(1-hz_1/z_2)^{N-j}\,,
\ee
see~\Ref{Ela}\:, we have
\[
\St_{\>\id}\:(v_i \ox v_j)\,=\,1-z_2/\gm_{i,1}\,,\qquad
\St_{\>\id}\:(v_j \ox v_i)\,=\,1-hz_1/\gm_{i,1}\,.
\vv-.4>
\]
Similarly,
\vvn-.6>
\[
\St_s(v_i \ox v_j)\,=\,1-hz_2/\gm_{i,1}\,,\qquad
\St_s(v_j \ox v_i)\,=\,1-z_1/\gm_{i,1}\,.
\vv-.1>
\]
Thus,
\vvn-.8>
\begin{align*}
R_{\:s\<,\:\id}\>(v_i\ox v_j)\, &{}=\,
\frac{ 1-z_2/z_1}{1-hz_2/z_1}\,v_i\ox v_j+\:
\frac{(1-h)\>z_2/z_1}{1-hz_2/z_1}\,v_j\ox v_i\,,
\\[4pt]
R_{\:s\<,\:\id}\>(v_j\ox v_i)\, &{}=\,\frac{ 1-h}{1-hz_2/z_1}\,v_i\ox v_j+\:
\frac{ h\>(1-z_2/z_1)}{1-hz_2/z_1}\,v_j\ox v_i\,.
\end{align*}
Therefore, \;$R_{\:s\<,\:\id}\:=\>\Rc(z_2/z_1,h)$\,.

\subsection{Geometric $R\:$-matrices for arbitrary $n$\>}
Since for any permutations \,$\si, \si'\!, \si''$,
\vvn.2>
\beq
\label{R=RR}
R_{\:\si''\!\<,\>\si}\>=\,R_{\:\si''\!\<,\>\si'}\>R_{\:\si'\!\<,\>\si}\,,
\vv.2>
\eeq
it is enough to describe the geometric $R\:$-matrices
\,$R_{\>\si s_{a,a+1},\>\si}$ \,that correspond to permutations
\,$\si s_{\aa+1},\,\si\in S_n$\>.
\vvn.2>

\begin{thm}
\label{thm R general} We have
\be
R_{\>\si s_{a,a+1},\>\si}\>=\,
\Rc^{(\si(a),\>\si(a+1))}(z_{\si(a+1)}/z_{\si(a)})
\>\in\>\End\bigl(\CNn\bigr)\ox\C(\zz,h)\,,
\vv.2>
\ee
where \,$\Rc^{(\si(a),\>\si(a+1))}$ is the trigonometric \,$R$-matrix
\Ref{RR1}\:, \Ref{RR2} acting in the \,$\si(a)$-th \>and \,${\si(a+1)}$-th tensor
factors.
\end{thm}

\begin{proof}
The geometric $R$-matrix is defined by formulae \Ref{Rgeom}\:,
\Ref{def stab map}\:, \Ref{kappa}\:, \Ref{Wti}\:. Now the statement follows from
Theorem \ref{thm:recur} and formulae \Ref{RR1}\:, \Ref{RR2}\:.
\end{proof}

\subsection{Proof of Theorem \ref{thm orth}}
\label{Proof of orthogonality relations}

For \,$\si\in S_n$, introduce a matrix
\vvn.2>
\beq
\label{Wh}
\Wh_{\si}(\zz,h)=(\Wt_{\si\<,\:J}(\zz_I,\zz,h))_{\IJ\in\Il}
\vv.2>
\eeq
where the subscripts \,$I,J$ \,label rows and columns, respectively.
Consider the matrix
\vvn.2>
\beq
\label{Rh}
\Rh(\zz,h)\,=\,\Wh^{-1}_{\si_0}(\zz,h)\,\Wh_{\id}(\zz,h)\,.
\vv.1>
\eeq
This is the matrix of the restriction of the geometric \,$R$-matrix
\,$R_{\:\si_0,\:\id}$\>, see \Ref{Rgeom}\:, on the span of
$\{\:v_I\;|\;I\?\in\Il\}$\>.
By Theorem~\ref{thm R general} and formulae \Ref{R=RR}\:, \Ref{dual R}\:, we have
\vvn.1>
\beq
\label{Rt}
(\Rh(\zz,h))^t\>=\,M\:\Rh(\zz^{-1},h^{-1})\:\Mt\,,
\eeq
where the superscript \;$t$ \,denotes transposition and \,$M,\Mt$ \,are
diagonal matrices. The entries of $M$ are \,$M_{\IIc}=h^{-p(I)}$,
and an explicit formula for the entries of \,$\Mt$ will not be used.
Formulae \Ref{Rh}\:, \Ref{Rt} yield
\vvn.2>
\be
\bigl(\Wh_{\id}(\zz,h)\bigr)^t\>\bigl(\Wh^{-1}_{\si_0}(\zz,h)\bigr)^t\>=\,
M\>\Wh^{-1}_{\si_0}(\zz^{-1}\<,h^{-1})\,\Wh_{\id}(\zz^{-1}\<,h^{-1})\:\Mt\,.
\vv-.4>
\ee
Hence,
\vvn-.2>
\beq
\label{WW=WW}
\Wh_{\si_0}(\zz^{-1}\<,h^{-1})\>M^{-1}\>\bigl(\Wh_{\id}(\zz,h)\bigr)^t\>=\,
\Wh_{\id}(\zz^{-1}\<,h^{-1})\:\Mt\bigl(\Wh_{\si_0}(\zz,h)\bigr)^t\>.
\vv.2>
\eeq
By Lemma \ref{lem:triang}, \;$\Wt_{\id,\:J}(\zz_I,\zz,h)=0$ \,if \,$I>_{\id}J$
\,and \,$\Wt_{\si_0,\:J}(\zz_I,\zz,h)=0$ \,if \,$I<_{\id}J$\>.
That is, the matrices \,$\bigl(\Wh_{\id}(\zz,h)\bigr)^t$ and
\,$\Wh_{\si_0}(\zz^{-1}\<,h^{-1})$ \,are lower triangular,
and so is the left\:-hand side of \Ref{WW=WW}\:. Similarly, the matrices
\,$\Wh_{\id}(\zz^{-1}\<,h^{-1})$ \,and \,$\bigl(\Wh_{\si_0}(\zz,h)\bigr)^t$
are upper triangular, and so is the right\:-hand side of \Ref{WW=WW}\:.
Therefore,
\vvn.2>
\be
\Wh_{\si_0}(\zz^{-1}\<,h^{-1})\>M^{-1}\>\bigl(\Wh_{\id}(\zz,h)\bigr)^t\>
=\,S\,,
\vv.2>
\ee
where \,$S$ \>is a diagonal matrix with entries
\beq
\label{DII}
S_{I,\:I}=\:h^{p(I)}\>\Wt_{\id,\:I}(\zz_I,\zz,h)\,
\Wt_{\si_0,\:I}(\zz^{-1}_I,\zz^{-1},h^{-1})\,=\,
\frac{R(\zz_I)\>Q(\zz_I,h)}{P(\zz_I)}\;.
\eeq
Here the second equality follows from Lemma~\ref{lem:princ} and
formula~\Ref{Pe}\:. Hence,
\vvn.2>
\beq
\label{WDWM}
\bigl(\Wh_{\id}(\zz,h)\bigr)^t\>S^{-1}\,
\Wh_{\si_0}(\zz^{-1}\<,h^{-1})\>M^{-1}\>=\,1\,,
\vv.2>
\eeq
which is the matrix form of formula \Ref{ORT}\:.

\section{Inverse of the map \;$\Stab_{\>\id}$\,}
\label{invstab}

\subsection{$S_n$-action on functions}
\label{sec S-actions}

Let $P^{(\ij)}$ be the permutation of the $i$-th and $j$-th factors of
\,$\CNn$. Let
\vvn-.3>
\beq
\label{Ki}
K_i:f(\zzz) \mapsto f(z_1\lc z_{i-1}, z_{i+1},z_i, z_{i+2}\lc z_n)
\vv.2>
\eeq
be the operator interchanging the variables \>$z_i$ and \,$z_{i+1}$\>.
\vsk.2>
Define an action of the symmetric group \,$S_n$ on \,$\CNn$-valued functions
of \,$\zzz,h$\>. Let the $i$-th elementary transposition \,$s_i\in S_n$ act
by the formula
\vvn.3>
\beq
\label{Sn-}
\sti_i\>=\,P^{(\ii+1)}\,\Rc^{(i,i+1)}(z_i/z_{i+1})\>K_i\,,
\vvn.2>
\eeq
where \,$\Rc$ \>is the trigonometric $R$-matrix \Ref{RR1}\:, \Ref{RR2}\:.

\begin{lem}
\label{lem S action}
The \,$S_n$-action \Ref{Sn-} is well-defined, that is,
\vvn.1>
\be
(\sti_i)^2 = 1, \qquad \sti_i\sti_{i+1}\sti_i\>=\,
\sti_{i+1}\sti_i\:\sti_{i+1}\,,\qquad
\sti_i\:\sti_j=\sti_j\sti_i \quad \text{if}\;\;\,|\:i-j|>1\,.
\vv.1>
\ee
Moreover, \,$\sti_i\:z_i\sti_i=\:z_{i+1}$ \,and
\,$\sti_iz_j=\:z_j\sti_i$ \,if \,$j\ne i,i+1$\>,
\vvn.1>
where \>$\zzz$ are considered as the scalar operators on \,$\CNn$ of
multiplication by the respective variable.
\end{lem}
\begin{proof}
The \,$S_n$-action is well-defined due to the inversion relation~\Ref{Rinv} and
the Yang-Baxter equation~\Ref{YBE}\:. The rest of the statement is clear.
\end{proof}

\subsection{Vectors \,$\xi_I$}
\label{secxi}

Recall the partial ordering \,$\leq_\si$ on \,$\Il$ \,defined in
Section~\ref{sec:flagvar}. Set
\vvn.3>
\begin{gather}
\label{Imax}
\Imin\,=\,\bigl(\{1\lc\la_1\}\>,\{\la_1+1\lc\la_1+\la_2\}\>,\;\ldots\;,
\{n-\la_N+1\lc n\}\bigr)\in\Il\,,
\\[4pt]
\notag
\Imax\,=\,\bigl(\{n-\la_1+1\lc n\}\>,\{n-\la_1-\la_2+1\lc n-\la_1\}\>,
\;\ldots\;,\{1\lc\la_N\}\bigr)\in\Il\,.\kern-3em
\end{gather}
Clearly, \,$\Imin\leqid I\leqid\Imax$ \,for any \,$I\?\in\Il$\,.

\vsk.2>
Let
\vvn-.4>
\beq
\label{denom}
D\,=\!\prod_{1\leq b<a\leq n}\!\<(1-hz_b/z_a)\,.
\eeq

\begin{thm}
\label{xi-}
There exist unique elements \,$\{\:\xi_I\in\CNn\ox\CzhD\>\ |\ I\?\in\Il\}$
such that \;$\xi_\Imin=v_\Imin$ \,and
\vvn-.3>
\beq
\label{xi-si}
\xi_{s_i(I)}\>=\,\sti_i\>\xi_I
\vv.2>
\eeq
for every \,$I\?\in\Il$ \,and \,$i=1\lc n-1$\,. Moreover,
\vvn.2>
\beq
\label{xi-v}
\xi_I\,=\,\sum_{J\leq_\id\:I}\,X_{\IJ}\,v_J\,,
\eeq
where \,$X_{\IJ}\in \CzhD$ \>and
\vvn-.2>
\be
X_{I,\:I}\>=\,\prod_{k<l}\>\prod_{a\in I_k}\<
\prod_{\satop{b\in I_l}{b<a}}\frac{1-z_b/z_a}{1-hz_b/z_a}\;.
\vv-.4>
\ee
In particular,
\vvn-.3>
\beq
\label{X-}
X_{\Imax\<,\>\Imax}\,=\,\frac{R(\zz_\Imax)}{Q(\zz_\Imax,h)}\;.
\vv-.2>
\eeq
Furthermore,
\vvn-.3>
\beq
\label{xi-V}
X_{\IJ}\,=\,h^{p(J)}\,\Wt_{\si_0,\:J}(\zz_I^{-1}\?,\zz^{-1}\?,h^{-1})
\,\frac{P(\zz_I)}{Q(\zz_I,h)}\;.
\vv.2>
\eeq
Here \,$\si_0\in S_n$ is the longest permutation, and
\,$p(J)\>,\,P(\zz_I)\>,\,R(\zz_I)\>,\,Q(\zz_I,h)$ are defined in
Section~\ref{sec Orthogonality property}\:, see formulae
\Ref{PI}--\,\Ref{QI}\:.
\end{thm}

Notice that
\vvn-.6>
\be
X_{\IIc}\>=\,
\frac{P_{\si_0,\:I}\;e_{\id\:,\:I\<,+}^{\hor}}{e_{\id\:,\:I\<,-}^{\vert}}\;,
\vv.3>
\ee
where \,$e_{\id\:,\:I\<,+}^{\hor}$\,, \,$e_{\id\:,\:I\<,-}^{\vert}$\,, and
\,$P_{\si_0,\:I}$ \>are given by formulae \Ref{ehor}\:, \Ref{evert}\:, and
\Ref{PsiI}\:, respectively.

\begin{proof}
The properties \;$\xi_\Imin=v_\Imin$ \>and property \Ref{xi-si} imply that
\;$\xi_{\si(\Imin)}\<=\:\sit\>v_\Imin$ \>for any \,$\si\in S_n$\>, which proves
uniqueness.

\vsk.2>
To show existence, define the elements \;$\xi_I$ \>by the rule:
\,$\xi_I=\:\sit\>v_\Imin$ \>provided \,$I=\si(\Imin)$\,.
By Lemma~\ref{lem S action} and the property \,$\sit\>v_\Imin=\:v_{\Imin}$
\,for any \,$\si\in S_n$ \>such that \,$\si(\Imin)=\Imin$, these elements
\;$\xi_I$ are well-defined and satisfy \Ref{xi-si}\:, and
\;$\xi_\Imin=\:v_\Imin$\>.

\vsk.2>
Let \,$\si$ be the shortest permutation such that \,$\si(\Imin)=I$\>, and
\,$\si=s_{i_1}\ldots s_{i_k}$ be a reduced presentation. Then the equality
\,$\xi_I=\:\sti_{i_1}\?\ldots\sti_{i_k}v_\Imin$ and formulae \Ref{RR1}\:,
\Ref{RR2} \>for the $R$-matrix \,$\Rc(z,h)$ \>yield formula \Ref{xi-v}
with some coefficients \,$X_{\IJ}$\>, as well as the explicit formula for
\,$X_{\IIc}$\>.

\vsk.2>
To get formula \Ref{xi-V} for \,$X_{\IJ}$\>, we denote by \>$Y_{\IJ}$ the
right\:-hand side of formula \Ref{xi-V} and will show that the elements
\vvn-.2>
\beq
\label{eta}
\eta_I\,=\,\sum_{J\in\Il}\,Y_{\IJ}\,v_J
\eeq
satisfy the defining properties of the elements \,$\xi_I$\>. The property
\,$\eta_\Imin=v_\Imin$ \>is immediate from Lemmas~\ref{lem:triang},
\ref{lem:princ}, that imply \,$Y_{\Imin,J}=\dl_{\Imin,J}$\>.

\vsk.2>
{\bls 1.08\bls
Let \,$\Yt_{\IJ}=Y_{\IJ}\big/R(\zz_I)$ \,and
\;$\ett=\sum_{I\in\Il}\!\Yt_{IJ}\,v_J$\>. The properties
\,$\eta_{s_i(I)}=\sti_i\>\eta_I$ \>and \,$\ett_{s_i(I)}=\sti_i\>\ett_I$ \>are
equivalent, and we will check the second one. Define a matrix
\>$\Yh=(\Yt_{\IJ})_{\JI\in\Il}$, where the subscripts \,$J,I$ \,label rows
and columns, respectively, Consider \,$\Yh$ as a linear operator on the space
\,$\CNnl$ \,with basis \,$\{\:v_I\;|\;I\in\Il\}$\>, see
Section~\ref{fixed}. Then \,$\Yh:v_I\mapsto\ett_I$\>.
\vsk.2>
}

The linear maps \,$\Rc^{(i,i+1)}$ and \>$P^{(\ii+1)}$ preserve the space
\>$\CNnl$\:.
The relations \,$\ett_{s_i(I)}=\sti_i\>\ett_I$ \>are equivalent to
\vvn.1>
\beq
\label{RYY}
P^{(\ii+1)}\,\Rc^{(i,i+1)}(z_i/z_{i+1})\,K_i\,\Yh\:K_i\>P^{(\ii+1)}\:=\,\Yh\:.
\vv.1>
\eeq
Recall the matrices \,$\Wh_\si$\>, $S$ \>given by \Ref{Wh}\:, \Ref{DII}\:,
respectively, and \,$M$ \>with entries \,$M_{\IJ}=h^{-p(I)}\:\dl_{\IJ}$\>.
We have
\vvn-.3>
\be
\Yh\>=\,M^{-1}\>\bigl(\Wh_{\si_0}(\zz^{-1}\<,h^{-1})\bigr)^t\>S^{-1}\>=\,
\bigl(\Wh_{\id}(\zz,h)\bigr)^{-1}\>,
\vv.2>
\ee
the second equality following from formula \Ref{WDWM}\:.
Thus formula \Ref{RYY} transforms to
\vvn.2>
\be
P^{(\ii+1)}\,\Rc^{(i,i+1)}(z_i/z_{i+1})\,K_i\,
\bigl(\Wh_{\id}(\zz,h)\bigr)^{-1}\:K_i\>P^{(\ii+1)}\:=\,
\bigl(\Wh_{\id}(\zz,h)\bigr)^{-1}
\vv-.1>
\ee
and then to
\vvn-.2>
\beq
\label{RWW}
\Rc^{(i,i+1)}(z_{i+1}/z_i)\,=\,K_i\>P^{(\ii+1)}\,
\bigl(\Wh_{\id}(\zz,h)\bigr)^{-1}\:P^{(\ii+1)}\:K_i\,\Wh_{\id}(\zz,h)\,.
\vv.2>
\eeq
The right-hand side of~\Ref{RWW} equals the geometric \,$R$\:-matrix
\,$R_{s_i\<,\:\id}$\>, see~\Ref{Rgeom}\:. Thus formula~\Ref{RWW} follows from
Theorem~\ref{thm R general}, which proves the desired relation
\,$\ett_{s_i(I)}=\sti_i\>\ett_I$\>.
\end{proof}

\begin{example}
Let \,$N=n=2$ \,and \,$\bla=(1,1)$\>. Then
\,$\xi_{(\{1\},\{2\})}(z_1,z_2,h)=v_{(\{1\},\{2\})}$ \,and
\vvn.1>
\be
\xi_{(\{2\},\{1\})}(z_1,z_2,h)\,=\,
(1-h)\,\frac{z_1/z_2}{1-hz_1/z_2}\;v_{(\{1\},\{2\})}+
\frac{1-z_1/z_2}{1-hz_1/z_2}\;v_{(\{2\},\{1\})}\;.
\vv.3>
\ee
\end{example}

\begin{cor}
\label{basis}
The set of vectors \;$\xi_I$, \,$I\?\in\Il$, \>is a \,$\C(\zz,h)$-basis
of the space \,$\Czhhl$\>.
\qed
\end{cor}

Let \,$f_I(\zz,h)$\>, \,$I\?\in\Il$\>, be a collection of scalar functions.

\begin{lem}
\label{xinv}
The function \;$\sum_{I\in\Il} f_I(\zz,h)\,\xi_I$
\vvn.1>
\,is invariant under the \,$S_n$-action \Ref{Sn-} \,if and only if
\;$f_{\si(I)}(\zz,h)\,=\,f_I(z_{\si(1)}\lc z_{\si(n)},h)$
\>for any \,$I\?\in\Il$ and any \,$\si\in S_n$.
\qed
\end{lem}

\subsection{Inverse of \;$\Stab_{\>\id}$\,}
\label{sec:invstab}

Recall the stable envelope map
\vvn.3>
\beq
\label{stabid}
\Stab_{\>\id}:\>\Czhh\,\to\,
\KTX\ox\C(\zz,h)\,,\qquad v_I\>\mapsto\>\ka_{\>\id,\:I}\,,
\vv.3>
\eeq
where \>$I\?\in\Il$ and \,$\bla\in\Z^n_{\geq 0}$\>, \,$|\bla|=n$\>,
cf.~\Ref{def stab map}\:.

\vsk.2>
Define the homomorphism \,${\nu:\KTX\ox\C(\zz,h)\,\to\,\Czhh}$ \,of
\,$\C(\zz,h)$-modules by the rule
\vvn-.4>
\beq
\label{nu}
\nu\::\:[f(\GG\<,\zz,h)]\,\mapsto\,
\sum_{I\in\Il}\,\frac{f(\zz_I,\zz,h)}{R(\zz_I)}\;\xi_I
\eeq
for any \,$f\in\CGs\?\ox\C(\zz,h)$\>. Here \,$f(\zz_I,\zz,h)$ \,is obtained
\vvn.1>
from \,$f(\GG\<,\zz,h)$ \>by the substitution
\,$\{\gm_{k,1}\lc\<\gm_{k,\la_k}\}\mapsto\{z_a\,|\;a\in I_k\}$ \>for all
\vvn.1>
\,$k=1\lc N$. This substitution was denoted \;$\Loc_I$, see~\Ref{LocI}\:.

\begin{thm}
\label{thm stab nu}
The maps \;$\Stab_{\>\id}$ and \;$\nu$ are the inverse isomorphisms.
\end{thm}
\begin{proof}
The statement follows from the orthogonality relation \Ref{ORT} and
the formulae \Ref{xi-v}\:, \Ref{xi-V} for vectors \,$\xi_I$\>.
\end{proof}

Theorem \ref{thm stab nu} is a K-theoretic analog of \cite[Lemma~6.7]{RTV}.

\begin{rem}
\bls=1.12\bls
The group \,$S_n$ acts on \,$\KTX\ox\C(\zz,h)$ by permutations of
\,$\zzz$ \>in the second factor, \,$s_i:[\:f\:]\>\mapsto\>[K_i\:f\:]$
\,for \,$i=1\lc n-1$ \>and \>$f\in\CGs\?\ox\C(\zz,h)$\>. Under the isomorphism
\,$\Stab_{\>\id}$\>, this \,$S_n$-action is identified with the \,$S_n$-action
\Ref{Sn-} on \,$\Czhh$\>. This corollary of Theorem~\ref{thm R general} could
be considered as a motivation for the $S_n$-action \Ref{Sn-}\:.
\end{rem}

\section{Space \,$\DV$}
\label{alg sec}

\subsection{Invariant functions}
\label{secinv}
Define the operators \,$\sha_1\lc\sha_{n-1}$ \:acting on functions
of $\zzz,h$ as follows:
\vvn-.4>
\beq
\label{sha}
\sha_i\>=\,\frac{1-hz_{i+1}/z_i}{1-z_{i+1}/z_i}\,K_i +\>
\frac{h-1}{1-z_{i+1}/z_i}\;.
\vv.2>
\eeq
We consider \,$\zzz$ \:as operators of multiplication by the respective
variable.

\begin{lem}
\label{lemHecke1}
The operators \,$\sha_1\lc\sha_{n-1}$, \,$\zzz$ \:satisfy the relations
\vvn.1>
\begin{gather}
\label{Hecke1}
(\sha_i\<+1)\>(\sha_i\<-h)\:=\:0\,,\qquad
\sha_i\>\sha_{i+1}\:\sha_i=\:\sha_{i+1}\:\sha_i\:\sha_{i+1}\,,\qquad
\sha_i\:\sha_j=\:\sha_j\:\sha_i\quad \text{if}\;\,\,|\:i-j\:|>1\,,
\\[2pt]
\sha_i\:z_{i+1}\:\sha_i=\:hz_i\,,\qquad
\sha_i\:z_j=\:z_j\:\sha_i\,, \quad \text{if}\;\;\,j\ne i\:,i+1\,.
\notag
\end{gather}
\end{lem}
\begin{proof}
The statement follows from formula \Ref{sha} by direct verification.
\end{proof}

\begin{rem}
Set \,$h=q^{-2}$, \,$t_i=\:q^{-1}\:\sha_i^{-1}$\>.
Then $t_1\lc t_{n-1}$\>, $\zzz$ \:satisfy the relations
\begin{gather}
\label{t-Hecke}
(t_i\<-q)\>(t_i\< + q^{-1})=0\,,\qquad
t_i\>t_{i+1}\>t_i= t_{i+1}\> t_i\> t_{i+1}\,,\qquad
t_i\>t_j= t_j\>t_i \quad \text{if}\;\,\,|\:i-j\:|>1\,,
\\[2pt]
t_i\:z_i\:t_i=\:z_{i+1}\,,\qquad
t_i\:z_j=\:z_j\:t_i \quad \text{if}\;\;\,j\ne i\:,i+1\,.
\notag
\end{gather}
The algebra generated by $t_1\lc t_{n-1}$\>, \,$\zzz$ subject to
relations \Ref{t-Hecke} is the affine Hecke algebra of type \,$A_{n-1}$\>.
\end{rem}

Let \,$f_I(\zz,h)$\>, \,$I\?\in\Il$\>, be a collection of scalar functions and
\vvn.3>
\beq
\label{fzh}
f(\zz,h)\,=\,\sum_{I\in\Il}\,f_I(\zz,h)\;v_I\,.
\eeq
Recall the \,$S_n$-action \Ref{Sn-} on \,$\CNn$-valued functions defined
in Section~\ref{sec S-actions}.

\begin{lem}
\label{S-a}
Given $j$,\, we have \,$\sti_j f = f$ \,if and only if for any
\,$I=(I_1\lc I_N)\in\Il$ the following three conditions are satisfied:\\
\hp{\rm(iii)}\llap{\rm(i)}
$\,f_{I}=K_j\:f_I$, \;if \,$j\:,j+1\in I_a$ for some \,$a$\:;\\
\hp{\rm(iii)}\llap{\rm(ii)}
$\,f_{s_j(I)}=\sha_j^{-1} f_I$, \;if \,$j\in I_a$, \,$j+1\in I_b$, and
\,$a<b$\:;\\
{\rm(iii)}
$\,f_{s_j(I)}=\sha_j f_I$, \;if \,$j\in I_a$, \,$j+1\in I_b$, and
\,$a>b$\>.
\end{lem}
\begin{proof}
Denote by \>$\ft_I$ the \>$I$-th coordinate of \,$\sti_j f$. Then
\beq
\label{jj1}
\ft_I\>=\>K_j\:f_I
\eeq
if \,$j\:,j+1\in I_a$ \>for some \,$a$\:,
\beq
\label{j1b}
\ft_I\>=\,h\,\frac{1-z_j/z_{j+1}}{1-hz_j/z_{j+1}}\,K_j\:f_{s_j(I)} +
(1-h)\,\frac{z_j/z_{j+1}}{1-hz_j/z_{j+1}}\,K_j\:f_I
\eeq
if \,$j\in I_a$\>, \,$j+1\in I_b$\>, \,$a<b$\:, \,and
\beq
\label{j1a}
\ft_I\>=\,\frac{1-z_j/z_{j+1}}{1-hz_j/z_{j+1}}\,K_j\:f_{s_j(I)} +
\frac{1-h}{1-hz_j/z_{j+1}}\,K_j\:f_I
\eeq
if \,$j\in I_a$\>, \,$j+1\in I_b$\>, \,$a>b$\:. If \,$\sti_j f=f$, we have
\>$\ft_I=f_I$ \>for any \,$I$, and formulae \Ref{jj1}\,--\,\Ref{j1a} are
equivalent to conditions (i)\,--\,(iii)\:, respectively.
\end{proof}

\begin{rem}
Lemma~\ref{S-a} could be considered as a motivation for the operators
\,$\sha_1\lc\sha_{n-1}$\>.
\end{rem}

For \,$\si\in S_n$\>, let \,$\si=s_{j_1}\dots s_{j_l}$ be a reduced
presentation. Define \,$\sih=\sha_{j_1}\dots\sha_{j_l}$.
By Lemma~\ref{lemHecke1}, the operator \,$\sih$ \,does not depend
on the choice of the reduced presentation.

\vsk.2>
Let \,$\zz_\si\<=\:(\zzzsi)$\>. Denote by \,$\Slmax\?\subset S_n$ the isotropy
subgroup of \,$\Imax$, see \Ref{Imax}\:. For \,$I\?\in\Il$\>, let
\,$\si_I\in S_n$ be the shortest permutation such that \,$I=\si_I(\Imax)$\>.

\begin{prop}
\label{fmax}
The function \,$f(\zz,h)$, see \Ref{fzh}\:, is invariant under the
\,$S_n$-action \Ref{Sn-} \,if and only if \;$f_\Imax(\zz,h)=f_\Imax(\zz_\si,h)$
\>for any \,$\si\in\Slmax$, and \,$f_I=\sih_I(f_\Imax\<)$ \>for any
\,$I\?\in\Il$\>. Moreover,
\vvn-.6>
\beq
\label{twovx}
f(\zz,h)\,=\,
\sum_{I\in\Il}f_\Imax(\zz_{\si_0(I)},h)\,\frac{Q(\zz_I,h)}{R(\zz_I)}\;\xi_I\,,
\vv-.1>
\eeq
where \,$\si_0\<\in S_n$ is the longest permutation.
\end{prop}
\begin{proof}
The first part of the proposition follows from Lemmas~\ref{lemHecke1}
and~\ref{S-a}. To prove formula~\Ref{twovx}\:, observe that the function
in the right-hand side of~\Ref{twovx} is invariant under the \,$S_n$-action
\Ref{Sn-} by Lemma~\ref{xinv}, and has the \,$\Imax$-th coordinate
\,$f_\Imax(\zz,h)$ \,in the basis \,$\{\:v_I\;|\;I\in\Il\}$
by Theorem~\ref{xi-}.
\end{proof}

\begin{example}
By Theorem~\ref{thm:recur}, the collection
\,$f_I(\zz,h)\:=\:h^{p(I)}\,\Wt_{\si_0,\:I}(\ttt^{-1}\?,\zz^{-1}\?,h^{-1})$\>,
\,$I\?\in\Il$\>, satisfies the assumption of Lemma~\ref{S-a}. Hence,
\vvn.2>
\beq
\label{Wtb}
\Wtb(\ttt\:,\zz,h)\,=\,\sum_{I\in\Il}h^{p(I)}\,
\Wt_{\si_0,\:I}(\ttt^{-1}\:,\zz^{-1}\?,h^{-1})\,v_I
\vv-.2>
\eeq
as a function of \,$\zz,h$ \>is invariant under the \,$S_n$-action \Ref{Sn-}\:,
and formula~\Ref{twovx} yields
\beq
\label{Wtb2}
\Wtb(\ttt\:,\zz,h)\,=\,\sum_{I\in\Il}
\,\Wt_{\id\:,\:\Imin}(\ttt^{-1}\?,\zz_I^{-1}\?,h^{-1})\,
\frac{Q(\zz_I,h)}{R(\zz_I)}\;\xi_I\,,
\vv-.1>
\eeq
since \,$\Wt_{\si_0,\:\Imax}(\ttt^{-1}\?,\zz_{\si_0(I)}^{-1},h^{-1})\:=\:
\Wt_{\id\:,\:\Imin}(\ttt^{-1}\?,\zz_I^{-1}\?,h^{-1})$ \,by formula~\Ref{Wsi}\:.
Notice also that
\vvn-.2>
\beq
\label{xiI}
\xi_I\>=\,\Wtb(\zz_I,\zz,h)\,\frac{P(\zz_I)}{Q(\zz_I,h)}\;,
\eeq
either by formulae~\Ref{Wtb} and \Ref{xi-V}\:, or by formulae~\Ref{Wtb2}\:,
\Ref{PI}\:, \Ref{QI}\:, and Lemmas~\ref{lem:triang}, \ref{lem:princ}.
See also a remark on the function \,$\Wtb(\zz_I,\zz,h)$ \,at the end of
Section~\ref{sec:actxi}.
\end{example}

\subsection{Space \,$\DV$}
\label{secDV}

Let \,$\Di\Czh$ \>be the space of functions of the form $\Di f$\>,
\vvn.16>
where $f\in\Czh$ \,and \,$D$ \>is given by formula \Ref{denom}\:.
\vvn.2>
Set \,$\DiV=\CNn\?\ox\Di\Czh$ \,and \,$\DiV_\bla=\CNnl\?\ox\Di\Czh$\,.

\begin{lem}
\label{shaDiV}
The operators \,$\sha_1\lc\sha_{n-1}$ \,preserve the space \,$\Di\Czh$\>.
\end{lem}
\begin{proof}
Let \,$f\in\Czh$. Then
\be
D\cdot\sha_i\bigl(\Di f\:\bigr)\,=\,\frac{(1-hz_i/z_{i+1})\>K_i\>f+
(h-1)\>f}{1-z_{i+1}/z_i}\>\in\Czh
\vv.1>
\ee
since the numerator of the right-hand side vanishes if \,$z_i=\:z_{i+1}$.
\end{proof}

Define the operators \,$\sch_1\lc\sch_{n-1}$ as follows:
\vvn.1>
\beq
\label{sch}
\sch_i\,=\,\frac{1-hz_{i+1}/z_i} {1-z_{i+1}/z_i}\,\sti_i +
(h-1)\,\frac{z_{i+1}/z_i}{1-z_{i+1}/z_i}\;,
\vv.2>
\eeq
where \,$\sti_i$ \>is given by formula~\Ref{Sn-}\:.

\begin{lem}
\label{lemHecke2}
The operators \,$\sch_1\lc\sch_{n-1}$, \,$\zzz$ \:satisfy the relations
\vvn.2>
\begin{gather}
\label{Hecke2}
(\sch_i\<-1)\>(\sch_i\<+h)\:=\:0\,,\qquad
\sch_i\>\sch_{i+1}\:\sch_i=\:\sch_{i+1}\:\sch_i\:\sch_{i+1}\,,\qquad
\sch_i\:\sch_j=\:\sch_j\:\sch_i\quad \text{if}\;\,\,|\:i-j\:|>1\,,
\\[2pt]
\sch_i\:z_i\:\sch_i=\:hz_{i+1}\,,\qquad
\sch_i\:z_j=\:z_j\:\sch_i\,, \quad \text{if}\;\;\,j\ne i\:,i+1\,.
\notag
\end{gather}
\end{lem}
\begin{proof}
The statement follows from Lemma \ref{lem S action} and
formula \Ref{sch} by direct verification.
\end{proof}

\begin{lem}
\label{schDiV}
The operators \,$\sch_1\lc\sch_{n-1}$ \,preserve the spaces \,$\DiV$ \>and
\,$\DiV_\bla$\>.
\end{lem}
\begin{proof}
Let \,$g\in\CNnl\?\ox\Czh$\,. Then by formulae \Ref{sch}\:, \Ref{Sn-}\:,
\Ref{denom}\:,
\vvn.2>
\be
D\cdot\sch_i\bigl(\Di\:g\:\bigr)\,=\,\frac{(1-hz_i/z_{i+1})\>P^{(\ii+1)}\,
\Rc^{(i,i+1)}(z_i/z_{i+1})\,K_i\>g+(h-1)\>(z_{i+1}/z_i)\,g}{1-z_{i+1}/z_i}\;.
\vv.2>
\ee
The numerator of the right-hand side belongs to \,$\CNnl\?\ox\Czh$
\vvn.1>
\,by formulae \Ref{RR1}\:, \Ref{RR2}\:, and vanishes if \,$z_i=\:z_{i+1}$.
This proves the statement for \,$\DL$.
\vsk.2>
Since \,$\DV\!=\:\bigoplus_{|\bla|=n}\DL$\>, the lemma follows.
\end{proof}

\begin{lem}
\label{tisch}
Let \,$f(\zzz,h)$ be a \,$\CNn$\?-valued function.
Then for any \,$i=1\lc n-1$\>, \;$\sti_i f=f$ \>if and only if
\;$\sch_i f=f$\>.
\end{lem}
\begin{proof}
The statement follows from formula \Ref{sch}\:.
\end{proof}

\begin{rem}
Lemmas~\ref{schDiV}, \ref{tisch} could be considered as a motivation for
the operators \,$\sch_1\lc\sch_{n-1}$\>.
\end{rem}

{\bls1.16\bls
Denote by \,$\DV\!\subset\DiV$ \>and \,$\DL\!\subset\DiV_\bla$ the subspaces of
\vvn.1>
invariants of the operators \,$\sch_1\lc\sch_{n-1}$\>. By Lemma~\ref{tisch},
\,$\DV\!\subset\DiV$ \>and \,$\DL\!\subset\DiV_\bla$ are also the subspaces of
functions invariant under the \,$S_n$-action \Ref{Sn-}\:. All four spaces
\,$\DiV$, \,$\DiV_\bla$, \,$\DL$, and \,$\DV$ are \,$\CZH$\<-modules.
\vsk.2>}

For a function \,$f(\zz,h)$, let \,$\fc(\zz,h)=f(\zzzn,h)$\>. Denote by
\vvn.2>
\,$\Slmin\?\subset S_n$ the isotropy subgroup of \,$\Imin$, see \Ref{Imax}\:.
\vvn.1>
Let \,$\Czhl$ be the algebra of Laurent polynomials in \,$\zzz\:,\alb\:h$
\,such that \,$f(\zz,h)=f(\zz_\si,h)$ \,for any \,$\si\in\Slmin$.
Set \,$\Qc(\zz,h)=Q(\zz_\Imax,h)$\>.

\begin{lem}
\label{this}
The homomorphism \;$\thil:\Czhl\<\to\:\DL$,
\vvn-.1>
\beq
\label{thi}
\thil:f(\zz,h)\,\mapsto\,
\sum_{I\in\Il}\sih_I\Bigl(\>\frac{\fc(\zz,h)}{\Qc(\zz,h)}\:\Bigr)\,v_{I}
\vv-.2>
\eeq
is an isomorphism of \;$\CZH$-modules.
\end{lem}
\begin{proof}
Since \,$\Qc(\zz,h)$ \,divides \,$D$ \,in \,$\Czh$\,, the right-hand side of
formula~\Ref{thi} belongs to \,$\DiV_\bla$ by Lemma~\ref{shaDiV}. Moreover, it
is invariant under the \,$S_n$-action \Ref{Sn-} by Proposition~\ref{fmax} and
is preserved by the operators \,$\sch_1\lc\sch_{n-1}$ \>by Lemma~\ref{tisch}.
Hence, \,$\thil(f)\in\DL$.

\vsk.2>
The map \,$\thil$ \,is clearly injective. To prove surjectivity,
\vvn.1>
let \,$\sum_{I\in\Il}g_I(\zz,h)\,v_I\in\DL$. By Proposition~\ref{fmax},
\,$g_I=\sih_I(g_\Imax)$ \>for any \,$I\<\in\Il$, and
\,$g_\Imax(\zz,h)=g_\Imax(\zz_\si,h)$ \,for any \,$\si\in\Slmax$\>.
Therefore, the function \,$g_\Imax$ cannot have poles at the hyperplanes
\,$z_i=hz_j$ \,if \,$i\:,j\in\Imx_a$ \>for some \,$a$\>.
Hence, \,$g_\Imax=\fc/\Qc$ \,for some \,$f\in\Czhl$\>.
\end{proof}

Corollary~\ref{twovx} yields another formula for the map \,$\thil$\,:
\vvn.3>
\beq
\label{thi-xi}
\thil:f(\zz,h)\,\mapsto\,
\sum_{I\in\Il}\,\frac{f(\zz_I,h)}{R(\zz_I)}\;\xi_I\,.
\eeq

\begin{cor}
\label{isonu}
The homomorphism \,$\nu:\KTX\>\to\>\DV\!\ox\Czb$
\vvn.3>
\beq
\label{nu2}
\nu\::\:[f(\GG\<,\zz,h)]\,\mapsto\,
\sum_{I\in\Il}\,\frac{f(\zz_I,\zz,h)}{R(\zz_I)}\;\xi_I
\vv-.1>
\eeq
for any \,$f\in\CGs\?\ox\Czh$\>, \,is an isomorphism of \,$\Czh$-modules.
\qed
\end{cor}

Abusing notation, we use here the same letter for the homomorphism \,$\nu$
\>as in Section~\ref{sec:invstab}.

\begin{cor}
\label{canonic}
The canonical embeddings
\be
\iota:\DV\!\ox\C(\zz,h)\,\to\,\Czhh\,,\kern1.4em
\iota_\bla:\DL\ox\C(\zz,h)\,\to\,\Czhhl
\ee
are isomorphisms of \;$\C(\zz,h)$-modules.
\qed
\end{cor}

\subsection{Subspace \,$\DLl$}
\label{polV}

Let \,$\Dti\Czhp$ \,be the space of functions of the form \,$\Dti f$\>,
\vvn.1>
where \,$f\in\Czhp$ \,and \,$\Dt=\prod_{1\leq i<j\leq n}\,(z_j\<-hz_i)$\,,
\,cf.~\Ref{denom}\:.

\begin{lem}
\label{shaDt}
The operators \,$\sha_1\lc\sha_{n-1}$ \,preserve the space \,$\Dti\Czhp$\>.
\end{lem}
\begin{proof}
Let \,$f\in\Czhp$\>. Then
\be
\Dt\cdot\sha_i\bigl(\Dti f\:\bigr)\,=\,\frac{(z_{i+1}-hz_i)\>K_i\>f+
(h-1)\>z_i\:f}{z_i-z_{i+1}}\>\in\Czhp
\vv.2>
\ee
since the numerator of the right-hand side vanishes if \,$z_i=\:z_{i+1}$.
\end{proof}

Let \,$\Czphl$ be the algebra of polynomials in \,$\zzz\:,\,h$ \,such that
\vvn.1>
\,$f(\zz)=f(\zz_\si)$ \,for any \,$\si\in\Slmin$.
Recall \,$\la^{(a)}\<=\la_1\lsym+\la_a$\,, \,$a=1\lc N$\:.
For any \,$f\in\Czphl$\,, set
\vvn.2>
\beq
\label{thitla}
\thitl(f)\:=\:\thil\Bigl(f(\zz,h)\,
\prod_{a=1}^N\>\prod_{\>i\in\Imn_a}\!z_i^{\la^{(a)}\?-\:n}\Bigr)\,.
\eeq

Denote \;$\DLl\<=\:\DL\cap\bigl(\CNnl\ox\Dti\Czhp\bigr)$\,.

\begin{lem}
\label{that}
The homomorphism \;$\thitl:\Czphl\<\to\:\DLl$,
is an isomorphism of \;$\C[\zz]^{\>S_n}\!\ox\C[\:h\:]$-modules.
\end{lem}
\begin{proof}
Let \,$f\<\in\Czphl$ and \,$g=f\cdot
\prod_{a=1}^N\:\prod_{\>i\in\Imn_a}z_i^{\la^{(a)}\?-\:n}$.
\vvn.16>
Then \,$\thitl(f)=\thil(g)\alb\in\DL$ \:by Lemma~\ref{this}.
Also, \,$\gc/\Qc\in\Dti\Czhp$\>, \,so \,$\thitl(f)\in\CNnl\ox\Dti\Czhp$
\vvn.1>
\,by Lemma~\ref{shaDt}. Thus \,$\thitl(f)\in\DLl$, and
the map \,$\thitl$ is injective.

\vsk.2>
To prove surjectivity of \,$\thitl$, let
\,$\gt(\zz,h)=\Dti\sum_{I\in\Il}g_I(\zz,h)\,v_I\in\DLl$.
By Proposition~\ref{this}, \,$g_\Imax=\Dt\fc/\Qc$ \,for some \,$f\in\Czhl$\>.
Hence, the function
\be
\ft(\zz,h)\,=\,f(\zz,h)\,\prod_{a=1}^N\>\prod_{\>i\in\Imn_a}\!z_i^{n-\la^{(a)}}
\<=\,\gc_\Imax(\zz,h)\,
\prod_{a=1}^N\,\prod_{\satop{\ij\in\Imn_a}{i<j}}\!\<(z_i-hz_j)^{-1}\>
\vv-.2>
\ee
is regular at \,$z_k=0$ \,for all \,$k=1\lc n$\>, and at \,$h=0$\,.
Therefore, \,$\ft\in\Czphl$ \,and \,$\gt=\thitl(\ft)$.
\end{proof}

\subsection{Grading}
\label{sec:grad}
Introduce the degrees of the variables \,$\zzz$ \,and \,$h$ \,by the rule
\beq
\label{degs}
\deg\:z_1\lsym=\:\deg\:z_n=1\,,\qquad\deg\:h=0\,.
\eeq
This defines the grading on the space \,$\Czhr$\>. This grading induces
gradings on tensor products of \,$\Czhr$ \,with other vector spaces and
subspaces of those tensor products. In particular, \,$\DV$ and \,$\DL$ are
graded \,$\Czh$\:-modules, and \,$\DLl$ is a graded \,$\Czhp$\:-module.

\begin{lem}
\label{lem:thigr}
The maps \;$\thil:\Czhl\<\to\:\DL$ \>and
\;$\thitl:\Czphl\<\to\:\DLl$ \>are isomorphisms of graded spaces.
\qed
\end{lem}

For the variables \,$\gm_{\ij}$\>, set \;$\deg\:\gm_{\ij}=1$ \,for all
\,$i\:,\<j$\>. This defines the grading on the algebra \,$\KTX$\>,
see~\Ref{KTX}\:, \Ref{Hrel}\:, making it into a graded algebra.

\begin{lem}
\label{lem:nugr}
The map \;$\nu:\KTX\>\to\>\DV\!\ox\Czb$ is an isomorphism of graded spaces.
\qed
\end{lem}

\section{Quantum loop algebra}
\label{Uq}
\subsection{Quantum loop algebra $\Uqn$}
\label{Quantum loop algebra Uqn}

Let \,$q,\,u$ \>be parameters. Let \,$\C(u,q)$ \>be the algebra of rational
functions of \,$q,\,u$\>.
\vvn.2>
The {\it Cherednik-Drinfeld-Jimbo $R$-matrix}
$R(u,q)\in\End(\C^N\ox\C^N)\ox\C(u,q)$ is defined by the conditions:
\begin{enumerate}
\item[$\bullet$] For $i=1\lc N$,
\beq
\label{Rq1}
R(u,q) : v_i\ox v_i \mapsto (q\:u-q^{-1})\,v_i\ox v_i ;
\eeq

\item[$\bullet$] For $1\leq i<j\leq N$, on the two-dimensional subspace with
ordered basis $v_i\ox v_j$, $v_j\ox v_i$, the $R$-matrix $R(u,q)$
is given by the matrix
\beq
\label{Rq2}
\left( \begin{array}{cccc}
u-1 & u\>(q-q^{-1}) \\ q-q^{-1} & u-1 \end{array} \right),
\eeq
\end{enumerate}
cf.~\Ref{RR1}\:, \Ref{RR2}\:.

\vsk.2>
{\it The quantum loop algebra\/} $\Uqn$ is a unital associative algebra over
$\C(q)$ with generators $\Lh_{\ij,+0}, \Lh_{\ji,-0}$, \,$1\le j\le i\le N$,
and $\Lh_{\ij,\:\pm s}$, \,$i,j=1\lc N$, \,$s\in\Z_{>0}$, \,subject to relations
\Ref{UqL}\:, \Ref{UqR}\:, see~\cite{RS}, \cite{DF}. For convenience, set
$\Lh_{\ij,+0}=\Lh_{\ji,-0}=0$ \,for \,$1\le i< j\le N$.

\vsk.2>
Introduce the generating series \,$\Lh_{\ij,\pm}(u)=
\Lh_{\ij,\pm0}+\sum_{s=1}^\infty\Lh_{\ij,\:\pm s}\>u^{\pm s}$,
\vv.2>
and consider them as entries of $N{\times}N$ matrices
\,$\Lh_\pm(u)=(\Lh_{\ij,\pm}(u))_{i,j=1}^N$.
The relations in \,$\Uqn$ have the form
\vvn.2>
\begin{gather}
\label{UqL}
\Lh_{\ii,+0}\>\Lh_{\ii,-0}\>=\,\Lh_{\ii,-0}\>\Lh_{\ii,+0}\>=\,0\,,
\qquad i=1\lc N\:,
\\[6pt]
\label{UqR}
R^{(1,2)}(u/v,q)\>\Lh^{(1)}_\al(u)\>\Lh^{(2)}_\bt(v)\,=\,
\Lh^{(2)}_\bt(v)\>\Lh^{(1)}_\al(u)\>R^{(1,2)}(u/v,q)\,,
\end{gather}
where \,$(\al,\bt)=(+,+), (+,-), (-,-)$\,.

\subsection{Algebra \,$\Uen$}
\label{Uh}
\vsk.2>
Let $h=q^{-2}$. Let \,$\C(h)$ \>be the algebra of rational functions of
\vvn.16>
\,$h$\>. The algebra \,$\Uen$ is the unital associative algebra over \,$\C(h)$
\vvn.16>
generated by the following elements of \,$\Uqn$\,:
\vvn-.1>
\beq
\label{Lijs}
\Lt_{\ij,s}\>=\>\Lh_{1,1,+0}\dots\Lh_{i-1,\:i-1,+0}\,\;\Lh_{\ij,s}\;
\Lh_{1,1,+0}\dots\Lh_{\jj,+0}
\vv-.1>
\eeq
and
\vvn-.1>
\beq
\label{Lii}
\Lt_{\ii,+0}^{-1}=\bigl(\:\Lh_{1,1,-0}\dots\Lh_{\ii,-0}\bigr)^2
\eeq
for all possible \,$i,j,s$. Notice that \>$\Lt_{\ij,+0}=\Lt_{\ji,-0}=0$
\,for \,$1\le i<j\le N$. Consider the generating series
\vvn-.3>
\be
\Lt_{\ij,\pm}(u)\,=\,
\Lt_{\ij,\pm0}+\sum_{s=1}^\infty\,\Lt_{\ij,\:\pm s}\>u^{\pm s}
\vv-.1>
\ee
and the matrices \,$\Lt_\pm(u)=\bigl(\Lt_{\ij,\pm}(u)\bigr)_{i,j=1}^N$.
The relations in \,$\Uen$ have the form
\vvn.2>
\begin{gather}
\label{UeL1}
\Lt_{\ii,+0}\>\Lt_{\ii,+0}^{-1}\>=
\,\Lt_{\ii,+0}^{-1}\>\Lt_{\ii,+0}\>=\,1\,,\qquad i=1\lc N\:,
\\[4pt]
\label{UeL}
\Lt_{1,1,-0}\>=\>1,\qquad
\Lt_{\ii,+0}\>=\>\Lt_{i+1,\:i+1,-0}\,,\quad i=1\lc N\?-1\,,
\\[4pt]
\label{UeR}
\Rc^{(1,2)}(v/u,h)\>\Lt^{(1)}_\al(u)\>\Lt^{(2)}_\bt(v)\,=\,
\Lt^{(2)}_\bt(v)\>\Lt^{(1)}_\al(u)\>\Rc^{(1,2)}(v/u,h)\,,
\end{gather}
where \,$(\al,\bt)=(+,+), (+,-), (-,-)$\,, \,and \,$\Rc(z,h)$ is defined
by \Ref{RR1}\:, \Ref{RR2}\:.

\vsk.2>
Denote by \,$\Ueh\subset\Uen$ \,the subalgebra generated over \,$\C(h)$
by the elements
\vvn.2>
\beq
\label{Ueh}
\Lt_{1,1,\:+0}\>\lc\>\Lt_{N\?,\:N\?,\:+0}\>,
\;\Lt_{1,1,\:+0}^{-1}\>\lc\>\Lt_{N\?,\:N\?,\:+0}^{-1}\>.
\vv.2>
\eeq
The subalgebra \,$\Ueh$ \,is commutative. The elements
\,$\Lt_{N\?,\:N\?,\:+0}$ \>and \>$\Lt_{N\?,\:N\?,\:+0}^{-1}$ \,are central.

\vsk.2>
The algebra $\Uen$ is a Hopf algebra with the coproduct
\vvn.1>
\;$\Dl:\Uen\to\Uen\ox\Uen$ \,given by
\vvn-.7>
\beq
\label{Dl}
\Dl\>:\>\Lt_{\ij,\:\pm}(u)\,\mapsto\,
\sum_{k=1}^N\,\Lt_{\kj,\:\pm}(u)\ox\Lt_{\ik,\:\pm}(u)\,,\qquad
i,j=1\lc N\:.\kern-2em
\eeq

The algebra \,$\Uen$ \,is graded by the rule:
\;$\deg\:\Lt_{\ij,\:\pm s}=\:\pm\:s$ \;for all \,$i\:,\<j=1\lc N$ and
\,$s\in\Z_{\ge0}$\>, \>and \;${\deg\:h=0}$\>.

\subsection{Quantum minors}
\label{minors}
For \;$p=1\lc N$, \;$\ib=\{1\leq i_1\lsym<i_p\leq N\}$,
\;$\jb=\{1\leq j_1\lsym<j_p\leq N\}$, \,define quantum minors
\vvn.3>
\be
M_{\ijb,\:\pm}(u)\,=\,\sum_{\si\in S_p}(-1)^\si\,
\Lt_{i_1\<,\:j_{\si(1)},\:\pm}(u)\>\Lt_{i_2\<,\:j_{\si(2)},\:\pm}(uh)\>\dots\,
\Lt_{i_p\<,\:j_{\si(p)},\:\pm}(uh^{p-1})\,.
\ee

\begin{lem}
\label{lempi}
For any permutation \,$\pi\<\in S_p$\>, we have
\vvn.2>
\be
M_{\ijb,\:\pm}(u)\,=\,(-1)^\pi\sum_{\si\in S_p}(-1)^\si\,
\Lt_{i_{\pi(1)}\<,\:j_{\si(1)},\:\pm}(u)\>
\Lt_{i_{\pi(2)}\<,\:j_{\si(2)},\:\pm}(uh)\>\dots\,
\Lt_{i_{\pi(p)}\<,\:j_{\si(p)},\:\pm}(uh^{p-1})\,.
\vv-.2>
\ee
\end{lem}
\begin{proof}
The statement follows from commutation relations \Ref{UeR}\:, see, for example,
\cite[formulae (4.9), (4.10)\:]{MTV1}
\end{proof}
\begin{rem}
Though \cite{MTV1} deals with the case of the Yangian \>$\Yn$\>, the proofs
given there rely upon general properties of \>$R$-matrices and can be easily
tuned for the case of the algebra \,$\Uen$ under consideration.
\end{rem}

We also have
\vvn-.3>
\beq
\label{DlM}
\Dl\::\:M_{\ijb,\:\pm}(u)\,\mapsto\!
\sum_{\kb\>=\<\>\{1\le\>k_1\lsym<\>k_p\le N\}}\!\!
M_{\ikb,\:\pm}(u)\ox M_{\kjb,\:\pm}(u)\,,
\vv.2>
\eeq
see, for example, \cite[Proposition~1.11]{NT1} or \cite[Lemma 4.3]{MTV1}.

\vsk.3>
Introduce the series \,$A_{1,\:\pm}(u)\lc A_{N\?,\:\pm}(u)$,
\,$E_{1,\:\pm}(u)\lc E_{N-1,\:\pm}(u)$,
\,\rlap{$F_{1,\:\pm}(u)\lc F_{N-1,\:\pm}(u)$}\\
\,as follows: given \,$p$\,,
take \,$\ib=\{1\lc p\}$\,, \;$\jb=\{1\lc p-1,p+1\>\}$\,, \;and set
\vvn.1>
\begin{gather}
\label{A}
A_{p,\:\pm}(u)\,=\,M_{\iib,\:\pm}(u)\;
\Lt_{1,1,\:\pm0}^{-1}\>\dots\>\Lt_{\pp,\:\pm0}^{-1}\>=\,
1+\sum_{s=1}^\infty\>\Bin_{p,\:\pm s}\,u^{\pm s}\:,
\\[4pt]
\label{Ep}
E_{p,\:\pm}(u)\,=\,(1-h)^{-1}\>
M_{\jib,\:\pm}(u)\>\bigl(M_{\iib,\:\pm}(u)\bigr)^{-1}\,,
\\[4pt]
\label{Fp}
F_{p,\:\pm}(u)\,=\,(1-h)^{-1}\>
\bigl(M_{\iib,\:\pm}(u)\bigr)^{-1}M_{\ijb,\:\pm}(u)\,.
\\[-15pt]
\notag
\end{gather}
The coefficients of these series together with \,$\Ueh$ generate the algebra
\,$\Uen$\,. In what follows we will describe the action of \,$\Uen$ \,by using
the series \,$A_{p,\:\pm}(u)\:,\,E_{p,\:\pm}(u)\:,\, F_{p,\:\pm}(u)$\,.

\vsk.2>
By formula \Ref{DlM}\:, the series \,$A_{N\?,\pm}(u)$ \,are group\:-like:
\vvn.2>
\beq
\label{DlA}
\Dl\::\:A_{N\?,\pm}(u)\,\mapsto\,A_{N\?,\pm}(u)\ox A_{N\?,\pm}(u)\,.
\eeq

\vsk.2>
Let \,$\Bci\!\subset\Uen$ \,be the unital subalgebra generated by \,$\Ueh$
\vvn.06>
and the elements \,$\Bin_{p,\:\pm s}$ \,for \,$p=1\lc N$, \,$s\in\Z_{>0}$\>,
\vvn.06>
\,see \Ref{A}\>. The subalgebra \,$\Bci\<$ is called
the {\it\GZi/ subalgebra\/} of \,$\Uen$\>. For any \,$\Bci\<$-module \,$V$
\vvn.16>
we denote by \,$\Bci(V)$ \>the image of \,$\Bci$ in \,$\End(V)$ \>and
call \,$\Bci(V)$ \>the {\it\GZi/ algebra of\/} $V$.

\begin{thm}[\cite{KS}]
\label{BGZ}
The subalgebra \,$\Bci\<$ is commutative.
The elements \,$\Bin_{N\?,\:\pm s}$\>, \,$s\in\Z_{>0}$\>, are central.
\qed
\end{thm}

\subsection {Bethe algebra \,$\Bck$}
For \,$q=(\kkk)\in(\C^\times)^N$ and \,$p=1\lc N$, \,define
\beq
\label{Beq}
\Bk_{p,\:\pm}(u)\,=\!\sum_{\ib\>=\:\{1\leq i_1\lsym<\>i_p\leq N\}}\!\!
q_{i_1}\!\dots q_{i_p}\>M_{\iib,\:\pm}(u)
\,\prod_{r=1}^p\,\Lt_{i_r\<,\:i_r\<,\:-0}^{-1}
\,=\,\sum_{s=0}^\infty\>\Bk_{p,\:\pm s}\,u^{\pm s}\,.
\eeq
In particular,
\vvn-.3>
\be
\Bk_{p,+0}\>=\,e_p(\qt_1\lc\qt_N)\,,\qquad
\Bk_{p,-0}\>=\,e_p(q_1\lc q_N)\,,
\vv.2>
\ee
where \,$e_p$ \>is the \:$p\>$-th elementary symmetric function,
and \,$\qt_{i,\:\pm}=\:q_i\>\Lt_{\ii,\:+0}\Lt_{\ii,\:-0}^{-1}$\>.

\vsk.2>
Let \,$\Bck\<\subset\Uen$ \,be the unital subalgebra generated \,by \,$\Ueh$
\>and the elements \,$\Bk_{p,\:\pm s}$\>, \,$p=1\lc N$, \,$s\in\Z_{>0}$\,.
It is easy to see that the subalgebra \,$\Bck\<$ does not change
\vvn.06>
if all $q_1\lc q_N$ are multiplied simultaneously by the same number.
\vvn.06>
The algebra \,$\Bck\<$ is called the {\it Bethe subalgebra\/} of \,$\Uen$.
\vvn.16>
For any \,$\Bck\<$-module \,$V$ we denote by \,$\Bck(V)$ \>the image of
\,$\Bck$ in \,$\End(V)$ \>and call \,$\Bck(V)$ \>the {\it Bethe algebra of\/}
$V$.

\begin{thm}[\cite{KS}]
\label{BY}
The subalgebra \,$\Bck\<$ is commutative.
\qed
\end{thm}

The elements \,$\Bk_{p,\:\pm s}$ depend polynomially on \,$q_1\lc q_N$\>.
\vvn.1>
Suppose \,$q_1=1$ \>and \,$q_{i+1}/q_i\to 0$ \,for all \,$i=1\lc N-1$.
In this limit,
\vvn.3>
\beq
\label{Blim}
\Bk_{p,+}(u)\,=\,q_1\dots q_p\>\bigl(A_{p,+}(u)\>\Lt_{\pp,\:+0}+o(1)\bigr)\,,
\quad\Bk_{p,-}(u)\,=\,q_1\dots q_p\>\bigl(A_{p,-}(u)+o(1)\bigr)\,.
\kern-.6em
\eeq

\subsection{Difference operators}
\label{sec:diff}
Let \,$\tau$ \,be the multiplicative shift operator acting on functions
of \,$u$ \,as follows: \,$\tau f(u)=f(hu)$\>.
For \,$r=(r_1\lc r_N)\in(\C^\times)^N$ and \,$i,j=1\lc N$, set
\,$X_{\ij,\:\pm}=\:\dl_{\ij}-r_i\>\Lt_{\ij,\:\pm}(u)\>\tau$\>.
Define the difference operators
\vvn.1>
\be
\De_\pm\>=\:\sum_{\si\in S_N\!}(-1)^\si\>
X_{1,\:\si(1),\:\pm}\>X_{2,\:\si(2),\:\pm}\dots X_{N\?,\:\si(N),\:\pm}\,.
\vv-.4>
\ee
Then
\vvn-.8>
\beq
\label{Depm}
\De_\pm\>=\,1+\sum_{p=1}^N\,\sum_{\ib\>=\:\{1\leq i_1\lsym<\>i_p\leq N\}}\!\!
r_{i_1}\!\dots r_{i_p}\>M_{\iib,\:\pm}(u)\,(\<-\:\tau)^p\>,
\vv-.6>
\eeq
cf.~formula~\Ref{Beq}

\subsection{More subalgebras of \,$\Uen$}
\label{pmalg}
Let \,$\Ue_\pm,\,\Bci_\pm,\,\Bck_\pm$ \,be the subalgebras of \,$\Uen$
generated by \,$\Ueh$ and the following elements, respectively:
\begin{enumerate}
\item[$\Ue_\pm$:]
$\;\Lt_{\ij,\:\pm s}$ \,for \,$i,j=1\lc N$, \,$s\in\Z_{>0}$\,;
\vsk.2>
\item[$\Bci_\pm$:]
$\;\Bin_{p,\:\pm s}$ \,for \,$p=1\lc N$\>, \,$s\in\Z_{>0}$\,;
\vsk.2>
\item[$\Bck_\pm$:]
$\;\Bk_{p,\:\pm s}$ \,for \,$p=1\lc N$\>, \,$s\in\Z_{>0}$\,;
\end{enumerate}
\noindent
\vsk.3>
Let \,$\Zc_\pm$ be the subalgebras of \,$\Uen$ \>generated over \,$\C(h)$
by \>$\Lt_{N\?,N\?,+0}\,,\,\Lt_{N\?,N\?,+0}^{-1}$ \>and the respective elements
\,$\Bin_{N\?,\:\pm s}$\>, \,$s\in\Z_{>0}$. Let \,$\Zc\subset\Uen$ \>be
the subalgebra generated by \,$\Zc_+$ and \,$\Zc_-$\>. Recall that \,$\Zc$
\>lies in the center of \,$\Uen$.

\vsk.2>
Denote by \,$\Uen^\hg$ the commutant of \,$\Ueh$ in \,$\Uen$\>.
Recall that \,$\Bci\<,\:\Bck\<\subset\Uen^\hg$.

\section{Space \,$\DV\!\ox\C(h)$ \,as a module over \,$\Uen$}
\label{Yaction}
\subsection{Action of \,$\Uen$}
\label{tyactions}
Recall the \,$R$-matrix \,$\Rc(u)$ defined in
Section~\ref{Trigonometric $R$-matrix}. Set
\vvn.2>
\beq
\label{Lpm}
L(u)\,{}=\,\Rc^{(0,n)}(z_n/u) \dots\Rc^{(0,1)}(z_1/u)\;
\prod_{i=1}^n\,\frac{1-hz_i/u}{1-z_i/u}\,,
\eeq
where the factors of \,$(\C^N)^{\ox(n+1)}\<=\:\C^N\!\ox\CNn$ are labeled by
\vvn.1>
\,$0,1\lc n$ from left to right. We think of \>$L(u)$ \>as
an $N\!\times\!N$-matrix with \,$\End\bigl(\CNn\bigr)$\:-valued entries
\,$L_{\ij}(u)$, depending on $u, \zzz,h$.
\vsk.2>
Expand \,$L_{\ij}(u)$ into Laurent series at \,$u=0$ \,and \,$u=\infty$\>:
\beq
\label{L coeffs}
L_{\ij}(u)\,=\,L_{\ij,+0}\:+\sum_{s=1}^\infty\,L_{\ij,\:s}\,u^s\>,\qquad
L_{\ij}(u)\,=\,L_{\ij,-0}\:+\sum_{s=1}^\infty\,L_{\ij,-s}\,u^{-s}\>.
\eeq
Then \>$L_{\ij,+s}\in\End\bigl(\CNn\bigr)\ox\Czhm$ \>and \>$L_{\ij,-s}\in\End\bigl(\CNn\bigr)\ox\Czhp$
\,for \>$s\in\Z_{\ge0}$\>, \>and the degree in \>$h$ \>of each \>$L_{\ij,s}$
\>is at most \>$n$\>.

\begin{prop}
\label{pho}
The assignment \,$\pho\::\:\Lt_{\ij,\:s}\:\mapsto\:L_{\ij,\:s}$ \>for all
\>$i,j,s$, \>defines a homomorphism of graded algebras
\;$\Uen\to\:\End\bigl(\CNn\bigr)\ox\Czhr$\>.
\end{prop}
\begin{proof}
The claim follows from formulae \Ref{RR1}\:, \Ref{RR2} and the Yang-Baxter
equation \Ref{YBE}\:.
\end{proof}

Notice that
\,$\pho(\Ue_+)\subset\End\bigl(\CNn\bigr)\ox\C[\zz^{-1}\:]\ox\C(h)$ \,and
\,$\pho(\Ue_-)\subset\End\bigl(\CNn\bigr)\ox\C[\zz]\ox\C(h)$\>.

\vsk.2>
We will indicate for a while the dependence of the homomorphism \>$\pho$ \>on
\>$n$ \>by denoting it \>$\pho_n$\>. Notice that \>$\pho_n$ is the composition
of a tensor power of \>$\pho_1$ with the iteration of the coproduct \Ref{Dl}\>:
\vvn-.2>
\beq
\label{phon}
\pho_n\>=\,\pho_1^{\ox n}\?\circ\Dl^{(n)}\>.
\vv.2>
\eeq
This observation will be explored in several proofs in this section.

\begin{lem}
All operators \;$\pho(\Bin_{N\?,\:s})$ \>are scalar, and
\vvn.1>
\begin{gather}
\label{BiN+}
1+\sum_{s=1}^\infty\>\pho(\Bin_{N\?,\:s})\,u^s\:=\,
\prod_{i=1}^n\,\frac{1-h^{-1}u/\<z_i}{1-u/\<z_i}\,,
\\[4pt]
\label{BiN-}
1+\sum_{s=1}^\infty\>\pho(\Bin_{N\?,-s})\,u^{-s}\:=\,
\prod_{i=1}^n\,\frac{1-hz_i/u}{1-z_i/u}\,,
\\[-14pt]
\notag
\end{gather}
where the products in the right-hand sides are expanded at \>$u=0$ \>and
\>$u=\infty$, respectively.
\end{lem}
\begin{proof}
We will prove formula \Ref{BiN-}\:. The proof of formula \Ref{BiN+} is similar.
\vsk.2>
Let \>$n=1$\>. Let \>$v_1\lc v_N$ be the standard basis of \>$\C^N$.
To show that
\vvn.1>
\be
\pho\bigl(A_{N\?,-}(u)\bigr)\,v_i\>=\,
\frac{1-hz_1/u}{1-z_1/u}\,\,v_i\>,
\ee
we compute the quantum minor in $A_{N\?,-}(u)$ using Lemma \ref{lempi} and
taking the permutation \,$\pi$ \>such that \>$\pi(1)=i$\>.
\vsk.2>
Formula \Ref{BiN-} for general \>$n$ \>follows from formulae \Ref{phon} and
\Ref{DlA}\:, and the result for \>$n=1$\>.
\end{proof}

\begin{cor}
\label{Zpm}
We have \;$\pho(\Zc_+)\:=\>\Czms\?\ox\C(h)$ \,and
\;$\pho(\Zc_-)\:=\>\Czps\?\ox\C(h)$\>.
\end{cor}
\begin{proof}
The images \,$\pho(\Zc_\mp)$ \>contain the power sums of \>$\zzz$ \>or
\>$z_1^{-1}\lc z_n^{-1}$, respectively. For instance, formula \Ref{BiN-} yields
\vvn.2>
\be
\log\:\Bigl(1+\sum_{s=1}^\infty\>\pho(\Bin_{p,-s})\,u^{-s}\Bigr)\>=\,
\sum_{s=1}^\infty\,\frac{h^s\?-1}s\;(z_1^s\<\lsym+z_n^s)\,u^{-s}\>.
\vv-1.1>
\ee
\end{proof}
Denote \;$y\:=z_1\dots z_n$\>. Notice that \;$y\in\pho(\Zc_-)$ \>and
\;$y^{-1}\!\in\pho(\Zc_+)$\,. Corollary~\ref{Zpm} implies
\vvn.2>
\beq
\label{Zpmy}
\pho(\Zc)\,=\,\pho(\Zc_+)\ox\C[\:y\:]\,=\,\pho(\Zc_-)\ox\C[\:y^{-1}\:]\,.
\vv.1>
\eeq

\begin{cor}
\label{Upm}
We have
\vvn.2>
\begin{align}
\label{Upmy}
\pho\bigl(\Uen\bigr)\, &{}=\,
\pho(\Ue_+)\ox\C[\:y\:]\,=\,\pho(\Ue_+)\ox\pho(\Zc_-)\,,
\\[4pt]
\notag
&{}=\,\pho(\Ue_-)\ox\C[\:y^{-1}\:]\,=\,\pho(\Ue_-)\ox\pho(\Zc_+)\,,
\end{align}
\vv-.8>
\begin{align}
\label{Biy}
\pho(\Bci)\,&{}=\,\pho(\Bci_+)\ox\C[\:y\:]\,=\,\pho(\Bci_+)\ox\pho(\Zc_-)\,,
\\[4pt]
\notag
&{}=\,\pho(\Bci_-)\ox\C[\:y^{-1}\:]\,=\,\pho(\Bci_-)\ox\pho(\Zc_+)\,,
\end{align}
\vv-.8>
\begin{align}
\label{Bky}
\pho(\Bck)\,&{}=\,\pho(\Bck_+)\ox\C[\:y\:]\,=\,\pho(\Bck_+)\ox\pho(\Zc_-)\,,
\\[4pt]
\notag
&{}=\,\pho(\Bck_-)\ox\C[\:y^{-1}\:]\,=\,\pho(\Bck_-)\ox\pho(\Zc_+)\,.
\end{align}
\end{cor}
\begin{proof}
\baselineskip=1.2\baselineskip
The product \,$T_+(u)=L(u)\,\prod_{i=1}^n(1-u/\<z_i)$ \,is a polynomial
in \>$u$\>. By Corollary~\ref{Zpm}, the coefficients of \,$T_+(u)$
\,belong to \,$\pho(\Ue_+)$ \,and together with \,$\pho(\Zc_+)$ \,generate
\,$\pho(\Ue_+)$\>. Similarly, \,$T_-(u)=L(u)\,\prod_{i=1}^n(1-z_i/u)$
\,is a polynomial in \>$u^{-1}$ and the coefficients of \,$T_-(u)$
\,together with \,$\pho(\Zc_-)$ \,generate \,$\pho(\Ue_-)$. Since
\,$(-u)^n\>T_-(u)=y\>T_+(u)$ \,and taking into account~\Ref{Zpmy}\,, the first
equalities in relation~\Ref{Upmy} follow. The second equalities in~\Ref{Upmy}
hold because \;$y\in\pho(\Zc_-)$ \>and \;$y^{-1}\!\in\pho(\Zc_+)$\,.
\vsk.2>
Relations \Ref{Biy} and \Ref{Bky} follow by the same reasoning from
the definition of the sub\-algebras involved.
\end{proof}

\vsk.2>
The homomorphism \,$\pho:\Uen\to\:\End\bigl(\CNn\bigr)\ox\Czhr$ defines an action of
the algebra \,$\Uen$ on \,$\CNn$-valued functions of \>$\zzz$ \>and \>$h$\>.
In what follows when acting by $X\?\in\Uen$ \>on a \,$\CNn$-valued function
$f(\zz,h)$\>, we will write \,$Xf$ \>instead of \,$\pho(X)\>f$\>.
Clearly, for any $i=1\lc N$ and \,$I\?\in\Il$ \,we have
\vvn.2>
\beq
\label{Ltl}
\Lt_{\ii,+0}\>v_I\>=\,h^{\la_1\lsym+\la_i}\:v_I\,,\qquad
\Lt_{\ii,-0}\>v_I\>=\,h^{\la_1\lsym+\la_{i-1}}\:v_I\,.
\vv.2>
\eeq
Therefore, for any \>$X\<\in\Uen^\hg$ the operator \,$\pho(X)$ preserves
the subspaces \,$\CNnl$. We denote by
\vvn-.2>
\beq
\label{phola}
\phol\<:\Uen^\hg\<\to\:\End\bigl(\CNnl\bigr)\ox\Czhr
\vv.2>
\eeq
the corresponding homomorphism.

\begin{lem}
\label{lem ac comS}
The \,$\Uen$-action commutes with the \>$S_n\?$-action defined by \>\Ref{Sn-}\:,
that is, \;$\sti_i\>\pho(X)=\pho(X)\>\sti_i$ \,for any \,$i=1\lc n-1$ \>and
$X\<\in\Uen$\>.
\end{lem}
\begin{proof}
The statement follows from the Yang-Baxter equation \Ref{YBE}\:.
\end{proof}

\begin{cor}
\label{Usch}
The \,$\Uen$-action commutes with the operators \,$\sch_1\lc\sch_{n-1}$
\>given by \Ref{sch}\>.
\qed
\end{cor}

\begin{cor}
\label{DVmod}
The homomorphism \,$\pho:\Uen\to\:\End\bigl(\CNn\bigr)\ox\Czhr$
makes the spaces \,$\DV\!\ox\C(h)$ \>and \,$\DV\!\ox\Czb\ox\C(h)$
\>into graded \>$\Uen$-modules.
\qed
\end{cor}

\subsection{Action of \,\:$\Bci\?$ on the vectors \,$\xi_I$\>}
\label{sec:actxi}

By formulae~\Ref{Ltl}\:, for any $i=1\lc N$ and \,$I\?\in\Il$ \,we have
\beq
\label{Ltxi}
\Lt_{\ii,+0}\,\xi_I\>=\,h^{\la_1\lsym+\la_i}\>\xi_I\,,\qquad
\Lt_{\ii,-0}\, \xi_I\>=\,h^{\la_1\lsym+\la_{i-1}}\>\xi_I\,.
\vv.2>
\eeq

\begin{thm}
\label{AEFxi}
We have
\beq
\label{Axi}
A_{p,+}(u)\,\xi_I\,=\,\xi_I\;
\prod_{a=1}^p\,\prod_{i\in I_a}\,\frac{1-h^{-1}u/\<z_i}{1-u/\<z_i}\;,\qquad
A_{p,-}(u)\,\xi_I\,=\,\xi_I\;
\prod_{a=1}^p\,\prod_{i\in I_a}\,\frac{1-hz_i/u}{1-z_i/u}
\vv-.3>
\eeq
\begin{gather}
\label{Exip}
E_{p,+}(u)\,\xi_I\,=\,-
\sum_{i\in I_{p+1}\!}\,\frac{\xi_{I^{\ip}}}{1-u/\<z_i}\,
\prod_{\satop{j\in I_{p+1}\!}{j\ne i}}\!\frac{1-hz_j/\<z_i}{1-z_j/\<z_i}\;,
\\[2pt]
\label{Exim}
E_{p,-}(u)\,\xi_I\,=
\sum_{i\in I_{p+1}\!}\,\xi_{I^{\ip}}\,\frac{z_i/u}{1-z_i/u}\,
\prod_{\satop{j\in I_{p+1}\!}{j\ne i}}\!\frac{1-hz_j/\<z_i}{1-z_j/\<z_i}\;,
\end{gather}
\vvn-.8>
\begin{gather}
\label{Fxip}
F_{p,+}(u)\,\xi_I\,=\,-
\sum_{i\in I_p\!}\;\xi_{I^{\ipi}}\,\frac{u/\<z_i}{1-u/\<z_i}\;
\prod_{\satop{j\in I_p\!}{j\ne i}}\,\frac{1-hz_i/\<z_j}{1-z_i/\<z_j}\;,
\\[2pt]
\label{Fxim}
F_{p,-}(u)\,\xi_I\,=
\sum_{i\in I_p\!}\;\frac{\xi_{I^{\ipi}}}{1-z_i/u}\;
\prod_{\satop{j\in I_p\!}{j\ne i}}\,\frac{1-hz_i/\<z_j}{1-z_i/\<z_j}\;,
\end{gather}
where the sequences \,$I^{\ip}$ and \,$I^{\ipi}$ are defined as
follows: \,$I^{\ip}_a\!=I^{\ipi}_a\!=I_a$ \,for \,$a\ne p,p+1$,
\,and \,$I^{\ip}_p\!=I_p\cup\iset$\>, \,$I^{\ip}_{p+1}\!=I_{p+1}-\iset$\>,
\,$I^{\ipi}_p\!=I_p-\iset$\>, \,$I^{\ipi}_{p+1}\!=I_{p+1}\cup\iset$\>.
\end{thm}
\begin{proof}
First observe that by formula \Ref{xi-si} and Lemma \ref{lem ac comS}, it
suffices to prove formulae~\Ref{Axi}\,--\,\Ref{Fxim} only for \,$I=\Imin$.
In this case, formula \Ref{Axi} for \,$n>1$ follows from the coproduct
formula~\Ref{DlM} and the \,$n=1$ \,case of \Ref{Axi}\:. The proof of \Ref{Axi}
for \,$n=1$ \,is straightforward by using Lemma~\ref{lempi}.

\vsk.2>
The proofs of formulae \Ref{Exip}\,--\,\Ref{Fxim} are similar to each other.
As an example, we prove formula \Ref{Fxim}\:. To verify \Ref{Fxim} for
\,$I=\Imin$, observe that by formulae \Ref{xi-v}\:, \Ref{DlM}\:, \Ref{Fp}\:,
\Ref{Axi}\:, we have
\,$F_{p,-}(u)\,\xi_\Imin=\sum_{\>i\in\Imn_p\!}\>c_i\,\xi_\Iminp$\>.
The smallest element of \,$\Imn_p$ \>equals
\,$i_{\min}=\la_1\<\lsym+\la_{p-1}\<+1$.
The coefficient \,$c_{i_{\min}}\?$ can be calculated due to the triangularity
property \Ref{xi-v}\:, and does have the required form. The coefficient \,$c_i$
for other \,$i\in I^{\:\min}_p$ can be obtained from \,$c_{i_{\min}}\?$ by
permuting \,$z_i$ \>and \,$z_{i_{\min}}\?$ because \,$\Imin$ is invariant under
the transposition of \,$i$ and \,$i_{\min}$\>. Thus all the coefficients
\,$c_i$ are as required, which proves formula \Ref{Fxim}\:.
\end{proof}

\begin{rem}
Notice that the right-hand sides of formulae~\Ref{Exip} and \Ref{Exim} coincide
as rational functions. This function is expanded at \,$u=0$ \,in \Ref{Exip}
and \,at $u=\infty$ \,in \Ref{Exim}\:. Similarly, the right-hand sides of
\Ref{Fxip} and \Ref{Fxim} are the same rational function that is expanded
at \,$u=0$ \,in \Ref{Fxip} and at \,$u=\infty$ \,in \Ref{Fxim}\:.
\end{rem}

\begin{rem}
The \,$\CNn$\<-valued function \,$\Wtb(\ttt\:,\zz,h)$ \,in formula~\Ref{Wtb}
is known as the off-shell Bethe vector. The values of that function at
\;$\ttt=\zz_I$, \,$I\in\Il$\>, give the eigenvectors \,$\xi_I$ \>of the Bethe
algebra $\Bci$, see formula~\Ref{xiI}\:.
\end{rem}

\subsection{Isomorphism \,${\psil:\Czhlr\to\phol(\Bci)}$\,}
\label{sec:psi}

Let \,$\Czl$ \<be the algebra of\\[1pt]
Laurent polynomials such that
\vvn.16>
\,$f(\zz)=f(\zz_\si)$ \,for any \,$\si\in\Slmin$. For \,${g\in\Czhlr}$\>,
define \,$\psil(g)\in\End\bigl(\CNnl\bigr)\ox\C(\zz,h)$ \>by the rule
\vvn.1>
\beq
\label{psi(g)}
\psil(g)\::\>\xi_I\,\mapsto\,g(\zz_I,h)\,\xi_I\,,\qquad I\in\Il\,,\kern-2em
\eeq
see Lemma~\ref{basis}. The map
\vvn.1>
\beq
\label{psila}
\psil:\:\Czhlr\,\to\,\End\bigl(\CNnl\bigr)\ox\C(\zz,h)\,.
\vv.1>
\eeq
is clearly a monomorphism of graded algebras.

\begin{lem}
\label{psiS}
For any \,$f\in\CZSr$\>, we have \,$\psil(f)\:=\:\id\ox f$\>.
\qed
\end{lem}

\begin{thm}
\label{thm B-infty alg}
We have \;$\psil\bigl(\Czhlr\bigr)\>=\>\phol(\Bci)$\>.
\end{thm}
\begin{proof}
By formulae \Ref{Axi}\:, \;$\phol(\Bci)$ \,is generated over \,$\C(h)$
\>by the images of the power sums \,$z_1^s\lsym+z_{\la^{(p)}}^s$\>, where
\,$\la^{(p)}\?=\:\la_1\<\lsym+\la_p$\,, for all \,$p=1\lc N$ and \,$s\in\Z$\>,
cf.~the proof of Corollary~\ref{Zpm}.
\end{proof}

\begin{cor}
\label{corpsi}
We have \;$\psil\bigl(\Czhlr\bigr)\subset\End\bigl(\CNnl\bigr)\ox\Czhr$\>.
\qed
\end{cor}

For any \,${g\in\Czhl}$\>, the operator \,$\psil(g)$ preserves the space
$\DL$, see Lemma~\ref{this} and formula~\Ref{thi-xi}\:. The restriction
of \,$\psil(g)$ \>to \,$\DL$ \>can be also presented as follows:
\vvn.2>
\beq
\psil(g)\,=\,\thil\circ m(g)\circ\thil^{-1}\:,
\vv.2>
\eeq
where \,$m(g)$ \>is the operator of multiplication by \,$g$ \,on \,$\Czhl$\>,
see~\Ref{thi}\:.

\vsk.2>
Let \,$\Ac$ \,be a commutative algebra. The algebra \,$\Ac$ \,considered
as an \,$\Ac$-module with any element of \,$\Ac$ \,acting by multiplication
on itself is called the {\it regular representation\/} of $\Ac$.

\begin{thm}
\label{regrep}
The \;$\CZSr$-module isomorphism
\vvn.3>
\beq
\label{thib}
\thil:\Czhlr\:\to\,\DL\?\ox\C(h)
\vv.2>
\eeq
\,and the \,${\CZSr}$-algebra isomorphism
\;${\psil\<:\Czhlr\:\to\:\phol(\Bci)}$
\vvn.16>
\>identify the \,$\phol(\Bci)$-module \,$\DL\?\ox\C(h)$ \,with
the regular representation of \,\,$\Czhlr$\>.
\end{thm}

The isomorphism \,$\thil$ in~\Ref{thib} is the natural extension of
the isomorphism~\Ref{thi} denoted by the same letter.

\section{Space \,$\KTX\ox\C(h)$ \,as a module over \,$\Uen$}
\label{sec Equi}
\subsection{Action of \,$\Uen$ on \,$\KTX\ox\C(h)$\>}
\label{uenaction}
Recall
\vvn.2>
\be
K_T(\tfl)\,=\,\CGs\?\ox\Czh\;\big/\bigl\bra\:f(\GG)=f(\zz)\,,
\,\,f\in\C[\zz^{\pm1}]^{\>S_n}\bigr\ket
\vv-.3>
\ee
and
\vvn-.4>
\be
\KTX\,=\tbigoplus_{|\bla|=n}\<K_T(\tfl)\,.
\vv-.2>
\ee
Define the graded algebra monomorphism
\vvn.2>
\be
\mul:K_T(\tfl)\ox\C(h)\:\to\:\End\bigl(\CNnl\bigr)\ox\Czhr
\vvgood
\vv-.1>
\ee
by the rule
\vvn-.5>
\be
\mul\bigl([f(\GG\<,\zz,h)])\::\>\xi_I\,\mapsto\,f(\zz_I,\zz,h)\,\xi_I\,,
\qquad I\in\Il\,,\kern-2em
\vv.3>
\ee
for any \,$f\in\CGs\?\ox\Czhr$\>, cf.~\Ref{psi(g)}\:. Let
\vvn.2>
\be
\mu\::\KTX\ox\C(h)\:\to\:\End\bigl(\CNn\bigr)\ox\Czhr
\vv.1>
\ee
be the direct sum of the monomorphisms \,$\mul$\>.

\begin{lem}
\label{muS}
For any \,$f\in\CZSr$\>, we have
\,$\mu\bigl([\:1\ox f\:]\bigr)\:=\:\id\ox f$\>.
\qed
\end{lem}

\begin{thm}
\label{thmu}
We have the induced isomorphisms of graded algebras
\be
\mul:K_T(\tfl)\ox\C(h)\,\to\,\phol(\Bci)\ox\Czb\,,
\vv-.1>
\ee
\be
\mu\>:\KTX\ox\C(h)\,\to\,\pho(\Bci)\ox\Czb\,,\kern-.66em
\vv.1>
\ee
where the homomorphisms \,$\pho$ \:and \,$\phol\<$ are defined by
Lemma~\ref{pho} and formula~\Ref{phola}\:.
\end{thm}
\begin{proof}
The statement follows from Theorem~\ref{thm B-infty alg}.
\end{proof}

Consider the graded \,$\Czhr$-module isomorphism
\vvn.3>
\beq
\label{nu3}
\nu\::\KTX\ox\C(h)\,\to\,\DV\!\ox\Czhr\,,
\vv.2>
\eeq
that is the natural extension of the isomorphism~\Ref{nu2} denoted by the same
letter.

\begin{thm}
\label{thnu}
The \;$\Czhr$-module isomorphism \;$\nu:\<\KTX\ox\C(h)\to\DV\!\ox\Czhr$
\,and the \,$\Czhr$-algebra isomorphism
\vvn.1>
\;$\mu:\<\KTX\:\to\:\pho(\Bci)\ox\Czb$
\,identify the \,$\pho(\Bci)\ox\Czb$-module \,$\DV\!\ox\Czhr$ \,with
the regular representation of \,$\KTX\ox\C(h)$\>.
\qed
\end{thm}

\vsk.2>
Recall that the space \,$\DV\!\ox\Czhr$ \,is a \,$\Uen$\:-module,
see Corollary~\ref{DVmod}. Thus the isomorphism \Ref{nu3} makes
\,$\KTX\ox\C(h)$ \>into a \,$\Uen$\:-module. We denote
the \,$\Uen$\:-module structure on \,$\KTX\ox\C(h)$ by \,$\rho$\>.

\vsk.3>
Let \,$K_\bla=K_T(\tfl)\ox\C(h)$\>. For any $i=1\lc N$ and \,$f\in K_\bla$\>,
\vvn.3>
\beq
\label{Ltlf}
\Lt_{\ii,+0}\,f\,=\,h^{\la_1\lsym+\la_i}\>f\,,\qquad
\Lt_{\ii,-0}\,f\,=\,h^{\la_1\lsym+\la_{i-1}}f\,,
\vv.1>
\eeq
see~\Ref{Ltxi}\:. Thus \,$K_\bla$ are eigenspaces
for the action of the subalgebra \,$\Ueh\subset\Uen$\>.

\vsk.2>
The operators
\vvn.1>
\;$\rho\bigl(A_{p,\pm}(u)\bigr)$, \,$p=1\lc N$, preserve the subspaces
\,$K_\bla$ and act as follows:
\vvn-.3>
\begin{gather}
\label{Arho}
\rho\bigl(A_{p,+}(u)\bigr)\>:\>[f(\GG\<,\zz,h)]\;\mapsto\,
\>\Big[f(\GG\<,\zz,h)\;\prod_{a=1}^p\,\prod_{i=1}^{\la_a}\,
\frac{1-h^{-1}u/\gm_{\ai}}{1-u/\gm_{\ai}}\Big]\,,
\\[4pt]
\rho\bigl(A_{p,-}(u)\bigr)\>:\>[f(\GG\<,\zz,h)]\;\mapsto\,
\>\Big[f(\GG\<,\zz,h)\;\prod_{a=1}^p\,\prod_{i=1}^{\la_a}\,
\frac{1-h\gm_{\ai}/u}{1-\gm_{\ai}/u}\Big]\,,
\notag
\\[-15pt]
\notag
\end{gather}
for any \,$f\in\CGs\?\ox\Czhr$, see \Ref{Axi} and \Ref{nu2}\:.

\goodbreak
\vsk.3>
For \,$p=1\lc n-1$\>, let \,$\al_p=(0\lc0,1,\<-\>1,0\lc0)$,
with \,$p-1$ \>first zeros.

\begin{thm}
\label{rho}
We have
\vvn-.8>
\begin{gather*}
\rho\bigl(E_{p,\:\pm}(u)\bigr)\::\>K_{\bla\:-\al_p}\,\mapsto\,K_\bla\,,
\\[7pt]
\rho\bigl(E_{p,+}(u)\bigr)\::\>[f]\;\mapsto\,\biggl[\,-
\sum_{i=1}^{\la_p}\;\frac{f(\GG^{\ipi}\!,\zz,h)}{1-u/\gm_{\pci}}\;\,
\prod_{\satop{j=1}{j\ne i}}^{\la_p}\,\frac1{1-\gm_{\pci}/\gm_{\pcj}}\;
\prod_{k=1}^{\la_{p+1}\!}\,(1-h\gm_{p+1,k}/\gm_{\pci})\,\biggr]\,,
\\[4pt]
\rho\bigl(E_{p,-}(u)\bigr)\::\>[f]\;\mapsto\,\biggl[\;
\sum_{i=1}^{\la_p}\,f(\GG^{\ipi}\!,\zz,h)\;
\frac{\gm_{\pci}/u}{1-\gm_{\pci}/u}\;
\prod_{\satop{j=1}{j\ne i}}^{\la_p}\,\frac1{1-\gm_{\pci}/\gm_{\pcj}}\;
\prod_{k=1}^{\la_{p+1}\!}\,(1-h\gm_{p+1,k}/\gm_{\pci})\,\biggr]\,,
\end{gather*}
\vv-.5>
\begin{gather*}
\rho\bigl(F_{p,\:\pm}(u)\bigr)\::\>K_{\bla\:+\al_p}\,\mapsto\,K_\bla\,,
\\[7pt]
\begin{aligned}
\rho\bigl(F_{p,+}(u)\bigr)\::\>[f]\;\mapsto\,\biggl[\,-
\sum_{i=1}^{\la_{p+1}\!}{}&\,f(\GG^{\ip}\!,\zz,h)\;
\frac{u/\gm_{\poi}}{1-u/\gm_{\poi}}\,\times{}
\\
& \!{}\times{}\>
\prod_{\satop{j=1}{j\ne i}}^{\la_{p+1}\!}\,\frac1{1-\gm_{\poj}/\gm_{\poi}}\;
\prod_{k=1}^{\la_p}\,(1-h\gm_{\poi}/\gm_{p,k})\,\biggr]\,,
\end{aligned}
\\[4pt]
\rho\bigl(F_{p,-}(u)\bigr)\::\>[f]\;\mapsto\,\biggl[\;
\sum_{i=1}^{\la_{p+1}\!}\;\frac{f(\GG^{\ip}\!,\zz,h)}{1-\gm_{\poi}/u}\,\;
\prod_{\satop{j=1}{j\ne i}}^{\la_{p+1}\!}\,\frac1{1-\gm_{\poj}/\gm_{\poi}}\;
\prod_{k=1}^{\la_p}\,(1-h\gm_{\poi}/\gm_{p,k})\,\biggr]\,,
\\[-16pt]
\end{gather*}
where
\vvn-.5>
\begin{gather*}
\GG^{\ipi}=\,(\Gm_1\:\lsym;\Gm_{p-1}\:;\Gm_p-\{\gm_{\pci}\};
\Gm_{p+1}\cup\{\gm_{\pci}\};\Gm_{p+2}\:\lsym;\Gm_N)\,,
\\[8pt]
\GG^{\ip}=\,(\Gm_1\:\lsym;\Gm_{p-1}\:;\Gm_p\cup\{\gm_{\poi}\};
\Gm_{p+1}-\{\gm_{\poi}\};\Gm_{p+2}\:\lsym;\Gm_N)\,.
\\[-14pt]
\end{gather*}
\end{thm}
\begin{proof}
The statement follows from Corollary~\ref{isonu} and Theorem~\ref{AEFxi}.
\end{proof}

\begin{rem}
Notice that the expressions for \,$\rho\bigl(E_{p,\:\pm}(u)\bigr)$
\,coincide as rational functions, and the same is true for
\,$\rho\bigl(F_{p,\:\pm}(u)\bigr)$\>. These functions are expanded at \,$u=0$
\,for \,$\rho\bigl(E_{p,+}(u)\bigr)\>,\alb\,\rho\bigl(F_{p,+}(u)\bigr)$\>,
and at \,$u=\infty$ \,for
\,$\rho\bigl(E_{p,-}(u)\bigr)\>,\alb\,\rho\bigl(F_{p,-}(u)\bigr)$\>.
\end{rem}

\begin{rem}
The \,$\Uen$\:-action \,$\rho$ \>on \,$\KTX$ defined via the isomorphism
\vvn.1>
\Ref{nu3}\:, that is, with the help of the map \;$\Stab_{\>\id}$ introduced
\vvn.1>
in Section~\ref{sec:stab} is related to the \,$\Uqn$\:-action introduced
in \cite{GV,Vas1,Vas2} in terms of the convolution operators acting on
$\KTX$.
\vsk.2>
More precisely, let
\vvn-.8>
\be
Q_\bla(\GG\<,h)\,=\!\prod_{1\le a<b\le n}\,\prod_{i=1}^{\la_a}\,
\prod_{j=1}^{\la_b}\,(1-h\gm_{\:\bj}/\gm_{\ai})\,.
\vv-.1>
\ee
Denote by \,$\Kt_\bla$ \,the ideal in \,$K_\bla$ \,generated by
\,$[\:Q_\bla\:]$\,. Since \,$[\:Q_\bla\:]$ \,is not a zero divisor, the map
\,$\chi_\bla:K_\bla\<\to\Kt_\bla$\,, \,$[\:f\:]\:\mapsto\:[\:Q_\bla f\:]$
\,is an isomorphism of vector spaces.

\vsk.2>
Let \,$\Kt=\bigoplus_{|\bla|=n}\Kt_\bla$, and let \,$\chi:\KTX\to\Kt$ \,be
the direct sum of isomorphisms \,$\chi_\bla$\>. It is straightforward
to verify that \,$\Kt$ \,is a \,$\Uen$\:-submodule of \,$\KTX$\>.
This defines a new \,$\Uen$-action \,$\rho^+\?$ on \,$\KTX$ by the rule:
\vvn.1>
\,$\rho^+(X)=\chi^{-1}\rho(X)\,\chi$ \,for any \,$X\?\in\Uen$\>.
\vvn.1>
The $\Uen$-actions \,$\rho$ \,and \,$\rho^+\?$ on \,$\KTX$ are conjecturally
not isomorphic, but become isomorphic as actions on \,$\KTX\ox\C(\zz,h)$
\,since \,$[\:Q_\bla]$ \,is invertible in \,$K_\bla\ox\C(\zz,h)$\:.

\vsk.2>
Now the \,$\Uen$-action \,$\rho^+$ essentially coincides with
\vvn.08>
the \,$\Uqn$\:-action introduced in \cite{GV,Vas1,Vas2} in terms of
the convolution operators acting on $\KTX$, cf.~formulae in Theorem~\ref{rho}
and in \cite[page~287]{Vas2}.
\end{rem}

\subsection{Topological interpretation of the actions in Theorem \ref{rho}}
Consider the vectors
$\mub'=(\la_1\lc\la_p-1,\,1,\>\la_{p+1}\lc,\la_N)$ (if $\la_p>0$), and
$\mub''\!=(\la_1\lc\la_p,\,1,\>\la_{p+1}-1\lc\la_N)$ (if $\la_{p+1}>0$).
There are natural forgetful maps
\beq
\label{eqn:forgetfuls}
\F_{\bla\<\>-\al_p}\,\xlto{\pi'_1}\>\F_{\mub'}\xto{\pi'_2}\,\F_{\bla}\,,
\qquad\qquad\qquad
\F_{\bla\<\>+\al_p}\,\xlto{\pi''_1}\>\F_{\mub''}\xto{\pi''_2}\,\F_{\bla}\,.
\vv.4>
\eeq
The rank $\la_p-1$, $1$, $\la_{p+1}$ bundles over $\F_{\mub'}$ with fibers
$F_p/F_{p-1}$, $F_{p+1}/F_p$, $F_{p+2}/F_{p+1}$ will be respectively denoted by
$A',\>B',\>C'$. The rank $\la_p$, $1$, $\la_{p+1}-1$ bundles over $\F_{\mub''}$
with fibers $F_p/F_{p-1}$, $F_{p+1}/F_p$, $F_{p+2}/F_{p+1}$ will be
respectively denoted by $A'',\>B'',\>C''$.

\vsk.3>
For a $(\C^{\times})^n$-equivariant bundle $\xi$, let $e(\xi)$ be its equivariant K-theoretic
Euler class. We can make the extra \,$\C^\times$ (whose Chern root is $h$) act
on any bundle by fiberwise multiplication. The equivariant Euler class with this extra action will be denoted by $e_h(\xi)$.

\vsk.3>
Recall that an equivariant proper map $f:X\to Y$ induces the pullback $f^*:K_T(Y)\to K_T(X)$
and push-forward (also known as Gysin) $f_*:K_T(X)\to K_T(Y)$ maps.

\begin{thm}
\label{Varag}
The operators \,$\rho\bigl(E_{p,\pm}(u)\bigr)$, \,$\rho\bigl(F_{p,\pm}(u)\bigr)$,
are equal to the following topological operations
\vvn-.6>
\begin{gather*}
\rho\bigl(E_{p,+}(u)\bigr)\>:\>x\,\mapsto(-1)^{\lambda_p}\pi'_{2*}
\biggl(\<\pi_1^{'*}(x)\cdot\frac{1}{1-u/[B']}
\frac{e_{h}\bigl(\Hom(C',B')\bigr)[\Lambda^{\top}A']}{[B']^{\lambda_p-1}}\biggr)\,,
\\[6pt]
\rho\bigl(E_{p,-}(u)\bigr)\>:\>x\,\mapsto(-1)^{\lambda_p-1}\pi'_{2*}
\biggl(\<\pi_1^{'*}(x)\cdot\frac{[B']/u}{1-[B']/u}
\frac{e_{h}(\Hom\bigl(C',B')\bigr)[\Lambda^{\top}A']}{[B']^{\lambda_p-1}} \biggr)\,,
\end{gather*}
\vv-.3>
\begin{gather*}
\rho\bigl(F_{p,+}(u)\bigr)\>:\>x\,\mapsto(-1)^{\lambda_{p+1}}\pi''_{2*}
\biggl(\<\pi_1^{''*}(x)\cdot\frac{u/[B'']}{1-u/[B'']}
\frac{e_{h}\bigl(\Hom(B'',A'')\bigr)[B'']^{\lambda_{p+1}-1}}{[\Lambda^{\top}C'']}\biggr)\,,
\\[6pt]
\rho\bigl(F_{p,-}(u)\bigr)\>:\>x\,\mapsto(-1)^{\lambda_{p+1}-1}\pi''_{2*}
\biggl(\<\pi_1^{''*}(x)\cdot\frac{1}{1-[B'']/u}
\frac{e_{h}(\Hom\bigl(B'',A'')\bigr)[B'']^{\lambda_{p+1}-1}}{[\Lambda^{\top}C'']}\biggr)\,,
\end{gather*}
\end{thm}

\begin{proof}
If we write down equivariant localization formulae for the given topological operations we obtain the formulae of Theorem \ref{rho}.
\end{proof}

\section{Bethe algebra \,$\Bck$ \<and discrete Wronskian}
\label{sec:Wr}

\subsection{Wronski map}
\label{sec:Wron}

Throughout this section we use the following grading of functions
in the variables \,$\gm_{\ij}\>,\,z_i$\>, and \,$h$\>:
\,$\deg\:\gm_{\ij}=\:\deg\:z_i=\:1$ \,for all \,$i=1\lc N$\:,
\,$j=1\lc\la_i$\,, and \,$\deg\:h=0$\,, cf.~Section~\ref{sec:grad}.

\vsk.2>
Let \,$q_1\lc q_N$ \>be distinct nonzero complex numbers.
Set
\vvn.1>
\beq
\label{W}
\Wk(u)\,=\,\det\>\biggl(q_i^{\:i-j}\,\prod_{k=1}^{\la_i}\,
(1-h^{i-j}\gm_{\ik}/u)\<\biggr)_{i,j=1}^N\,.
\eeq
The function \,$\Wk(u)$ \>is essentially a discrete Wronskian (multiplicative
Casorati determinant) of functions
\vvn-.4>
\beq
\label{giu}
g_i(u)\,=\,q_i^{-\log u/\log h}\>
\prod_{k=1}^{\la_i}\,\bigl(1-h^{i-1}\gm_{\ik}/u)\,.\qquad i=1\lc N\:,\kern-2em
\vv-.3>
\eeq
namely,
\vvn-.7>
\be
\Wk(u)\,=\,\det\>\bigl(\:g_i(uh^{j-1})\bigr)_{i,j=1}^N
\;\prod_{i=1}^N\,q_i^{\:i-1+\log u/\log h}\,.
\ee

\vsk.2>
Let \,$\CGl$ be the algebra of polynomials in the variables \,$\gm_{\ij}$
\>invariant under the permutations of the variables with the same first
\vvn.1>
subscript. Define the elements \,$\elq_0(\GG\<,h)\lc\elq_N(\GG\<,h)\in\CGlh$
via the coefficients of \,$\Wk(u)$\>:
\be
\Wk(u)\,=\,\sum_{p=0}^N\,(-1)^p\,\elq_p(\GG)\,u^{-p}\!
\prod_{1\le i<j\le N}\!(1-q_j/q_i)\,.
\vv-.4>
\ee
In particular,
\vvn-.2>
\be
\elq_0(\GG\<,h)\>=\>1\,,\qquad
\elq_N(\GG,h)\>=\>\gm_{1,1}\ldots\gm_{N\?,\la_N}\>.
\vv.3>
\ee
Define the {\it Wronski map\/}
\beq
\label{Wq}
\Wrq_\bla:\Czps\to\,\CGlh
\vv.3>
\eeq
to be the algebra homomorphism such that
\vvn.2>
\be
e_p(\zz)\,\mapsto\,\elq_p(\GG\<,h)\,,\qquad p=1\lc N\>,
\vv.1>
\ee
where \,$e_p(\zz)$ \>is the \,$p$\:-th elementary symmetric polynomial
of \,$\zzz$\>. The map \,$\Wrq_\bla$ \>makes the space \,$\CGlh$ \,into
the graded \,$\Czps$\<-module.

\begin{lem}
\label{injW}
The map \;$\Wrq_\bla$ \:is injective.
\end{lem}
\begin{proof}
Observe that \,$\elq_p(\GG,1)=e_p(\gm_{1,1}\lc\gm_{N\?,\la_N})$ \,for all
\,$p=0\lc N$. The statement follows.
\end{proof}

\vsk.2>
The Wronski map \Ref{Wq} induces the monomorphism \,$\Czpshr\to\CGlhr$ \,that
will be also denoted \,$\Wrq_\bla$\>.

\subsection{Polynomial isomorphisms}
\label{sec:psq}

Recall the multiplicative shift operator \,$\tau$ acting on functions of
\,$u$\>: \,$\tau f(u)=f(hu)$. Let
\vvn-.3>
\beq
\label{De}
\De\,=\,1+\sum_{p=1}^N\,\bk_p(u)\,(\<-\:\tau)^p\>.
\eeq
be the \,$N\?$-th order linear difference operator \,annihilating the functions
\,$g_1(u/h)\lc g_N(u/h)$\>. Its coefficients \,$\bk_1(u)\lc\bk_N(u)$ \,are
described below, see formula~\Ref{bkp}\:. We set \,$\bk_0(u)=1$.

\vsk.2>
Let \,$x$ \,be a complex variable. Set
\vvn-.3>
\beq
\label{Whk}
\Whk(u,x)\,=\,\det\>\biggl(q_i^{\:i-j}\,\prod_{k=1}^{\la_i}\,
(1-h^{i-j}\gm_{\ik}/u)\<\biggr)_{i,j=0}^N\,,
\vv-.4>
\eeq
where \,$q_0\<=x^{-1}$ and \,$\la_0=\:0$\>. Then
\vvn.1>
\beq
\label{bkp}
\frac{\Whk(u,x)}{\Wk(u)}\,=\,1+\sum_{p=1}^N\,\bk_p(u)\,(\<-\:x)^p\>,
\eeq
and the coefficients \,$\bk_p(u)$ \,have the following expansion
at \,$u=\infty$\>:
\beq
\label{bps-}
\bk_p(u)\,=\,e_p(\kkk)\>+\:\sum_{s=1}^\infty\,\bk_{p,-s}\:u^{-s}\:,
\vv.1>
\eeq
where \,$e_p(\kkk)$ \>is the \,$p$\:-th elementary symmetric polynomial.
Notice that
\vvn.3>
\beq
\label{bkN}
\bk_N(u)\,=\,q_1\ldots q_N\,\frac{\Wk(u/h)}{\Wk(u)}\;.
\vv.1>
\eeq

\begin{prop}
\label{bps}
The elements \;$\bk_{p,-s}$\>, \,$p=1\lc N$, \,$s\in\Z_{>0}$\>,
together with \,$\C(h)$ \,generate the algebra \,$\CGlhr$\>.
\end{prop}
\begin{proof}
Recall the functions \,$g_1\lc g_N$, see~\Ref{giu}\:, and the difference
operator \Ref{De}\:. The equation \,$\De\:g_i(u/h)=\:0$ \,yields
\beq
\label{sum0}
\sum_{p=0}^N\,(\<-\:q_i)^{-p}\;\bk_p(u)\,
\prod_{k=1}^{\la_i}\,(1-h^{i-p}\gm_{\ik}/u)\,=\,0\,.
\eeq
Let \,$e_j(\Gm_{\<i})$ \,be the \,$j$\:-th elementary symmetric polynomial of
\,$\gm_{i,\:1}\lc\gm_{i,\la_i}$\>. Collecting the coefficient of \,$u^{-j}$
\>in \Ref{sum0}\:, we get that \,$e_j(\Gm_{\<i})\,=\,C_j$\,, where \,$C_j$
\,is expressed via the elements \,$\bk_{k,\:s}$ \,for \,$k\le j$\>, integer
powers of \,$h$\>, and \,$e_m(\Gm_{\<i})$ \>for \,$m<j$\>. The proposition
follows.
\end{proof}

Recall the subalgebras \,$\Bck_-$ and \,$\Zc_-$ defined in
Section~\ref{pmalg}, and the homomorphism \,$\phol$ \>given
by~\Ref{phola}\:. Corollary~\ref{Zpm} yields
\vvn.2>
\be
\pho_\la(\Zc_-)\>=\,\Czps\?\ox\C(h)\subset\Elzphr\,.
\ee

\begin{thm}
\label{psq}
The assignment \;$\bk_{p,-s}\mapsto\phol(\Bk_{p,-s})$\,, \,$p=1\lc N$,
\,$s\in\Z_{>0}$\>, \,defines an isomorphism
\vvn-.2>
\beq
\label{psql}
\psql:\:\CGlhr\,\to\,\phol(\Bck_-)
\vv.4>
\eeq
of graded \;$\Czpshr$-algebras. Here the algebra \;$\Czpshr$ \>acts
\vvn.07>
on \;$\CGlhr$ via the Wronski map \;$\Wrq_\bla$.
\end{thm}
\begin{proof}
The proof of Theorem~\ref{psq} is similar to the analogous proofs of
\cite[Theorem~6.3\:]{MTV2} and \cite[Theorem~5.2, item~(i)\:]{MTV3}\:.
The proof is based on the Bethe ansatz technique.
The details will be published elsewhere.
\end{proof}

Recall the spaces \,$\DiV_\bla\>,\,\DLl$ defined in Sections~\ref{secDV},
\ref{polV}, and the isomorphism \,$\thitl$\>, see Lemma~\ref{that}.
Keeping the same notation, we will think of \,$\thitl$ as the isomorphism
\,$\CGlhr\to\DLl\?\ox\C(h)$\>, where we identified \,$\CGl$ with \,$\Czpl$ by
\vvn.1>
the rule \,$f(\GG)\mapsto f(\zz_\Imin)$ for any $f\in\CGl$, and replaced
\,$\Chh$ \,by \,$\C(h)$\>.

\vsk.3>
Define the homomorphism
\vvn.1>
\begin{gather}
\label{thiq}
\thtql:\:\CGlhr\,\to\,\DiV_\bla\ox\C(h)\,,
\\[4pt]
\notag
f\,\mapsto\>\sum_{I\in\Il}\,\frac1{R(\zz_I)}\;
\prod_{a=1}^N\>\prod_{\>i\in I_a}z_i^{\la^{(a)}\?-\:n}\;
\psql(f)\,\xi_I\,,
\end{gather}
cf.~\Ref{thi-xi} and~\Ref{thitla}\:. The map \;$\thtql$ \>is graded.

\begin{lem}
We have \;$\thtql\bigl(\CGlhr\bigr)\subset\DLl\ox\C(h)$\,.
\end{lem}
\begin{proof}
Let \,$f\<\in\CGlhr$\>. Then \,$\thtql(f)=\psql(f)\,\thitl(1)$\>.
\vvn.16>
By Theorem~\ref{psq} and Lemma \ref{lem ac comS}, \;$\psql(f)\in\Elzphr$\>,
and \,$\psql(f)$ \>commutes with the \,$S_n$-action \Ref{Sn-}\:.
Since \,$\thitl(1)\in\DLl$, the statement follows.
\end{proof}

\begin{thm}
\label{thmthiq}
The map \;$\thtql\<:\CGlhr\to\DLl\?\ox\C(h)$ \>is
\vvn.1>
an isomorphism of graded \;$\Czpshr$-modules, where \;$\Czpshr$
\vvn.08>
acts on \;$\CGlhr$ via the Wronski map \;$\Wrq_\bla$.
\end{thm}
\begin{proof}
Notice that the graded components of the spaces \,$\Czpshr$\>, \,$\CGlhr$\>,
\,$\DLl\?\ox\C(h)$ \,are finite-dimensional \,$\C(h)$-modules.

\vsk.1>
By Theorem~\ref{psq}\:, the map \,$\thtql$ is a homomorphism of graded
\vvn.1>
\;$\Czpshr$-modules. Its kernel is an ideal in \,${\CGlhr}$ \,that has zero
intersection with \,$\Wrq_\bla\bigl(\Czpshr\bigr)$\>, hence, it is the zero
ideal. Thus the composition \,$(\thitl)^{-1}\>\thtql:\CGlhr\to\CGlhr$ \,is
an injective graded \,$\C(h)$-endomorphism, that sends \,$1$ \>to \,$1$\>,
since \,$\thtql(1)=\thitl(1)$\>. Therefore, both \,$(\thitl)^{-1}\>\thtql$
\,and \,$\thtql$ \,are bijections.
\end{proof}

\begin{cor}
\label{regrepq}
The \;${\Czpshr}$-module
\vvn.1>
isomorphism \,$\thtql\<:{\CGlhr}\:\to\,\DLl\?\ox\C(h)$
\,and the \,$\Czpshr$-algebra isomorphism
\;$\psql\<:\CGlhr\:\to\:\phol(\Bck_-)$
\>identify the \,$\phol(\Bck_-)$-module \,$\DLl\?\ox\C(h)$ \,with
\vvn.1>
the regular representation of \,\,$\CGlhr$\>. Here \,$\Czpshr$ acts
on \;$\CGlhr$ via the Wronski map \;$\Wrq_\bla$.
\qed
\end{cor}

\subsection{K-theoretic isomorphisms}
\label{sec:muq}
Let \,$\Dlq=\prod_{\:1\le i<j\le N}\,(1-q_j/q_i)$\,. Define the algebra
\vvn-.5>
\beq
\label{Kql}
\Kcql\>=\,\C[\:\GG^{\pm1}]^{\>S_\bla}\?\ox\Czh\:\Big/\?\Bigl\bra \Wk(u)\>=
\,\Dlq\,\prod_{a=1}^n\,(1-z_a/u)\Bigr\ket\,.\kern-1em
\vv-.3>
\eeq
Let \;$y=z_1\ldots z_n$\>. Since in the algebra \,$\Kcql$\>,
\be
\vvn-.1>
y\,=\,\frac{\Dl(q_1h^{\la_1}\?\lc q_N\:h^{\la_N})}{\Dlq}\;
\prod_{i=1}^N\,\prod_{j=1}^{\la_i}\,\gm_{\ij}\,,
\vv-.7>
\ee
we have
\vvn-.1>
\beq
\label{Kqly}
\Kcql\ox\C(h)\>=\>\CGl\?\ox\C[\zz,y^{-1}\:]\ox\C(h)\Big/\?\Bigl\bra \Wk(u)\>=
\,\Dlq\,\prod_{a=1}^n\,(1-z_a/u)\Bigr\ket\,.\kern-1em
\eeq

\begin{example}
Let \,$N=n=2$ \,and \,$\bla=(1,1)$\>.
Then
\be
\Wk(u)\,=\,\det\:\left(
\begin{array}{cccc}
1-\gm_{1,1}/u & q_1^{-1}(1-h^{-1}\gm_{1,1}/u)
\\[3pt]
q_2\:(1-h\gm_{2,1}/u) & 1-\gm_{2,1}/u
\end{array}
\right)\,,
\vv-.2>
\ee
and the relations in \,$\Kcql$ \>are
\vvn.1>
\beq
\label{Hq11}
\gm_{1,1}\frac{q_1\?-h^{-1}q_2}{q_1\?-q_2} +
\gm_{2,1}\frac{q_1\?-hq_2}{q_1\?-q_2}\>=\,z_1+z_2\,,\kern3em
\gm_{1,1}\>\gm_{2,1}\,=\,z_1\:z_2\,.
\eeq
\end{example}

\vsk.2>
It is easy to see that the algebra \,$\Kcql$ does not change if all \,$\kkk$
are multiplied simultaneously by the same number. Notice that in the limit
$q_{i+1}/q_i\to 0$ \,for all \,$i=1\lc N-1$, the relations in \,$\Kcql$
turn into the relations in \,$K_T(\tfl)$, see \Ref{Hrel}\:.

\begin{thm}
\label{thmuq}
The isomorphism \,\Ref{psql} induces the isomorphism
\vvn.2>
\beq
\label{muql}
\muql:\:\Kcql\ox\C(h)\,\to\,\phol(\Bck)\ox\Czb
\vv.2>
\eeq
of graded \;$\Czhr$-algebras.
\end{thm}
\begin{proof}
The statement follows from formulae~\Ref{Kqly} and \Ref{Bky}\:.
\end{proof}

Expand the coefficients \,$\bk_1(u)\lc\bk_N(u)$ \,of the difference operator
\,$\De$\>, see~\Ref{De}\:, at \,$u=0$\>:
\beq
\label{bps+}
\bk_p(u)\,=\,\Bigl(e_p(q_1h^{\la_1}\?\lc q_N\:h^{\la_N})\>+\:
\sum_{s=1}^\infty\,\bk_{\ps}\:u^s\Bigr)\>(-u)^{-n}\,
\prod_{i=1}^N\,\prod_{j=1}^{\la_i}\,\gm_{\ij}\,,
\eeq
where \,$e_p(q_1\:h^{\la_1}\lc q_N\:h^{\la_N})$ \>is the \,$p$\:-th elementary
symmetric polynomial.

\begin{cor}
\label{rhoq}
We have \;$\muql\bigl([\:\bk_{p,\:\pm s}\:]\bigr)=\phol(\Bk_{p,\:\pm s})$
\,for all \,$p=1\lc N$, \,$s\in\Z_{>0}$\>.
\qed
\end{cor}

\begin{lem}
\label{bpsK}
The elements \;$[\:\bk_{p,\:\pm s}\:]$\>, \,$p=1\lc N$, \,$s\in\Z_{>0}$\>,
together with \,$\C(h)$ \,generate the algebra \,$\Kcql\ox\C(h)$\>.
\end{lem}
\begin{proof}
By formula~\Ref{bkN}\:, \,${[\:\bk_N(u)\:]\,=\,q_1\ldots q_N\:h^n
\bigl[\,\prod_{a=1}^n\,(1-h^{-1}u/z_a)\big/(1-u/z_a)\:\bigr]}$\>.
\vvn.2>
Hence, the ideal generated by the elements \,$\bk_{N,-s}$ \>for all
${s\in\Z_{>0}}$ and \,$\C(h)$ \,contains all symmetric polynomials in
\,$z_1^{-1}\?\lc z_n^{-1}$\:, in particular, \,$z_1^{-1}\?\ldots z_n^{-1}$\:.
Therefore, the statement follows from equality~\Ref{Kqly} and
Proposition~\ref{bps}.
\end{proof}

Lemma~\ref{bpsK} implies that statement of Corollary~\ref{rhoq} can serve as
a definition of the isomorphism \,$\muql$\>.

\begin{thm}
\label{thnuq}
The map
\vvn-.3>
\beq
\label{nuql}
\nuql:\:\Kcql\ox\C(h)\,\to\,\DiV_\bla\ox\C(h)\,,\qquad
f\,\mapsto\>\sum_{I\in\Il}\,\frac1{R(\zz_I)}\;\muql(f)\,\xi_I\,,
\vv-.1>
\eeq
cf.~\Ref{thiq}\:, \,is an isomorphism of graded \;$\CZSr$-modules.
\end{thm}
\begin{proof}
Let \,$r=\bigl[\,\prod_{a=1}^N\>\prod_{\>i\in I_a}z_i^{n-\la^{(a)}}\>\bigr]
\in\Kcql$\,.
The isomorphism \,$\thtql\<:\CGlhr\to\DLl\?\ox\C(h)$\>,
\vvn.2>
see Theorem~\ref{thmthiq}, induces the isomorphism
\,$\nutql:\:\Kcql\ox\C(h)\,\to\,\DiV_\bla\ox\C(h)$ \,such that
\,$\nuql(f)=\nutql(fr)$\>. Since the element \,$r$ \,is invertible,
the statement follows.
\end{proof}

\begin{cor}
\label{regrepKq}
The \;${\Czhr}$-module
\vvn.1>
isomorphism \,$\nuql\<:\Kcql\ox\C(h)\to\DL\?\ox\Czb\ox\C(h)$
\vvn.16>
\,and the \,${\Czhr}$-algebra isomorphism
\;$\muql\<:\Kcql\ox\C(h)\:\to\:\phol(\Bck)\ox\Czb$
\vvn.1>
\>identify the \,$\phol(\Bck)\ox\Czb$-module \,$\DL\?\ox\Czb\ox\C(h)$ \,with
the regular representation of \,\,$\Kcql\ox\C(h)$\>.
\end{cor}
\begin{proof}
The statement follows from Theorems~\ref{thmuq} and~\ref{thnuq}.
\end{proof}

Recall \;$\KTX=\bigoplus_{|\bla|=n}K_T(\tfl)$\>, see~\Ref{KTX}\:. Let
\vvn.2>
\be
\nul:K_T(\tfl)\ox\C(h)\,\to\,\DL\?\ox\Czhr
\vv.3>
\ee
be the graded \,$\Czhr$-module isomorphism obtained by the restriction
of the isomorphism~\Ref{nu3}\:. By Theorems~\ref{thnu} and~\ref{thnuq},
the composition \,$\btl=(\nuql)^{-1}\:\nul$\,,
\vvn.3>
\beq
\label{btl}
\btl:K_T(\tfl)\ox\C(h)\,\to\,\Kcql\ox\C(h)
\vv.3>
\eeq
is a graded \,$\Czhr$-module isomorphism.
Notice \;$\btl(1)=1$ \,since \,$\nuql(1)=\nul(1)$\,.

\vsk.2>
Recall the \,$\Uen$\:-action \,$\rho$ \,on \,$\KTX\ox\C(h)$,
see~formulae~\Ref{Arho} and Theorem~\ref{rho}. Let
\vvn-.3>
\be
\rhol:\Uen^\hg\:\to\,\End\bigl(K_T(\tfl)\bigr)\ox\C(h)
\vv.2>
\ee
be the induced \,$\C(h)$-algebra homomorphism. Recall the generators
\,$\bk_{p,\:\pm s}$ \:of the algebra \,$\Kcql$\>,
see~\Ref{bps-}\:, \Ref{bps+}\:.

\begin{lem}
\label{lemall}
The assignment \;$[\:\bk_{p,\:\pm s}\:]\mapsto\rhol(\Bk_{p,\:\pm s})$\,,
\vvn.1>
\,$p=1\lc N$, \,$s\in\Z_{>0}$\>, \,defines a graded \;$\CZSr$-algebra
isomorphism \;$\all\<:\Kcql\ox\C(h)\:\to\:\rhol(\Bck)$\>.
\end{lem}
\begin{proof}
By the definition of \,$\rhol$\,, we have
\vvn.1>
\,$\rhol(X)=(\nul)^{-1}\>\rhol(X)\;\nul$ \,for any \,$X\?\in\Uen^\hg$\:.
Thus Corollary~\ref{rhoq} and Lemma~\ref{bpsK} imply that
\vvn.2>
\beq
\label{all}
\all(f)\,=\,(\nul)^{-1}\>\muql(f)\;\nul
\vv.2>
\eeq
for any \>$f\in\Kcql$\>, and Lemma~\ref{lemall} follows from
Theorems~\ref{thnu} and~\ref{thmuq}.
\end{proof}

\begin{cor}
\label{regKqK}
The \;${\Czhr}$-module
\vvn.16>
isomorphism \,$\btl^{-1}\<:\Kcql\ox\C(h)\to K_T(\tfl)\ox\C(h)$
\vvn.16>
\,and the \,${\Czhr}$-algebra isomorphism
\;$\all\<:\Kcql\ox\C(h)\:\to\:\rhol(\Bck)\ox\Czb$
\vvn.1>
\>identify the \,$\rhol(\Bck)\ox\Czb$-module \,$K_T(\tfl)\ox\C(h)$ \,with
the regular representation of \,\,$\Kcql\ox\C(h)$\>.
\qed
\end{cor}

\subsection{New multiplication on \,$K_T(\tfl)\ox\C(h)$\>}
\label{sec:*}

Define a new commutative associative multiplication \,$\asq$ \:on
\,$K_T(\tfl)\ox\C(h)$\>, depending on the parameters \,$\kkk$\:, by the rule:
\vvn-.3>
\beq
\label{*}
\btl(f\asq\?g)\,=\,\btl(f)\,\btl(g)
\eeq
for any \,$f,g\in K_T(\tfl)\ox\C(h)$\>.

\begin{lem}
\label{allbtl}
For any \,${f\<\in K_T(\tfl)\ox\C(h)}$\>, the operator \;$f\asq$ \:coincides
with the operator \,$\all\bigl(\btl(f)\bigr)\in\rhol(\Bck)$\>. The map
\vvn.2>
\be
K_T(\tfl)\ox\C(h)\>\to\>\rhol(\Bck)\,,\qquad f\>\mapsto f\asq\,,
\vv.2>
\ee
is an isomorphism of graded \;$\Czhr$-modules.
\end{lem}
\begin{proof}
Using the equality \,$\nuql=\btl\,\nul$\>, formula~\Ref{all}\:, and
Corollary~\ref{regrepKq}, we have
\vvn.3>
\begin{align*}
\all\bigl(\btl(f)\bigr)\,g\> &{}=\>
(\nul)^{-1}\bigl(\muql\bigl(\btl(f)\bigr)\,\nul(g)\bigr)\>=\>
(\nul)^{-1}\bigl(\muql\bigl(\btl(f)\bigr)\,\nuql\bigl(\btl(g)\bigr)\bigr)
\\[4pt]
& {}=\>(\nul)^{-1}\bigl(\nuql\bigl(\btl(f)\>\btl(g)\bigr)\bigr)\>=\>
(\btl)^{-1}\bigl(\btl(f)\>\btl(g)\bigr)\>=\>f\asq g\,.
\\[-16pt]
\end{align*}
Since both \,$\all$ \>and \,$\btl$ \>are graded \,$\Czhr$-module isomorphisms,
the statement follows.
\end{proof}

\subsection{Conjecture on the quantum equivariant K-theory}
\label{sec conjectures}

The quantum deformation of the equivariant K-theory algebra was introduced
by Givental and Lee, motivated, in particular, by a study of the relationship
between Gromov-Witten theory and integrable systems, see~\cite{G,GL}.
The quantum multiplication on \,$K_T(\tfl)$ depends on new parameters
\,$q_2/q_1\lc q_N/q_{N-1}$ \,and tends to the ordinary multiplication on
\,$K_T(\tfl)$ as all of these ratios tend to zero.

\begin{conj}
\label{conj}
The multiplication \Ref{*} on \,$K_T(\tfl)\ox\C(h)$ coincides with the quantum
multiplication.
\end{conj}

This conjecture is the K-theoretic analog of the main theorem in \cite{MO} that
describes the quantum multiplication in the equivariant cohomology of Nakajima
varieties. The case of the quantum multiplication in the equivariant
cohomologies of the varieties \,$\tfl$ was considered also in
\cite{GRTV,RTV,TV4}\:.

\vsk.2>
Lemmas~\ref{lemall} and~\ref{allbtl} mean that modulo Conjecture~\ref{conj},
the quantum equivariant K-theory algebra \;$QK_T(\tfl)\ox\C(h)$ \,is isomorphic to
the Bethe algebra \,$\rhol(\Bck)$ and is isomorphic to the algebra
\,$\Kcql\ox\C(h)$\>.

\subsection{Limit $h\to 0$}
\label{Limit h to 0}

Let \,$M^Q(u)$ be the \,$N\>{\times}\>N$ \>matrix with entries
\,$M^Q_{\ij}\<=0$ \,for \,$|\<\>i-j\<\>|>1$\>,
\begin{gather*}
M^Q_{\ii}(u)\,=\,\prod_{k=1}^{\la_i}\,(1-\gm_{\ik}/u)\,,\qquad i=1\lc N\>,
\\[6pt]
M^Q_{\ioi}(u)\,=\,Q_{i+1}/Q_i\,,
\ \quad M^Q_{\ii+1}(u)\,=\,(-u)^{-\la_i}\gm_{i,1}\<\dots\gm_{i,\:\la_i}\,,
\qquad i=1\lc N-1\,,
\\[-12pt]
\end{gather*}
depending on parameters \,$Q_1\lc Q_N$. Set \,$\Wt^Q(u)=\det\:M^Q(u)$\,.

\begin{lem}
Let \>$\la_i>0$ \>and \;$q_i=\:Q_i\>h^{\la_1\lsym+\la_{i-1}}$
for all \,$i=1\lc N$. Then
\vvn.2>
\be
\Wk(u)\,\to\,\Wt^Q(u)
\vv-.5>
\ee
as \,$h\to 0$.
\end{lem}
\begin{proof}
The proof is similar to that of \cite[Theorem~7.3]{GRTV}\>.
\end{proof}

Therefore, in the limit \,$h\to 0$ \>such that
\;$q_i\>h^{-\la_1\lsym-\la_{i-1}}$ are fixed for all \,$i=1\lc N$,
the algebra \,$\Kcql$ \>turns into the algebra
\be
\Kcqt_\bla\>=\,\C[\:\GG^{\pm1}]^{\>S_\bla}\?\ox\Czb\>\Big/
\Bigl\bra\Wt^Q(u)\>=\>\prod_{a=1}^n\,(1-z_a/u)\Bigr\ket\,.
\ee
For example, if \,$N=n=2$ \,and \,$\bla=(1,1)$, \,then
\be
\Wt^Q(u)\,=\,\det\:\left(
\begin{array}{cccc}
1-\gm_{1,1}/u & -\gm_{1,1}/u
\\[3pt]
Q_2/Q_1 & 1-\gm_{2,1}/u
\end{array}
\right)
\vv-.4>
\ee
and the relations in \,$\Kcqt_\bla$ \>are
\vvn.3>
\be
\gm_{1,1}\,(1-Q_2/Q_1)+\gm_{2,1}\>=\,z_1+z_2\,,\qquad
\gm_{1,1}\>\gm_{2,1}\,=\,z_1\:z_2\,,
\vv-.2>
\ee
cf.~\Ref{Hq11}

\begin{conj}
The quantum equivariant K-theory algebra
$\;QK_{(\C^\times)^n}(\Fla)$ is isomorphic to the algebra $\Kcqt_\bla$\,.
\end{conj}

This conjecture is the K-theoretic version of the observation in \cite[Theorem 7.13]{GRTV}
on how to obtain the presentation of the quantum equivariant cohomology
$QK_{GL_n}(\Fla)$ in \cite[Formula (2.22)] {AS}
from the presentation of the quantum equivariant cohomology
$QK_{GL_n\times\C^\times}(\tfl)$ in \cite[Theorem 7.10]{GRTV}.

\section{Appendix 1. Weight functions specialize to Grothendieck polynomials}

\def\gr{Gro\-then\-dieck }

Let $\al=(\al_1, \al_2,\dots)$ and $\beta = (\beta_1,\beta_2,\dots)$ be two
sequences of variables.
Double \gr polynomials $\G_w(\al;\beta)$ were introduced by Lascoux and
Schut\-zen\-ber\-ger in \cite{LS} by the following recursion.
\begin{itemize}
\item{} For the longest permutation $\si_0 \in S_n$, define
\beq \label{eqn:grot_initial}
\G_{\si_0}=\prod_{i+j\leq n} \left(1-{\beta_i}/{\al_j}\right).
\eeq
\item{} Let $s_i$ be the $i$-th elementary transposition. If the length of
$\si s_i$ is larger than the length of $\si$, then
\beq \label{eq:Grec}
\G_\si=\pi_{\al_i,\al_{i+1}}(\G_{\si s_i}),
\eeq
where $\pi_{\al_i,\al_{i+1}}$ is the trigonometric difference operator,
see \Ref{diff opers}\:.
\end{itemize}

\noindent Here is the list of double \gr polynomials for $\si\in S_3$:
\begin{align*}
& \G_{321}=\left( 1 -{\beta_1}/{\al_1}\right)
\left( 1 -{\beta_2}/{\al_1}\right) \left( 1 -{\beta_1}/{\al_2}\right),
\qquad
\G_{231}=\left( 1 -{\beta_1}/{\al_1}\right)
\left( 1 -{\beta_1}/{\al_2}\right),
\\
&
\G_{312}=\left( 1 -{\beta_1}/{\al_1}\right)
\left( 1 -{\beta_2}/{\al_1}\right), \qquad
\G_{213}=1 -{\beta_1}/{\al_1},
\\
&
\G_{132}= 1 -{\beta_1\beta_2}/{(\al_1\al_2)}, \qquad \G_{123}=1\,.
\end{align*}
Note that the usual choice of variables is $x_i=1-1/\al_i$ and $y_i=1-\beta_i$, see for example \cite{LS}.
In those variables the $\G_w$'s are indeed { polynomials}.

\vsk.2>
Consider the substitution $t^{(k)}_a=u_a$, $h=0$ into the weight function
$W_I$, and denote it by $\Wbar_I$.

For
\be
I=
(\{i^{(1)}_1<\ldots<i^{(1)}_{\la_1}\},
\{i^{(2)}_1<\ldots<i^{(2)}_{\la_2}\},
\ldots,
\{i^{(N)}_1<\ldots<i^{(N)}_{\la_N}\})
\ee
in $\Il$, define $\si_I\in S_n$ to be the permutation that maps the ordered list
$n,n-1\lc 1$ to the ordered list
\be
i^{(1)}_{\la_1}, i^{(1)}_{\la_1-1}\lc i^{(1)}_1,
i^{(2)}_{\la_2}, i^{(2)}_{\la_2-1}\lc i^{(2)}_1,
\ldots,
i^{(N)}_{\la_N}, i^{(N)}_{\la_N-1}\lc i^{(N)}_1.
\ee
Observe that $\si_I$ and $\si_J$ can belong to the same
group $S_n$ even if the their $\bla$'s are different as long as their corresponding $n$'s
are the same. For example $\si_{(\{1\},\{3\},\{2\})}=\si_{(\{1\},\{2,3\})}=231$.

\begin{thm}
\label{thm Gr h=0}
We have
\be
\Wbar_I=\G_{\si_I}( z_n^{-1}, z_{n-1}^{-1},\ldots ; u_1^{-1},u_2^{-1}, \ldots).
\ee
\end{thm}

\begin{proof}
For $I\?\in\Il$ let
\be
\Ic=
(\{i^{(1)}_{\la_1} \},
\{i^{(1)}_{\la_1-1} \},
\ldots,
\{i^{(1)}_{1} \},
\{i^{(2)}_{\la_2} \},
\{i^{(2)}_{\la_2-1} \},
\ldots,
\{i^{(2)}_{1} \},
\ldots,
\{i^{(N)}_{\la_N} \},
\{i^{(N)}_{\la_N-1} \},
\ldots,
\{i^{(N)}_{1} \})
\ee
in $\I_{(1,1\lc 1)}$. We have $\si_I=\si_{\Ic}$ and $\Wbar_I=\Wbar_{\Ic}$.
The first claim is obvious. Although the functions $W_I$ are $W_{\Ic}$ are rather different
(e.g. they depend on different sets of variables) it can be seen from the definition of the weight
functions that after the substitutions indicated by $W \mapsto \Wbar$ they are equal.
Hence it is enough to prove the theorem for $\la=(1,1\lc 1)\in\N^n$.

For this special case we will show that initial conditions and recursions for the
two sides of the statement agree. Indeed we have
$$\Wbar_{\{1\},\{2\}\lc\{n\}}=\prod_{j=1}^{n-1}\prod_{i=j+1}^{n} (1-z_i/u_j),$$
and
$\G_{n,n-1\lc 1}( z_n^{-1}, z_{n-1}^{-1},\ldots ; u_1^{-1},u_2^{-1}, \ldots)$ is the
same expression because of
\Ref{eqn:grot_initial}\:. Relations \Ref{k<lc}\:, \Ref{k<lpi} for $\Wbar$ read
$$\Wbar_{s_{a,a+1}(I)}=\pi_{z_a,z_{a+1}}\Wbar_I$$
if $I_a<I_{a+1}$. After the variable change given in the theorem this is equivalent
to the recursion \Ref{eq:Grec} for Grothendieck polynomials.
\end{proof}

\begin{cor}
The $t^{(k)}_a=u_a$, $h=0$ substitution into the formulae \Ref{WI} and \Ref{Res formula} give equivariant
localization and iterated residue expressions for the double Grothendieck polynomial $\G_{\si_I}( z_n^{-1}, z_{n-1}^{-1},\ldots ; u_1^{-1},u_2^{-1}, \ldots)$.
\end{cor}

Iterated residue formulae for stable \gr polynomials, as well as their applications to stability and positivity results
for quiver polynomials and Thom polynomials are explored in \cite{RSz, Al, AR}.

It is well known in Schubert calculus, see for example \cite{LS, Bu}, that the
class in \,$K_{(\C^\times)^n}(\Fla)$ \,of the structure sheaf of a Schubert
variety is represented by a double \gr polynomial. In our \gr polynomial and index conventions the class of the
Schubert variety $\overline{\Om}_{\id,I}$ is represented by the polynomial
$\G_{\sigma_I}(z_n^{-1},z_{n-1}^{-1},\ldots; \gamma_{1,1}^{-1},\gamma_{1,2}^{-1},\ldots)\cdot P_{\id, I}^{-1}$.
Denote this class by $\O_I$ and hence from Theorem \ref{thm Gr h=0} we obtain the following.

\begin{cor} \label{ap1eq}
We have
\[
\O_I=[\overline{\Om}_{\id,I}]=
[\G_{\sigma_I}(z_n^{-1},z_{n-1}^{-1},\ldots; \gamma_{1,1}^{-1},\gamma_{1,2}^{-1},\ldots)]\cdot P_{\id, I}^{-1}=
[W_I(\Gamma,\zz,0)]\cdot P_{\id, I}^{-1} \in K_{(\C^\times)^n}(\Fla).
\]
\end{cor}

\section{Appendix 2. Interpolation definition of K-theory classes of Schubert varieties}

In this appendix we give a new axiomatic definition of the classes $\O_I P_{\id,I}$, that is, the classes of \gr polynomials.
Note that the polarization $P_{\id,I}$ is a monomial, an invertible element of $\C[\zz^{\pm1}]$.

Let $f_1$ and $f_2$ be Laurent polynomials is the $\zz$-variables, with $\phi_I(N(f_1))=[0,m_1]$, $\phi_I(N(f_2))=[0,m_2]$. We write $f_1\prec_I f_2$ if $m_1<m_2$, cf.~the proof of Lemma \ref{lem:W_small}.
For the purpose of this section define $f$ to be $I$-small, if
\[
f \prec_I \prod_{k<l}\prod_{a\in I_k}\prod_{\satop{b\in I_l}{b>a}} (1-z_b/z_a).
\]

\begin{thm} \label{thm:Oaxiomatic}
The classes $\O_I P_{\id,I}$, that is, the images of the Grothendieck polynomials in $K_T(F_\bla)$, are uniquely determined by the properties
\begin{enumerate}
\item \label{a2uno} $\O_I|_{x_I}\cdot P_{\id,I}= \prod_{k<l}\prod_{a\in I_k}\prod_{\satop{b\in I_l}{b>a}} (1-z_b/z_a)$,
\item \label{a2dos} $\O_I|_{x_J}\cdot P_{\id,I}$ is $J$-small if $J\not= I$.
\end{enumerate}
\end{thm}

\begin{proof} The uniqueness proof is the obvious modification of the uniqueness proof in Section~\ref{sec:uni}. To prove existence we need to show that $\O_IP_{\id,I}$ satisfies the two properties. According to Corollary \ref{ap1eq} we have $\O_I P_{\id,I}=[W_I(\Gm,\zz,0)]$.
Property \Ref{a2uno} follows from either topology (the variety $\Om_{\id, I}$ is smooth at the point $x_I$) or from the $h=0$ substitution in Lemma~\ref{lem:princ}. We are going to prove property \Ref{a2dos}\:. Let $J\not=I$. We proved in Lemma~\ref{lem:W_small} that
\[
W_I(\zz_J,\zz,h) \prec_J W_J(\zz_J,\zz,h)=
\prod_{k<l}\prod_{a\in I_k}(
\prod_{\satop{b\in I_l}{b>a}} (1-z_b/z_a) \prod_{\satop{b\in I_l}{b<a}} (1-hz_b/z_a)
).
\]
We also saw in Lemma \ref{lem:div_e_ver} that $W_I(\zz_J,\zz,h)$ is divisible by
$\prod_{k<l}\prod_{a\in I_k}\prod_{\satop{b\in I_l}{b<a}} (1-hz_b/z_a)$. Hence for the quotient we have
\[
W_I(\zz_J,\zz,h)/\prod_{k<l}\prod_{a\in I_k}\prod_{\satop{b\in I_l}{b<a}} (1-hz_b/z_a)
\prec_J
\prod_{k<l}\prod_{a\in I_k}\prod_{\satop{b\in I_l}{b>a}} (1-z_b/z_a).
\]
Observe that the right hand side does not depend on $h$, so the same $\prec_J$ inequality also holds after plugging in $h=0$ in the left hand side. The $h=0$ substitution of the left hand side is $[W_{I}(\Gm,\zz,0)]|_{x_J}$, hence property \Ref{a2dos} is proved.
\end{proof}

\section{Appendix 3. Presentations and structure constants of algebras associated with the projective line}
\label{app 2}
In Section~\ref{sec conjectures} we described the conjectured presentation of the equivariant quantum K-theory algebra of partial flag manifolds. In this section we describe in detail this algebra for the special case of the projective line $\P^1$. The basic algebra $K(\P^1)$ can be decorated in three independent ways, namely by considering the equivariant version (with $z=(z_1,z_2)$ parameters), the
quantum version (with $q=(q_1,q_2)$ parameters), and the cotangent bundle version (with $h$ parameter). We describe all eight possible algebras, shown in the diagram below and obtained by considering or not considering any of the three decorations.
In all eight cases we give the presentation, as well as the structure constants in terms of a natural choice of bases.
The descriptions of the four algebras on the front face of the cube are known, and the descriptions of the four algebras on the back face are conjectural.

In the diagram, the symbol $\clubsuit$ means the limit $q_1=Q_1h^{-1}, q_2=Q_2, h\to 0$.
\[
\xymatrix{
& QK_{\C^\times}(T^*\P^1)\ar[ld]|{q_2/q_1\to 0}\ar[dd]|>>>>>>>{\clubsuit} & & QK_T(T^*\P^1) \ar[ll]|{z=1}\ar[ld]|{q_2/q_1\to 0} \ar[dd]|{\clubsuit} \\
K_{\C^\times}(T^*\P^1)\ar[dd]|{h=0} & & K_T(T^*\P^1) \ar[ll]|<<<<<<<<<<<<{z=1}\ar[dd]|>>>>>>>{h=0}& \\
& QK(\P^1)\ar[ld]|{Q_2/Q_1\to 0} & & QK_{(\C^\times)^2}(\P^1) \ar[ll]|>>>>>>>>>>>>{z=1}\ar[ld]|{Q_2/Q_1\to 0} \\
K(\P^1) & & K_{(\C^\times)^2}(\P^1) \ar[ll]|{z=1}&
}
\]
\begin{itemize}
\item{} For $K(\P^1)$ we have the presentation: $\C[\gm^{\pm1},\delta^{\pm1}]/\langle \on{relations} \rangle$, where the ideal of relations is generated by the
coefficients of $u$-powers in
\[
(1-\gm/u)(1-\delta/u)-(1-1/u)^2.
\]
With the choice of basis $\kr_0=1$, $\kr_1=1-1/\gm$, the multiplication is
\[
\kr_0\kr_0= \kr_0,\qquad
\kr_0 \kr_1 =\kr_1, \qquad
\kr_1 \kr_1=0.
\]
\item For $K_{(\C^\times)^2}(\P^1)$ we have the presentation $\C[\gm^{\pm1},\delta^{\pm1},z_1^{\pm1},z_2^{\pm1}]/\langle \on{relations} \rangle$, where the ideal of relations is generated by
the coefficients of $u$-powers in
\[
(1-\gm/u)(1-\delta/u)-(1-z_1/u)(1-z_2/u).
\]
With the choice of basis $\kr_0=1$, $\kr_1=1-z_2/\gm$, the multiplication is
\[
\kr_0\kr_0= \kr_0,\qquad
\kr_0 \kr_1 =\kr_1, \qquad
\kr_1 \kr_1=\left(1-\frac{z_2}{z_1}\right)\kr_1.
\]
\item{} For $QK(\P^1)$ we have the presentation: $\C[\gm^{\pm1},\delta^{\pm1},Q_1^{\pm1},Q_2^{\pm1}]/\langle \on{relations} \rangle$, where the ideal of relations is generated by
the coefficients of $u$-powers in
\[
\det\begin{pmatrix}
1-\gm/u & -\gm/u \\
Q_2/Q_1 & 1-\delta/u
\end{pmatrix}
-(1-1/u)^2.
\]
With the choice of basis $\kr_0=1$, $\kr_1=1-1/\gm$ the multiplication is
\[
\kr_0\kr_0= \kr_0,\qquad
\kr_0 \kr_1 =\kr_1, \qquad
\kr_1 \kr_1=\frac{Q_2}{Q_1} \kr_0.
\]
\item{} For $QK_{(\C^\times)^2}(\P^1)$ we have the presentation $\C[\gm^{\pm1},\delta^{\pm1},z_1^{\pm1},z_2^{\pm1},Q_1^{\pm1},Q_2^{\pm1}]/\langle \on{relations} \rangle$, where the ideal of relations is generated by
	the coefficients of $u$-powers in
\[
\det\begin{pmatrix}
1-\gm/u & -\gm/u \\
Q_2/Q_1 & 1-\delta/u
\end{pmatrix}
-(1-z_1/u)(1-z_2/u).
\]
With the choice of basis $\kr_0=1$, $\kr_1=1-z_2/\gm$ the multiplication is
\[
\kr_0\kr_0= \kr_0,\qquad
\kr_0 \kr_1 =\kr_1, \qquad
\kr_1 \kr_1=\frac{Q_2z_2}{Q_1z_1} \kr_0+\left(1-\frac{z_2}{z_1}\right)\kr_1.
\]
\item{} For $K_{\C^\times}(T^*\P^1)$ we have the presentation $\C[\gm^{\pm1},\delta^{\pm1},h^{\pm1}]/\langle \on{relations} \rangle$, where the ideal of relations is generated by
the coefficients of $u$-powers in
\[
(1-\gm/u)(1-\delta/u)-(1-1/u)^2.
\]
With the choice of basis $\kr_0=1-h/\gm$, $\kr_1=1-1/\gm$ the multiplication is
\[
\kr_0\kr_0= (1-h)\kr_0 + h(1-h)\kr_1,\qquad
\kr_0 \kr_1 =(1-h)\kr_1, \qquad
\kr_1 \kr_1=0.
\]
\item{} For $K_T(T^*\P^1)$ we have the presentation $\C[\gm^{\pm1},\delta^{\pm1},z_1^{\pm1}, z_2^{\pm1}, h^{\pm1}]/\langle \on{relations} \rangle$, where the ideal of relations is generated by
the coefficients of $u$-powers in
\[
(1-\gm/u)(1-\delta/u)-(1-z_1/u)(1-z_2/u).
\]
With the choice of basis $\kr_0=1-hz_1/\gm$, $\kr_1=1-z_2/\gm$ the multiplication is
\begin{align*}
\kr_0\kr_0= & (1-hz_1/z_2)\kr_0+h(1-h)z_1/z_2\kr_1,\\
\kr_0 \kr_1 = & (1-h)\kr_1, \\
\kr_1 \kr_1= &(1-z_2/z_1)\kr_1.
\end{align*}
\item{} For $QK_{\C^\times}(T^*\P^1)$ we have the presentation $\C[\gm^{\pm1},\delta^{\pm1}, q_1^{\pm1},q_2^{\pm1},h^{\pm1}]/\langle \on{relations} \rangle$, where the ideal of relations is generated by
	the coefficients of $u$-powers in
\[
{\det \begin{pmatrix}
1-\gm/u & 1/q_1(1-\gm/(uh)) \\
q_2(1-h\delta/u) & (1-\delta/u)
\end{pmatrix}-(1-q_2/q_1)(1-1/u)^2
}
\]
With the choice of basis $\kr_0=1-h/\gm$, $\kr_1=1-1/\gm$ the multiplication is (define $q=q_2/q_1$)
\begin{align*}
\kr_0\kr_0= & \frac{1-h}{1-qh}\kr_0 + \frac{(1-h)h}{1-qh}\kr_1
\\
\kr_0 \kr_1 = &\frac{q(1-h)}{1-qh}\kr_0+\frac{1-h}{1-qh}\kr_1
\\
\kr_1 \kr_1= &\frac{(1-h)q}{h(1-qh)} \kr_0+\frac{q(1-h)}{(1-qh)}\kr_1.
\end{align*}
\item{} For $QK_T(T^*\P^1)$ we have the presentation $\C[\gm^{\pm1},\delta^{\pm1},z_1^{\pm1},z_2^{\pm1},q_1^{\pm1},q_2^{\pm1},h^{\pm1}]/\langle \on{relations} \rangle$, where the ideal of relations is generated by
	the coefficients of $u$-powers in
\[
{\det \begin{pmatrix}
1-\gm/u & 1/q_1(1-\gm/(uh)) \\
q_2(1-h\delta/u) & (1-\delta/u)
\end{pmatrix}-(1-q_2/q_1)(1-z_1/u)(1-z_2/u)
}
\]
With the choice of basis $\kr_0=1-hz_1/\gm$, $\kr_1=1-z_2/\gm$ the multiplication is (define $q=q_2/q_1$)
\begin{align*}
\kr_0\kr_0= & \frac{1-qh(1-z_1/z_2)-z_1h/z_2}{1-qh}\kr_0 + \frac{(1-h)hz_1/z_2}{1-qh}\kr_1
\\
\kr_0 \kr_1 = &\frac{q(1-h)}{1-qh}\kr_0+\frac{1-h}{1-qh}\kr_1
\\
\kr_1 \kr_1= &\frac{(1-h)qz_2/z_1}{h(1-qh)} \kr_0+\frac{1-z_2/z_1+q(z_2/z_1-h)}{1-qh}\kr_1.
\end{align*}
\end{itemize}

\section{Appendix 4. Equivariant K-theoretic Schubert calculus on the cotangent bundle of partial flag manifolds}

In this section we present a result on the structure constants of the algebra $K_T(T^*\Fla)\ox\C(\zz,h)$
with respect to the basis $\{\ka_I:=\ka_{\id,I}\}_{I\in\Il}$.
In view of Theorem
\ref{thm Gr h=0} of Appendix 1, this result is a natural ``$h$-deformation'' of the multiplication rules for double Grothendieck polynomials.

Let $N,n,\Il, \si_0$ be as before.

\begin{thm} The following multiplication rules hold in $K_T(T^*\Fla)\ox\C(\zz,h)$. For $A,B\in\Il$, we have
$\ka_A \ka_B = \sum_{J\in\Il} c_{A,B}^J \ka_J$, where
\[
c_{A,B}^J=\sum_{I\in\Il}
\frac{\Wt_A(\zz_I,\zz,h) \Wt_B(\zz_I,\zz,h) \Wt_{\si_0,J}(\zz_I^{-1},\zz^{-1},h^{-1})}
{R(\zz_I)Q(\zz_I)}h^{p(J)}P(\zz_I).
\]
\end{thm}

\begin{proof}
Consider the $z_I$-substitution into $\ka_A \ka_B = \sum_{J\in\Il} c_{A,B}^J \ka_J$. We obtain \[\Wt_A(\zz_I,\zz,h)\Wt_B(\zz_I,\zz,h)=\sum_J c_{A,B}^J \Wt_J(\zz_I,\zz,h).\]
Choose a $K\in\Il$, and multiply both sides by
\[
\frac{\Wt_{\si_0,K}(\zz_I^{-1},\zz^{-1},h^{-1})}
{R(\zz_I)Q(\zz_I)}h^{p(K)}P(\zz_I).
\]
We obtain
\bea
&&
\frac{\Wt_A(\zz_I,\zz,h)\Wt_B(\zz_I,\zz,h)\Wt_{\si_0,K}(\zz_I^{-1},\zz^{-1},h^{-1})}
{R(\zz_I)Q(\zz_I)}h^{p(K)}P(\zz_I)=\hspace{1in}
\\
&&
\hspace{1in}
= \sum_J c_{A,B}^J
\frac{\Wt_J(\zz_I,\zz,h)\Wt_{\si_0,K}(\zz_I^{-1},\zz^{-1},h^{-1})}
{R(\zz_I)Q(\zz_I)}h^{p(K)}P(\zz_I).
\eea
Add this equation for all $I\?\in\Il$, then rearrangement gives
\bea
&&
\sum_I
\frac{\Wt_A(\zz_I,\zz,h)\Wt_B(\zz_I,\zz,h)\Wt_{\si_0,K}(\zz_I^{-1},\zz^{-1},h^{-1})}
{R(\zz_I)Q(\zz_I)}h^{p(K)}P(\zz_I)
= \hspace{1in}
\\
&&
\hspace{1in}
= \sum_J c_{A,B}^J \sum_I
\frac{\Wt_J(\zz_I,\zz,h)\Wt_{\si_0,K}(\zz_I^{-1},\zz^{-1},h^{-1})}
{R(\zz_I)Q(\zz_I)}h^{p(K)}P(\zz_I).
\eea
According to the orthogonality Theorem \ref{thm orth} the last sum $\sum_I(...)$ is $\delta_{J,K}$.
Hence the right-hand side is equal to $c_{A,B}^K$. This proves the theorem.
\end{proof}

\begin{example}
For $\lambda=(1,1)$ one recovers the multiplication table for $K_T(T^*\P^1)$ in the basis $\kr_1=\ka_{\{1\},\{2\}}$ and $\kr_0=\ka_{\{2\},\{1\}}$ of Section~\ref{app 2}. Here are some sample products for $\lambda=(2,2)$:
\begin{eqnarray*}
\ka_{\{1,2\},\{3,4\}}\ka_{\{1,4\},\{2,3\}} &=&
(1-h)(1-z_3/z_1)(1-z_4/z_1)(1-hz_3/z_2)\ka_{\{1,2\},\{3,4\}}, \hspace{.5in} \\
\ka_{\{1,3\},\{2,4\}}\ka_{\{1,4\},\{2,3\}} &=&
\end{eqnarray*}
\begin{eqnarray*}
&& \left( (1-h)^2(1-z_4/z_2)(z_2/z_1+h(1-z_2/z_1-z_3/z_1+z_2/z_3-z_2^2/(z_1z_3)))\right)
\ka_{\{1,2\},\{3,4\}} +\\
&& (1-h)(1-z_2/z_1)(1-z_4/z_1)(1-hz_2/z_3)
\ka_{\{1,3\},\{2,4\}}.
\end{eqnarray*}
The $h=0$ substitution in these formulae gives the multiplication of the equivariant K-theory classes of structure sheaves of Schubert varieties (up to the polarization) in the Grassmannian of 2-planes in $\C^4$.
\end{example}

\section{Appendix 5. Bethe algebra of the \XXZ/ model}
\label{app 5}

Let \,$\Upn$ \,be the subalgebra of \,$\Uen$ \,generated over \,$\Chh$
\vvn.1>
\,by elements~\Ref{Lijs}, \Ref{Lii}. Given a complex number \,$c\ne 0,1$\>,
set \,$\Uenc=\Upn/\bra\:h=c\:\ket$. Let \,$\Bck_c$ \,be the subalgebra of
\,$\Uenc$ generated by the images of elements \Ref{Ueh} and the elements
\,$\Bk_{p,\:\pm s}$\>, see~\Ref{Beq}\:, under the canonical projection
\,$\Upn\to\Uenc$\>.

\vsk.2>
By Proposition \ref{pho}, there is a \,$\Chh$-algebra homomorphism
\vvn.1>
\be
\pho\::\Upn\,\to\,\End\bigl(\CNn\bigr)\ox\Czh\>.
\ee
Given \,$\bb\in(\C^\times)^n$ \:and \,$c\ne 0,1$\>,
the evaluation at \,$\zz=\bb$\>, \,$h=c$ \,induces an algebra homomorphism
\vvn-.4>
\be
\pho_{\:\bb,\:c}:\Uenc\to\:\End\bigl(\CNn\bigr)
\ee
The homomorphism
\,$\pho_{\:\bb,\:c}$ \,makes \,$\CNn$ into a \,$\Uenc$-module denoted
\,$\CNn(\bb,c)$ \,and called the tensor product of vector representations
with evaluation parameters \,$\bb$\>.
The algebra \,$\Bck_{\bb,\:c}=\pho_{\:\bb,\:c}(\Bck_c)$ \,is called
the Bethe algebra of the associated \XXZ/ model on \,$\CNn$\>.

\begin{rem}
Usually the \XXZ/ model on \,$\CNn$ \>is defined by considering \,$\CNn$
as a module over the quantum loop algebra \,$\Uqn$\>.
% see Section~\ref{Quantum loop algebra Uqn}.
Under this definition, the Bethe subalgebra of the \XXZ/ model,
coincides with \,$\Bck_{\bb,\:q^{-2}}$.
\end{rem}

Recall the space \,$\DV\?$, see Section~\ref{secDV}. Let \,$\bb=(b_1\lc b_n)$
and \,$c$ \,be such that \,$b_i\ne c\:b_j$ \>for all \,$i,j=1\lc n$\>.
The evaluation at \,$\zz=\bb$\>, \,$h=c$ \,defines a homomorphism of
\,$\Uenc$-modules
\vvn-.1>
\be
\iota_{\:\bb,\:c}:\DV/\bra\:h=c\:\ket\,\to\,\CNn(\bb,c)\,.
\vv.2>
\ee

\begin{prop}
\label{DVsur}
Let \;$b_i\ne c\:b_j$ \>for all \,$i,j=1\lc n$\>.
Then \;$\iota_{\:\bb,\:c}$ \>is an epimorphism.
\end{prop}
\begin{proof}
The proof is similar to the proof of \cite[Proposition~10.4]{MTV3}\:.
The image of \,$\iota_{\:\bb,\:c}$ \,is not zero, while by~\cite{AK}\:,
the \,$\Uenc$-module \,$\CNn(\bb,c)$ is irreducible if \;$b_i\ne c\:b_j$
\>for all \,$i,j=1\lc n$\>, see also~\cite[Theorem~3.4]{NT2}\:.
\end{proof}

The elements of \,$\Bck_{\bb,\:c}$ \:preserve each of the subspaces
\,$\CNnl$\>. Let \,$\Bck_\bcl\?\subset\End\bigl(\CNnl\bigr)$ \,be
the subalgebra induced by the action of \,$\Bck_{\bb,\:c}$ on \,$\CNnl$.

\vsk.2>
Given \,$c$ and \,$\bla$, let \,$\Kcqb\?=\Kcql/\bra\:\zz=\bb\,, \;h=c\:\ket$\>.
\vvn.08>
The algebra \,$\Kcqb$ \,is the algebra of functions on the fiber of the Wronski
map.

\begin{cor}
\label{XXZ thm}
Let \,$\bb=(b_1\lc b_n)$ and \,$c$ \:be such that \,$b_i\ne c\:b_j$
for all \,$i,j=1\lc n$\>. Then isomorphisms\/~\Ref{muql}\:, \Ref{nuql}\:,
and the evaluation at \,$\zz=\bb$, \,$h=c$ \,induce an algebra isomorphism
\;$\muq_{\:\bcl}\<:\Kcqb\?\to\Bck_\bcl$ \,and an isomorphism of vector spaces
\;$\nuq_\bcl\<:\Kcqb\?\to\CNnl$\>. The isomorphisms \,$\muq_{\:\bcl}$ and
\,$\nuq_{\bcl}$ identify the regular representation of \,$\Kcqb$ and
the \,$\Bck_\bcl\<$-module \;$\CNnl$\>.
\qed
\end{cor}

\begin{cor}
\label{maxcomm}
The algebra \,$\Bck_{\bb,\:c}$ is a maximal commutative subalgebra of\/
\,$\End\bigl(\CNn\bigr)$\>.
\qed
\end{cor}

%!!!

\end{document}